\documentclass[a4paper,10pt,reqno,draft]{amsart}

\usepackage{amsthm, amssymb, amsmath, latexsym}
\usepackage{amsfonts, mathrsfs, color}
\usepackage{a4wide}
\usepackage{listings}
\usepackage{color}
\usepackage[usenames,dvipsnames]{xcolor}
\usepackage{rotating}
\usepackage{hhline}
\usepackage{subfigure}
\usepackage{multirow}
\usepackage{enumerate}
\usepackage[bookmarks]{}
\usepackage{amsfonts}   
\usepackage{dsfont}
\usepackage{tikz}
\usepackage{accents}

\usepackage{marginnote}

 \theoremstyle{plain}
 \begingroup
 \newtheorem{thm}{Theorem}[section]
 \newtheorem{lem}[thm]{Lemma}
 \newtheorem{prop}[thm]{Proposition}
 \newtheorem{cor}[thm]{Corollary}
 \endgroup

 \theoremstyle{definition}
 \begingroup
 \newtheorem{defn}[thm]{Definition}
 
 \endgroup

 \theoremstyle{remark}
 \begingroup
 \newtheorem{rem}[thm]{Remark}
 \endgroup

 \numberwithin{equation}{section}

\mathchardef\emptyset="001F

\newcommand{\hs}{{\mathcal H}}

\newcommand{\ds}{\displaystyle}

\newcommand{\G}{{\mathcal G}}
\newcommand{\dx}{\,dx}
\newcommand{\dy}{\,dy}

\newcommand{\B}{\mathscr{B}}
\newcommand{\F}{\mathcal{F}}
\newcommand{\T}{\mathcal{T}}
\newcommand{\mI}{\mathcal{I}}

\newcommand{\R}{{\mathbb R}}
\newcommand{\Z}{{\mathbb Z}}
\newcommand{\Q}{{\mathbb Q}}
\newcommand{\N}{{\mathbb N}}
\newcommand{\Sph}{{\mathbb S}}

\newcommand{\Om}{\Omega}

\newcommand{\e}{\varepsilon}

\newcommand{\om}{\omega}
\newcommand{\ie}{{; \it i.e., }}

\newcommand{\A}{\mathscr{A}}

\newcommand{\LM}[2]{\hbox{\vrule width.4pt \vbox to#1pt{\vfill
\hrule width#2pt height.4pt}}}
\newcommand{\LLL}{{\mathchoice {\>\LM{7}{5}\>}{\>\LM{7}{5}\>}{\,\LM{5}{3.5}\,}{\,\LM{3.35}
{2.5}\,}}}

\newcommand{\EEE}{\color{black}}

\title[]
{
A global method for deterministic and\\ stochastic homogenisation in $BV$ \EEE
}

\author[]{FILIPPO CAGNETTI}
\address[]{Department of Mathematics,
University of Sussex, Brighton, United Kingdom}
\email[]{F.Cagnetti@sussex.ac.uk}
\author[]{GIANNI DAL MASO}
\address[]{SISSA, Trieste, Italy}
\email[]{dalmaso@sissa.it}
\author[]{LUCIA SCARDIA}
\address[]{Department of Mathematics, Heriot-Watt University, United Kingdom}
\email[]{L.Scardia@hw.ac.uk}
\author[]{CATERINA IDA ZEPPIERI}
\address[]{Angewandte Mathematik,
WWU M\"unster, Germany}
\email[]{caterina.zeppieri@uni-muenster.de}
\begin{document}


\begin{abstract}

In this paper we study the deterministic and stochastic homogenisation of free-discontinuity functionals under \emph{linear} growth and coercivity conditions.  
The main novelty of our deterministic result is that we work under very general assumptions on the integrands which, in particular, are not required to be periodic in the space variable. Combining this result with the pointwise Subadditive Ergodic Theorem by Akcoglu and Krengel, we prove a stochastic homogenisation result, in the case of  stationary random integrands. In particular, we characterise the limit integrands in terms of asymptotic cell formulas, as in the classical case of periodic homogenisation.


\end{abstract}

\maketitle

{\small

\bigskip

\section{Introduction}

In this paper we derive deterministic and stochastic homogenisation results for \emph{free-discontinuity functionals}  in the space of functions of bounded variation. 

We consider families of free-discontinuity functionals of the form 
\begin{equation}\label{Intro:funct-e}
E_\e(\omega)(u,A)= \int_A f(\omega,\tfrac{x}{\e},\nabla u)\dx+\int_{S_u\cap A}g(\omega,\tfrac{x}{\e},[u],\nu_u)\,d\mathcal H^{n-1},
\end{equation}
where $\omega$ belongs to the sample space of a given probability space $(\Om, \T, P)$ and labels the realisations of the random integrands $f$ and $g$, while
$\e>0$ is a parameter of either geometrical or physical nature, and sets the scale of the problem. In \eqref{Intro:funct-e} the set $A$ belongs to the class $\A$ of 
bounded open subsets of $\R^n$ and the function $u$ belongs to the space $SBV(A, \R^m)$ of special
 $\R^m$-valued  functions of bounded variation  on  $A$ (see \cite{DGA}  and \cite[Section 4.5]{AFP}). Moreover, $\nabla u$ denotes the approximate  gradient  of $u$, 
$[u]$ stands for the difference $u^+-u^-$ between the approximate limits of $u$ on both sides of the discontinuity set $S_u$,  
$\nu_u$ denotes the (generalised) normal to $S_u$, and $\mathcal{H}^{n-1}$ is the $(n-1)$-dimensional Hausdorff measure in $\R^n$.

Functionals as in \eqref{Intro:funct-e} are commonly used in applications in which the physical quantity described by $u$ can exhibit discontinuities, e.g. in variational models in fracture mechanics, in the theory of computer vision and image segmentation, and in  problems involving phase transformations.
 
We are interested in determining the \emph{almost sure} limit behaviour of $E_\e$ as $\e \to 0^+$, when $f$ and $g$ satisfy \textit{linear} growth and coercivity conditions in the gradient and in the jump, respectively. The linear growth of the volume energy sets the limit problem naturally in the space $BV$ of functions with bounded variation. Indeed, 
in this setting, limits of  sequences of displacements with bounded energy can develop a Cantor component in the distributional gradient. 

This is in contrast with our previous work \cite{CDMSZ} (in the deterministic case) and \cite{CDMSZ-stoc} (in the stochastic case), where, under the assumption of superlinear growth for $f$, the limit problem was naturally set in the space $SBV$ of special functions with bounded variation.

The two main results of this paper are a deterministic homogenisation result for functionals of the type \eqref{Intro:funct-e}, when $\om$ is fixed and $f$ and $g$ are not necessarily periodic, and a stochastic homogenisation result, obtained for $P$-a.e.\ $\om \in \Om$, under the additional assumption of \emph{stationarity} of $f$ and $g$.

\smallskip

\subsection{The deterministic result} For the deterministic result we consider $\om$ as fixed in \eqref{Intro:funct-e} and write $E_\e(u,A)$ instead of $E_\e(\om)(u,A)$. 
We study the limit behaviour of the functionals $E_\e(\cdot, A)$, for every $A\in \A$, as $\e\to 0^+$, under the assumption that the energy densities $f$ and $g$ belong to suitable classes $\mathcal{F}$ and $\mathcal G$ of admissible volume and surface densities (see Definition \ref{def:FG}). As announced above, a key requirement for the class $\mathcal{F}$ is that $f$ satisfies \textit{linear} upper and lower bounds  in the gradient variable. Additionally, we require that the recession function $f^\infty$ of $f$ is defined at every point. For the class $\mathcal{G}$, we require that $g$ is bounded from below and from above by the amplitude of the jump, and that its directional derivative $g_0$ in the jump variable at $[u]=0$ exists and is finite. The functions $f^\infty$ and $g_0$ will play an important role in determining the limiting densities.

We stress here that we \textit{do not} require any periodicity in the spatial variable $x$ for the volume and surface densities; moreover, we \textit{do not} require any continuity in the spatial variable either, since it would be unnatural for applications. 

Under these general assumptions, using  the so-called localisation method of $\Gamma$-convergence \cite{DM93}, we can prove that there exists a subsequence $(\e_k)$ such that, 
for every $A\in \A$,  $(E_{\e_k}(\cdot,A))$ $\Gamma$-converges to an abstract functional $\widehat E(\cdot,A)$, that $\widehat E(\cdot,A)$ is finite only in $BV$, and that, for every $u\in BV_{\rm loc}(\R^n,\R^m)$, the set function $\widehat E(u,\cdot)$ is the restriction to $\A$ of a Borel measure (see Theorem \ref{thm:Gamma-conv}). 

Note that, without any additional assumptions, one cannot expect that $\widehat{E}(\cdot,A)$ can be written in an integral form. In particular,  since there is no guarantee that $z\mapsto \widehat E(u(\cdot-z),A+z)$ is continuous (which is instead automatically satisfied in the periodic case), we cannot directly apply the integral representation result in $BV$ \cite{BFM}. 
Our integral representation result is hence obtained under some additional assumptions, which are though more general than periodicity. We require that the limits of some rescaled minimisation problems, defined in terms of $f$, $g$, $f^\infty$, and $g_0$, exist and are independent of the spatial variable. These limits will then define the densities of $\widehat E$. 

More precisely, for $A\in \mathscr A$, $w\in SBV(A,\R^m)$, $f\in \mathcal{F}$, and $g\in \mathcal{G}$ we set
\begin{equation}\label{minminmin}
m^{f,g}(w,A):=\inf\Big\{\int_A  f (x,\nabla u)\dx+\int_{S_u\cap A}\!\!\!\!\!\! g (x,[u],\nu_{u})\,d\mathcal H^{n-1} : 
u\in SBV(A,\R^m),u=w \textrm{ near } \partial A\Big\},
\end{equation}
and 
we \textit{assume} that for every $\xi \in \R^{m\times n}$ the limit 
\begin{equation}\label{f-hom-000}
\lim_{r\to +\infty} \frac{
m^{f,g_0}(\ell_\xi,Q_r(rx))
}{ r^{n}}=:f_{\rm hom}(\xi)
\end{equation}
exists and is independent of $x\in \R^n$, and that for every $\zeta \in \R^{m}$ and $\nu\in \Sph^{n-1}$ the limit  
\begin{equation}\label{g-hom-000}
\lim_{r\to +\infty} \frac{m^{f^\infty,g}(u_{r x,\zeta,\nu},Q^\nu_r(rx))}{r^{n-1}}=:g_{\rm hom}(\zeta,\nu)
\end{equation} 
exists and is independent of $x\in \R^n$, where In \eqref{f-hom-000}, $\ell_\xi$ denotes the linear function with gradient $\xi$; 
in \eqref{g-hom-000}, $Q^\nu_r(rx):= 
R_\nu \big((-\tfrac{r}{2},\tfrac{r}{2})^{n}\big) + rx$, where $R_\nu$ is an orthogonal $n{\times} n$ matrix such that $R_\nu e_n=\nu$, and 
\begin{equation*}
u_{rx,\zeta,\nu}(y):=\begin{cases}
\zeta \quad & \mbox{if} \,\, (y-rx)\cdot \nu \geq 0,\\
0 \quad & \mbox{if} \,\, (y-rx)\cdot \nu < 0.
\end{cases}
\end{equation*}

The limits \eqref{f-hom-000} and \eqref{g-hom-000} are the counterpart of the asymptotic cell-formulas in the classical periodic homogenisation  \cite{BFM}. In that case, periodicity in the spatial variable $x$ ensures the existence of the limits and  their homogeneity in $x$, that here we have to 
postulate. Our assumptions are however weaker than periodicity: notably, they are fulfilled in the case of \textit{stationary} integrands, as we show in the present work. 

In line with the result in \cite{BFM}, the functional being minimised in \eqref{f-hom-000} (respectively in \eqref{g-hom-000}) has densities $f$ and $g_0$ (respectively $f^\infty$ and $g$) rather than $f$ and $g$. We note that, if the density $f$ satisfied a superlinear growth, then $f^\infty(\cdot,\xi)=+\infty$ for $\xi\neq0$. Since one always has $f^\infty(\cdot,0)=0$, in the superlinear case there would be the following changes in  \eqref{g-hom-000}: on the one hand the minimisation would be done over functions with $\nabla u=0$; on the other hand the functional to be minimised would reduce to just the surface term. This is indeed the situation in \cite{CDMSZ}, where we assume a superlinear growth for $f$. Moreover, in \cite{CDMSZ} we work under different growth conditions for $g$ as well,  which in particular satisfies $g\geq c$, for a given fixed positive constant. In that case $g_0\equiv +\infty$ (see \eqref{ach}), and hence formally the minimisation in \eqref{f-hom-000} would be over Sobolev functions, and the functional to be minimised would reduce to just the bulk term. Again, this is exactly what happens in \cite{CDMSZ} (see also \cite{BDfV, GP}).

 In Lemmas \ref{l:f-hom} and \ref{l:g-hom} we show that $f_{\rm hom}\in \mathcal F$  and  $g_{\rm hom}\in \G$; the fact that $f_{\rm hom}\in \mathcal F$ guarantees in particular the existence of its recession function $f^\infty_{\rm hom}$. In Propositions \ref{p:homo-vol}, \ref{p:homo-sur}, and \ref{p:homo-Can}, we show that  the functions $f_{\rm hom}$, $g_{\rm hom}$ and $f^\infty_{\rm hom}$ are the densities of the volume, surface, and Cantor terms of $\widehat E$, respectively. To do so 
we use the blow-up technique of Fonseca and M\"uller \cite{FMue} (see also \cite{BMSig}), extended to the $BV$-setting by Bouchitt\'e, Fonseca, and Mascarenhas \cite{BFM}. More precisely,  thanks to \eqref{f-hom-000} and \eqref{g-hom-000}, we prove that the following identities hold true for every $A \in \A$ and for every $u\in L^1_{\rm loc}(\R^n,\R^m)$ with $u|_A\in BV(A,\R^m)$: 

\begin{equation*} 
\frac{d\widehat E(u,\cdot)}{d \mathcal L^n}(x)=f_{\rm hom}(\nabla u(x))\qquad\hbox{for }\mathcal L^n\hbox{-a.e.\ }x\in A, 
\end{equation*}
\begin{equation*} 
\frac{d \widehat E(u,\cdot)}{d\mathcal H^{n-1}\LLL{S_u}}(x) = g_{\rm hom}([u](x),\nu_{u}(x))\qquad \hbox{for }\mathcal H^{n-1}\hbox{-a.e.\ } x\in S_u \cap A,
\end{equation*} 
and 
\begin{equation*} 
\frac{d\widehat E(u,\cdot)}{d|C(u)|}(x)= f_{\rm hom}^\infty\Big(\frac{dC(u)}{d|C(u)|}(x)\Big)\qquad\hbox{for }|C(u)|\hbox{-a.e.\ }x\in A,
\end{equation*}
where $\mathcal{L}^n$ is the Lebesgue measure in $\R^n$, $\mathcal H^{n-1}\LLL{S_u}$ is the measure defined by $(\mathcal H^{n-1}\LLL{S_u})(B):=\mathcal H^{n-1}(S_u\cap B)$ for every Borel set $B\subset\R^n$, and $C(u)$ is the Cantor part of the distributional derivative of~$u$.

In particular, since the right-hand sides of the previous equalities do not depend on $(\e_k)$, the homogenisation result holds true for the whole sequence $(E_\e)$. Moreover, the fact that $f_{\rm hom}\in \mathcal F$  and  $g_{\rm hom}\in \G$ implies that the classes $\F$ and $\G$ are closed under homogenisation, and that on $SBV$ the functionals $E_\e$ and their $\Gamma$-limit $\widehat E$ are free-discontinuity functionals of the same type. 

As observed before,  $f_{\rm hom}$ and $g_{\rm hom}$ depend on both the volume and the surface densities of the functional $E_\e$. Indeed, the minimisation problems in \eqref{f-hom-000} and \eqref{g-hom-000} involve both $f$ (or $f^\infty$) and $g$ (or $g_0$). In other words, volume and surface term \textit{do interact} in the limit, which is a typical feature of the linear-growth setting. This is in contrast with the case of  superlinear  growth considered in \cite{CDMSZ}, in which the limit volume density only depends on the volume density of $E_\e$, and similarly the  limit surface density only depends on the surface density of $E_\e$. The volume-surface decoupling is typical of the $SBV$-setting in presence of superlinear growth conditions on $f$ \cite{BDfV, CDMSZ, GP}. 
 
Note however that, even in the superlinear case, if $f$ and $g$ satisfy ``degenerate'' coercivity conditions, due for instance to the presence of perforations or ``weak'' inclusions in the  domain, the situation is more involved. Indeed, while in \cite{BF, BS, CS, FGP, PSZ1} the volume and surface terms do not interact in the homogenised limit, in \cite{BDM, BLZ, DMZ, PSZ, LS1, LS2} they do interact and produce rather complex limit effects. 

\subsection{The stochastic result} In Section \ref{section:stoc-hom} we prove  the almost sure $\Gamma$-convergence of the sequence of random functionals $E_\e(\om)$ in \eqref{Intro:funct-e} to a random homogenised integral functional, under the assumption that the volume and surface integrands $f$ and $g$ are stationary (see Definition \ref{def:stationary}  and Remark~\ref{rem:contg}). 
 In the random setting stationarity is the natural counterpart of periodicity, since it implies that $f$ and $g$ are ``statistically'' translation-invariant, or ``periodic in law''. 

The application of the deterministic result Theorem \ref{T:det-hom},  at $\om$ fixed, ensures that $E_\e(\om)$ $\Gamma$-converges to the free-discontinuity functional 
\begin{equation*}
E_{\mathrm{hom}}(\om)(u, A):= \int_A f_{\mathrm{hom}}(\om,\nabla u)\dx + \int_{S_u\cap A}g_{\mathrm{hom}}(\om,[u],\nu_u)\,d \mathcal{H}^{n-1}+\int_A f^\infty_{\rm hom}\Big(\om,\frac{dC(u)}{d|C(u)|}\Big)\,d|C(u)|,
\end{equation*}
with
\begin{equation}\label{f-hom-om}
f_{\rm hom}(\om,\xi):=\lim_{r\to +\infty} \frac{
m^{f(\om),g_0(\om)}(\ell_\xi,Q_r(rx))
}{r^{n}},
\end{equation}
for every $\xi \in \R^{m\times n}$, and 
\begin{equation}\label{g-hom-om}
g_{\rm hom}(\om,\zeta,\nu):=\lim_{r\to +\infty} \frac{m^{f^\infty(\om),g(\om)}(u_{r x,\zeta,\nu},Q^\nu_r(rx))}{r^{n-1}},
\end{equation} 
for every $\zeta \in \R^{m}$ and $\nu\in \Sph^{n-1}$, provided the limits in \eqref{f-hom-om} and \eqref{g-hom-om} (which are the same as \eqref{f-hom-000} and  \eqref{g-hom-000}, modulo the additional dependence on the random parameter $\om$) exist and are independent of $x\in \R^n$.
Therefore, to show that the $\Gamma$-convergence of $E_\e(\om)$ towards $E_{\rm hom}(\om)$ actually holds true for $P$-a.e.\ $\om \in \Om$
it is necessary to show that the limits in \eqref{f-hom-om} and \eqref{g-hom-om} exist and are independent of $x\in \R^n$ for $P$-almost every realisation $\om \in \Om$. 
To do so, we follow the general strategy firstly introduced in \cite{DMM2} in the Sobolev setting, and then extended to the $SBV$-setting in \cite{CDMSZ-stoc} (see also \cite{Alciru}). 

This strategy relies on the Subadditive Ergodic Theorem by Akcoglu and Krengel \cite{AK} (see Theorem~\ref{ergodic}) and requires, among other things,  to show that the minimisation problems in \eqref{f-hom-om} and \eqref{g-hom-om} define two \textit{subadditive stochastic processes} (see Definition \ref{Def:subadditive}). 

This task however poses a challenge even at the very first step: proving that $\om\mapsto m^{f(\om),g_0(\om)}$ and  
 $\om\mapsto m^{f^\infty(\om),g(\om)}$ are \textit{measurable}.  Indeed, while both $m^{f(\om),g_0(\om)}$ and $m^{f^\infty(\om),g(\om)}$, by \eqref{minminmin},  involve the minimisation of measurable functionals in the random variable $\om$, such minimisation is performed over the space $SBV$, which is not separable. Since  the infimum in \eqref{minminmin} cannot be reduced to a countable set, the measurability of $\om\mapsto m^{f(\om),g_0(\om)}$ and $\om\mapsto m^{f^\infty(\om),g(\om)}$ cannot be inferred directly from the measurability in $\om$ of $f$, $f^\infty$, $g$ and $g_0$ (see the Appendix).
Let us also observe that the situation here is substantially different from  that  treated in \cite{CDMSZ, CDMSZ-stoc}, for a number of reasons. Indeed in \cite{CDMSZ, CDMSZ-stoc}, as observed before, due to the different assumptions on $f$ and $g$, the $\Gamma$-limit exhibits a ``separation'' of the volume and surface term. In particular, the limit volume density is obtained as the limit of some minimisation problems similar to \eqref{minminmin}, but where the minimisation is done over the space of Sobolev functions. Hence in that case the limit volume density is $\om$-measurable, due to the separability of the space (see also \cite{DMM2}). On the other hand, the measurability of the minimisation problems defining the limit surface density was delicate also in \cite{CDMSZ-stoc}, since the minimum was taken over Caccioppoli partitions. However, in \cite{CDMSZ-stoc} the minimisation involved only the surface term of the functional, which makes the proof much simpler than the one required now. 

Once the measurability in $\om$ is established (see Proposition \ref{measurability}), we have to face  yet  another difficulty: determining the \textit{dimension} of the stochastic processes. Indeed, using the competitors $\ell_\xi$ and $u_{rx,\zeta,\nu}$ in the minimisation problems $m^{f(\om),g_0(\om)}$ and $m^{f^\infty(\om),g(\om)}$, respectively, suggests the rescalings in \eqref{f-hom-om} and \eqref{g-hom-om}. Hence it suggests that $m^{f(\om),g_0(\om)}$ should define an $n$-dimensional process, while $m^{f^\infty(\om),g(\om)}$ should define an $(n-1)$-dimensional process. On the other hand, in both cases the functionals appearing in the minimisation problems, if seen as set functions, are defined on $n$-dimensional sets. 

To solve this problem we proceed as in \cite{Alciru, BP} and \cite{CDMSZ-stoc}, where similar issues arise in the study of pure surface energies of spin systems, and in the case of free-discontinuity problems with superlinear growth, respectively. 

Finally, the last difficulty consists in showing that, as in \cite{Alciru, CDMSZ-stoc}, the limits in  \eqref{f-hom-om} and \eqref{g-hom-om} do not depend on $x$. This is particularly delicate for \eqref{g-hom-om}, due to the presence of an $x$-dependent boundary condition. 

We conclude by observing that our analysis also shows that, if $f$ and $g$ are ergodic, then the homogenised integrands $f_{\rm hom}$ and $g_{\rm hom}$ are $\om$-independent, and hence the limit $E_{\rm hom}$ is deterministic.  

\subsection{Outline of the paper.} This paper is organised as follows. In Section \ref{Notation} we introduce some notation. Section \ref{S:setting} consists of two parts: in Section \ref{Sub:setting} we introduce the stochastic free-discontinuity functionals and recall the Ergodic Subadditive Theorem; in Section \ref{sub:main_results} we state the main results of the paper.

Sections \ref{determ-hom}-\ref{Sect:Cantor} focus on the deterministic results. More in detail, in Section \ref{determ-hom} we state the deterministic $\Gamma$-convergence results and prove some properties of the limit densities; Section \ref{Sect:abstract} is devoted to the abstract $\Gamma$-convergence result; the volume, surface and Cantor terms of the abstract $\Gamma$-limit are then identified in Sections \ref{Sect:volume}, \ref{Sect:surface} and \ref{Sect:Cantor}, respectively.

Finally, Section \ref{section:stoc-hom} focuses on the stochastic homogenisation result, while the proof of the measurability of $\om\mapsto m^{f(\om),g_0(\om)}$ and $\om\mapsto m^{f^\infty(\om),g(\om)}$ is postponed to the Appendix.

\section{Notation}\label{Notation}

We introduce now some notation that will be used throughout the paper. 

\begin{itemize}
\item[(a)] $m$ and $n$ are fixed positive integers, with $n\ge 2$, $\R$ is the set of real numbers, 
while $\Q$ is the set of rational numbers.   The canonical basis of $\R^n$  is denoted by $e_1,\dots, e_n$. 
For $a, b \in \R^n$, $a \cdot b$ denotes the Euclidean scalar product between $a$ and $b$,
and $|\cdot|$ denotes the absolute value in $\R$ or the Euclidean norm in $\R^n$, $\R^m$, 
or $\R^{m{\times}n}$ (the space of $m \times n$ matrices with real entries), depending on the context. 
If $v \in \R^m$ and $w \in \R^n$, the symbol 
$v \otimes w$ stands for the matrix in $\R^{m \times n}$ 
whose entries are $(v \otimes w)_{ij} = v_i w_j$, for $i = 1, \ldots, m$
and $j = 1, \ldots, n$.

\item[(b)] $\Sph^{m-1}:=\{\zeta=(\zeta_1,\dots,\zeta_m)\in \mathbb{R}^{m}: \zeta_1^2+\dots+\zeta_m^2=1\}$, $\Sph^{n-1}:=\{x=(x_1,\dots,x_n)\in \mathbb{R}^{n}: x_1^2+\dots+x_n^2=1\}$, and $\widehat\Sph^{n-1}_\pm:=\{x\in \Sph^{n-1}: \pm x_{i(x)}>0\}$,
where $i(x)$ is the largest $i\in\{1,\dots,n\}$ such that $x_i\neq 0$. Note that $\Sph^{n-1} = \widehat\Sph^{n-1}_+ \cup \widehat\Sph^{n-1}_-$.
\item[(c)] $\mathcal{L}^n$ denotes the Lebesgue measure on $\R^n$ and $\mathcal{H}^{n-1}$ the $(n-1)$-dimensional Hausdorff measure on $\R^n$. 
\item[(d)] $\A$ denotes the collection of all bounded open subsets of $\R^n$; if $A$, $B\in \A$, by
$A\subset\subset B$ we mean that $A$ is relatively compact in $B$.
\smallskip
\item[(e)] For $u\in BV(A,\R^m)$, with  $A \in \A$, the jump of $u$ across the jump set $S_u$ is defined by $[u]:=u^+-u^-$, while $\nu_u$ denotes the (generalised) normal to $S_u$  (see \cite[Definition 3.67]{AFP}).
\item[(f)] 
For every $u\in BV(A,\R^m)$, with $A \in \A$, the distributional gradient, denoted by $Du$, is an $\R^{m{\times}n}$-valued Radon measure on $A$, whose absolutely continuous part with respect to $\mathcal{L}^n$, denoted by $D^au$, has a density $\nabla u\in L^1(A,\R^{m{\times}n})$ (which coincides with the approximate gradient of $u$), while the singular part $D^su$ can be decomposed as $D^su = D^ju + C(u)$, where the jump part $D^ju$ is given by 
$$D^ju (B)=\int_{B\cap S_u}[u]\otimes\nu_u d\hs^{n-1}
\quad \hbox{for every Borel set }B\subset A,$$
and the Cantor part $C(u)$ is an $\R^{m{\times}n}$-valued Radon measure on $A$ which vanishes on all Borel sets $B\subset A$ with $\hs^{n-1}(B)<+\infty$.
\item[(g)] For $x\in \mathbb{R}^n$ and $\rho>0$ we define 
\begin{eqnarray*}
& B_\rho(x):=\{y\in \R^n: \,|y-x|<\rho\},
\\
& Q_\rho(x):= \{y\in \R^n: \,|(y-x)\cdot e_i | <\tfrac\rho2 \quad\hbox{for } i=1,\dots,n \}.
\end{eqnarray*}
\item[(h)] For every $\nu\in \Sph^{n-1}$ let $R_\nu$ be an orthogonal $n{\times} n$ matrix such that $R_\nu e_n=\nu$; we assume that the restrictions of the function $\nu\mapsto R_\nu$ to the sets $\widehat\Sph^{n-1}_\pm$ defined in (b) are continuous and that $R_{-\nu}Q(0)=R_\nu Q(0)$ for every $\nu\in \Sph^{n-1}$;  moreover, we assume that 
$R_\nu \in O(n) \cap \mathbb{Q}^{n \times n}$ 
for every $\nu \in \mathbb{Q}^n \cap \Sph^{n-1}$.  
A map $\nu\mapsto R_\nu$ satisfying these properties is provided in \cite[Example A.1  and Remark A.2]{CDMSZ}.

\item[(i)] For $x\in \mathbb{R}^n$, $\rho>0$, and $\nu\in \Sph^{n-1}$ we set
$$
Q^\nu_\rho(x):= 
R_\nu Q_\rho(0) + x.
$$
For $k\in \R$, with $k>0$, we also define the  rectangle  
$$
Q^{\nu,k}_\rho(x):=Q^{\nu,k}_\rho(0)+x
$$ 
where $Q^{\nu,k}_\rho(0)$ is obtained from $Q^\nu_\rho(0)$ by a dilation of amplitude $k$ in the directions orthogonal to $\nu$\ie
$$
Q^{\nu,k}_\rho(0):=R_\nu \big((-\tfrac{k\rho}{2},\tfrac{k\rho}{2})^{n-1}\times (-\tfrac{\rho}{2},\tfrac{\rho}{2})\big). 
$$
We set
\begin{gather*}
\partial^\perp_\nu Q^{\nu,k}_\rho(x):= \partial Q^{\nu,k}_\rho(x)\cap R_\nu \big((-\tfrac{k\rho}{2},\tfrac{k\rho}{2})^{n-1}\times \R\big),
\\
\partial^\parallel_\nu Q^{\nu,k}_\rho(x):= \partial Q^{\nu,k}_\rho(x)\cap R_\nu \big(\R^{n-1}\times (-\tfrac{\rho}{2},\tfrac{\rho}{2})\big),
\end{gather*}
namely the union of the faces of $Q^{\nu,k}_\rho(x)$ 
that are orthogonal and parallel to $\nu$, respectively. 

\item[(j)] Let $\mu$ and $\lambda$ be two Radon measures on $A\in\A$, with values in a finite dimensional Hilbert space $X$ and in $[0,+\infty]$, respectively; for every $x\in A$ the Radon-Nikodym derivative  of  $\mu$ with respect to $\lambda$ is defined as 
$$
\frac{d\mu}{d\lambda}(x):=\lim_{r\to 0+}\frac{\mu(x+rC)}{\lambda(x+rC)}
$$
whenever the limit exists in $X$ and is independent of the choice of the bounded, closed set $C$ containing the origin in its interior (see \cite[Definition 1.156]{FonLeo}); according to the Besicovich differentiation theorem $\frac{d\mu}{d\lambda}(x)$ exists for $\lambda$-a.e.\ $x\in A$ and $\mu=\frac{d\mu}{d\lambda}\lambda +\mu^s$, where $\mu^s$ is the singular part of $\mu$ with respect to $\lambda$ (see \cite[Theorem 1.155]{FonLeo}). 
\item[(k)] For $\xi\in \R^{m{\times}n}$, the linear function from $\R^n$ to $\R^m$ with gradient $\xi$ is denoted by $\ell_\xi$\ie $\ell_\xi(x):=\xi x$, where $x$ is considered as an $n{\times}1$ matrix.
\item[(l)] For $x\in \mathbb{R}^n$, $\zeta \in \R^m$, and $\nu \in \Sph^{n-1}$ we define the function $u_{x,\zeta,\nu}$ as
\begin{equation*}
u_{x,\zeta,\nu}(y):=\begin{cases}
\zeta \quad & \mbox{if} \,\, (y-x)\cdot \nu \geq 0,\\
0 \quad & \mbox{if} \,\, (y-x)\cdot \nu < 0.
\end{cases}
\end{equation*}
\item[(m)] For  $x\in \mathbb{R}^n$ and $\nu \in \Sph^{n-1}$, we set
$$
\Pi^{\nu}_0:= \{y\in \R^n: y\cdot \nu = 0\}\quad\text{and}\quad\Pi^{\nu}_{x}:= \{y\in \R^n: (y-x)\cdot \nu = 0\}.
$$ 
\item[(n)] For a given topological space $X$, $\B(X)$ denotes its Borel $\sigma$-algebra. 
 For  every integer $ k \geq 1$, $\B^k$ is
the Borel $\sigma$-algebra of $\R^k$, while $\B^n_S$ 
stands for the Borel $\sigma$-algebra of $\mathbb{S}^{n-1}$. 

\end{itemize}

\section{Setting of the problem and statements of the main results}\label{S:setting}
This section consists of two parts: in Section \ref{Sub:setting} we introduce the stochastic free-discontinuity functionals and recall the Ergodic Subadditive Theorem; in Section \ref{sub:main_results} we state the main results of the paper.

\subsection{Setting of the problem}\label{Sub:setting}
\noindent Throughout the paper we fix the following constants: $c_1$,  $c_2$, $c_3$, $c_4$, $c_5\in [0,+\infty)$, with $0<c_2\leq c_3$,  and $\alpha\in (0,1)$. 
Moreover, we fix two nondecreasing continuous functions $\sigma_1$, $\sigma_2\colon  [0,+\infty) \to [0,+\infty)$ such that $\sigma_1(0)=\sigma_2(0)=0$.

\begin{defn}[Volume and surface  integrands]\label{def:FG} Let $\mathcal{F}=\mathcal{F}( c_1,c_2,c_3,c_4,c_5, \alpha, \sigma_1)$ be the collection of all functions $f\colon \R^n{\times} \R^{m{\times}n}\to [0,+\infty)$ satisfying the following conditions:
\begin{itemize}
\item[$(f1)$] (measurability) $f$ is Borel measurable on $\R^n{\times} \R^{m{\times}n}$;
\item[$(f2)$] (continuity in $\xi$) for every $x \in \R^n$ we have
\begin{equation*}
|f(x,\xi_1)-f(x,\xi_2)| \leq   \sigma_1(|\xi_1-\xi_2|)\big(f(x,\xi_1)+f(x,\xi_2)\big) + c_1|\xi_1-\xi_2| 
\end{equation*}
for every $\xi_1$, $\xi_2 \in \R^{m{\times}n}$;
\item[$(f3)$] (lower bound) for every $x \in \R^n$ and every $\xi \in \R^{m{\times}n}$
$$
c_2 |\xi| \leq f(x,\xi);
$$
\item[$(f4)$] (upper bound) for every $x \in \R^n$ and every $\xi \in \R^{m{\times}n}$
$$
f(x,\xi) \leq c_3|\xi|+c_4;
$$
\item[$(f5)$]  (recession function) for every $x \in \R^n$ and every $\xi \in \R^{m{\times}n}$ the limit
\begin{equation}\label{abe}
f^\infty(x,\xi):=\lim_{t\to +\infty}  \tfrac{1}{t}f(x,t\xi),
\end{equation}
which defines the recession function of $f$, exists and is finite; moreover,
  $f^\infty$ satisfies the inequality 
\begin{equation}\label{abf}
|f^\infty(x,\xi)- \tfrac{1}{t}f(x,t\xi) |\leq \tfrac{c_5}{t}+\tfrac{c_5}{t}f(x,t\xi)^{1-\alpha}
\end{equation}
for every $x \in \R^n$, every $\xi \in \R^{m{\times}n}$, and every $t>0$.
\end{itemize}

Let $\mathcal{G}=\mathcal{G}(c_2,c_3,\sigma_2)$ be the collection of all functions 
$g\colon \R^n{\times}\R^m{\times} \Sph^{n-1} \to [0,+\infty)$ satisfying the following conditions:
\begin{itemize}
\item[$(g1)$] (measurability) $g$ is Borel measurable on $\R^n{\times}\R^m{\times} \Sph^{n-1}$;
\item[$(g2)$] (continuity in $\zeta$) for every $x\in \R^n$ and every $\nu \in \Sph^{n-1}$ we have
\begin{equation*}
|g(x,\zeta_2,\nu)-g(x,\zeta_1,\nu)|\leq \sigma_2(|\zeta_1-\zeta_2|)\big(g(x,\zeta_1,\nu)+g(x,\zeta_2,\nu)\big)
\end{equation*}
for every $\zeta_1$, $\zeta_2\in \R^m$;
\item[$(g3)$] (lower bound) for every $x\in \R^n$, $\zeta\in \R^m$, and $\nu \in \Sph^{n-1}$
\begin{equation*}
c_2 |\zeta| \leq g(x,\zeta,\nu);
\end{equation*}
\item[$(g4)$] (upper bound) for every $x\in \R^n$, $\zeta\in \R^m$, and $\nu \in \Sph^{n-1}$
\begin{equation*}
g(x,\zeta,\nu) \leq c_3 |\zeta|;
\end{equation*}
\item [$(g5)$]   (directional derivative at $0$)  for every $x\in \R^n$, $\zeta\in \R^m$, and $\nu \in \Sph^{n-1}$ the limit
\begin{equation}\label{ach} 
g_0(x,\zeta,\nu):=\lim_{t\to 0+} \tfrac{1}{t} g(x,t\,\zeta,\nu)
\end{equation} 
exists, is finite, and is uniform with respect to $x\in \R^n$, $\zeta\in \Sph^{m-1}$, and $\nu \in \Sph^{n-1}$.
\item[$(g6)$] (symmetry) for every $x\in \R^n$, $\zeta\in \R^m$, and $\nu \in \Sph^{n-1}$
$$g(x,\zeta,\nu) = g(x,-\zeta,-\nu).$$ 
\end{itemize}
\end{defn}

For every $f \in \mathcal F$, $ g \in \G$, $A\in\A$, and $u\in SBV(A,\R^m)$ we set 
\begin{equation*} 
E^{ f,g}(u,A):= \int_A  f (x,\nabla u)\dx+\int_{S_u\cap A} g (x,[u],\nu_{u})\,d\mathcal H^{n-1} ,
\end{equation*}
 and for every 
 $w\in SBV(A,\R^m)$ we set
\begin{equation}\label{m-phipsi}
m^{ f,g}(w,A):=  \inf\{E^{ f,g}(u,A): u\in SBV(A,\R^m),\ u=w \textrm{ near } \partial A\}.
\end{equation}
The  expression ``$u=w \textrm{ near }\partial A$''
 means  that there exists a neighbourhood $U$ of $\partial A$ such that
$u=w$ $\mathcal{L}^n$-a.e.\ in $U\cap A$. 
More in general, if $\Lambda \subset \partial A$ is a relatively open subset of $\partial A$, the  expression ``$u=w \textrm{ near }\Lambda$'' means  that there exists a neighbourhood $U$ of $\Lambda$ in $\R^n$ such that $u=w$ $\mathcal{L}^n$-a.e.\ in $U \cap A$.

For technical reasons, related to the details of the statement of the Subadditive Ergodic Theorem, it is convenient to extend this definition to  an arbitrary bounded subset $A$  of $\R^n$,  by setting   
$m^{f,g}(w,A):=m^{f,g}(w,\mathrm{int}A)$, where $\mathrm{int}A$ denotes the interior of $A$.

\begin{rem}\label{BV estimate}
If $f\in \mathcal F$ and $g\in \mathcal G$, then $f^\infty\in \mathcal F$ and $g_0\in \mathcal G$. Moreover, the lower bounds $(f3)$ and $(g3)$ imply that
\begin{equation*}
 c_2 |Du|(A)\le \min\{E^{f,g}(u,A), E^{f,g_0}(u,A),E^{f^\infty,g}(u,A), E^{f^\infty,g_0}(u,A) \}
\end{equation*}
for every $A\in \A$ and $u \in L^1_{\rm loc}(\R^n,\R^m)$, with  $u|_A\in SBV(A,\R^m)$.
\end{rem}

\begin{rem} 
From $(f4)$ and $(f5)$ it follows that for every $L>0$ there exists $M>0$, depending on $c_3$, $c_4$, $c_5$, and $L$, such that
\begin{equation}\label{abb}
|f^\infty(x,\xi)-\tfrac{1}{t} f(x,t\xi)|\leq \tfrac{M}{t^\alpha}\quad\hbox{for every }x \in \R^n,\  \xi \in \R^{m{\times}n}\hbox{  with }|\xi|=1,\hbox{ and }t\ge L.
\end{equation} 
Conversely, if the limit in \eqref{abe} exists and  $f$ satisfies  $(f3)$, $(f4)$, and \eqref{abb} for some $L>0$  and $M>0$, then for every $x \in \R^n$,  $\xi \in \R^{m{\times}n}$,  and $t>0$ we have
$$
\begin{cases}
\medskip
|f^\infty(x,\xi)-\tfrac{1}{t} f(x,t\xi)|\leq \tfrac{M}{t^\alpha}|\xi|^{1-\alpha}= \tfrac{M}{t}|t\xi|^{1-\alpha}\le  
\tfrac{M}{t c_2^{1-\alpha}}f(x,t\xi)^{1-\alpha} &\quad\hbox{if }t|\xi|\ge L,
\\
|f^\infty(x,\xi)-\tfrac{1}{t} f(x,t\xi)|\leq\tfrac{2c_3L}{t}+\tfrac{c_4}{t}&\quad\hbox{if }t|\xi|< L,
\end{cases}
$$
where in the last inequality we used the fact that
that $f^\infty(x,\xi) \leq c_3 | \xi|$ for every $x\in\R^n$ and $\xi \in \R^{m \times n}$.
This implies that $f$ satisfies \eqref{abf} for a suitable constant $c_5$, depending only on $c_2, c_3$, $c_4$, $L$, and $M$.

\end{rem}

\begin{rem}\label{abc} If $(f4)$ holds, then $(f5)$ is equivalent to the fact that
\begin{equation}\label{abd}
| \tfrac{1}{s}f(x,s\xi) - \tfrac{1}{t}f(x,t\xi) |\leq  \tfrac{c_5}{s}+\tfrac{c_5}{s}f(x,s\xi)^{1-\alpha} + \tfrac{c_5}{t}+\tfrac{c_5}{t}f(x,t\xi)^{1-\alpha}
\end{equation} 
for every $x \in \R^n$, every $\xi \in \R^{m{\times}n}$, and every $s,\,t>0$. Indeed, using the triangle inequality we obtain
\eqref{abd} from \eqref{abf} for $s$ and $t$. Conversely, if $(f4)$ holds, then 
$\tfrac{c_5}{t}+\tfrac{c_5}{t}f(x,t\xi)^{1-\alpha}\to 0$ as $t\to+\infty$. Therefore \eqref{abd} implies that 
$t\mapsto \tfrac{1}{t}f(x,t\xi)$ satisfies the Cauchy condition as $t\to +\infty$, hence the limit in \eqref{abe} exists, while \eqref{abf} follows from \eqref{abd} by taking the limit as $s\to+\infty$.
\end{rem}

\begin{rem} 
Assume that $g\colon \R^n{\times}\R^m{\times} \Sph^{n-1} \to [0,+\infty)$ satisfies $(g5)$ and let $\lambda\colon [0,+\infty)\to [0,+\infty)$ be defined by
\begin{equation}\label{acb}
\lambda(t):=\sup\{|g_0(x,\zeta,\nu)- \tfrac{1}{\tau} g(x, \tau\,\zeta,\nu)|: x\in \R^n,\ \zeta\in \Sph^{m-1},\ \nu \in \Sph^{n-1}, \tau\in (0,t ] \}.
\end{equation}
Then $\lambda$ is nondecreasing and 
\begin{eqnarray}
&\ds\lim_{t\to0+}\lambda(t)=0,
\label{acc}
\\
&\ds|g_0(x,\zeta,\nu)-\tfrac{1}{t} g(x,t\,\zeta,\nu)|\le |\zeta|\lambda(t|\zeta|),
\label{acd}
\end{eqnarray}
for every $x\in \R^n$, $\zeta\in \R^m$, and $\nu \in \Sph^{n-1}$. If $g$ satisfies also $(g3)$, then
\eqref{acd} gives
\begin{equation}\label{ace}
|g_0(x,\zeta,\nu)-\tfrac{1}{t} g(x,t\,\zeta,\nu)|\le \tfrac{1}{c_2}\lambda(t|\zeta|)\tfrac{1}{t} g(x,t\,\zeta,\nu).
\end{equation}
Conversely, if the limit in \eqref{ach} exists (even with no uniformity assumptions) and $g$ satisfies $(g4)$ and \eqref{ace}, then it satisfies \eqref{acd} with $\lambda$ replaced by $\frac{c_3}{c_2}\lambda$, which implies that $g$ satisfies $(g5)$.
\end{rem}

\begin{rem} 
If $g\colon \R^n{\times}\R^m{\times} \Sph^{n-1} \to [0,+\infty)$ satisfies $(g3)$, $(g4)$, and $(g5)$, then by \eqref{ace} and by the triangle inequality we get
\begin{equation}\label{acg}
|\tfrac{1}{s} g(x,s\,\zeta,\nu)-\tfrac{1}{t} g(x,t\,\zeta,\nu)|\le \tfrac{1}{c_2}\lambda(s|\zeta|)\tfrac{1}{s} g(x,s\,\zeta,\nu)+\tfrac{1}{c_2}\lambda(t|\zeta|)\tfrac{1}{t} g(x,t\,\zeta,\nu)
\end{equation}
for every $s, t>0$, $x\in \R^n$, $\zeta\in \R^m$, and $\nu \in \Sph^{n-1}$. Conversely, if $(g4)$ and \eqref{acg} hold, with some function $\lambda$ satisfying \eqref{acc}, then 
$\lambda(t|\zeta|)\tfrac{1}{t} g(x,t\,\zeta,\nu)\to0$ as $t\to0+$ and hence, using \eqref{acg}, we deduce that the function $t\mapsto \tfrac{1}{t} g(x,t\,\zeta,\nu)$ satisfies the Cauchy condition as $t\to0+$. This implies  that the limit in \eqref{ach} exists and is finite. Moreover, passing to the limit as $s\to0+$ from \eqref{acg}   we obtain \eqref{ace}, which, in turn, yields $(g5)$.
\end{rem}

\bigskip

We are now ready to introduce the probabilistic setting of our problem.
In what follows $(\Om,\T,P)$ denotes a fixed probability space.

\begin{defn}[Random integrands]\label{ri}
A function $f \colon \Om \times \R^n \times \R^{m \times n} \to [0,+\infty)$ is called a \textit{random volume integrand} if
\begin{itemize}
\item[(a1)]  $f$ is $\T\otimes \B^n\otimes \B^{m\times n}$-measurable; 
\item[(b1)]  $f(\om,\cdot,\cdot)\in \mathcal{F}$ for every $\om\in \Om$. 
\end{itemize}
A function $g \colon \Om \times \R^n\times \R^m \times \Sph^{n-1}  \to [0,+\infty)$ is called a \textit{random surface integrand} if
\begin{itemize}
\item[(a2)]  $g$ is $\T\otimes \B^n\otimes \B^m\otimes \B^{n}_S$-measurable; 
\item[(b2)] $g(\om,\cdot,\cdot,\cdot)\in \mathcal{G}$ for every $\om\in \Om$.
\end{itemize}
\end{defn}

Let $f$ be a random volume integrand and let $g$ be a random surface integrand. For every $\om \in \Om$ and every $\e>0$ we consider the free-discontinuity functional $E_\e(\om)\colon L^{1}_{\rm loc}(\R^n,\R^m)\times \A \longrightarrow [0,+\infty]$ defined by
\begin{equation}\label{Eeps}
E_\e(\om)(u,A) :=\begin{cases}
\ds\int_A\!\! f(\om,\tfrac{x}{\e}, \nabla u) \dx + \!\! \int_{S_u\cap A} \!\!\!\!\!  g(\om,\tfrac{x}\e,[u],\nu_u)\,d \mathcal{H}^{n-1}& \!\!\text{if} \; u|_A\!\in SBV(A,\R^m),\cr
+\infty & \!\! \text{otherwise\! in}\,L^{1}_{\rm loc}(\R^n,\R^m).
\end{cases}
\end{equation}

\begin{defn}
Il $f$ is a random volume integrand, we define 
$f^\infty \colon \Om \times \R^n \times \R^{m \times n} \to [0,+\infty)$ by
\begin{equation}\label{DM-abe}
f^\infty(\om,x,\xi):=\lim_{t\to +\infty}  \tfrac{1}{t}f(\om,x,t\xi).
\end{equation}
If $g$ is a random surface integrand, we define
$g_0 \colon \Om \times \R^n\times \R^m \times \Sph^{n-1}  \to [0,+\infty)$ by
\begin{equation}\label{DM-ach} 
g_0(\om,x,\zeta,\nu):=\lim_{t\to 0+} \tfrac{1}{t} g(\om,x,t\,\zeta,\nu).
\end{equation} 
\end{defn}

\begin{rem}\label{DM-rem}
The existence of the limit in \eqref{DM-abe} follows from (b1) in Definition~\ref{ri} and from 
$(f5)$. Since for every $t>0$ the functions $(\om,x,\xi)\mapsto\tfrac{1}{t}f(\om,x,t\xi)$ 
are $\T\otimes \B^n\otimes \B^{m\times n}$-measurable by (a1), the same property holds for $f^\infty$.
Moreover, from Remark~\ref{BV estimate} and from (b1) we deduce that $f^\infty(\om,\cdot,\cdot)\in \mathcal{F}$ 
for every $\om\in \Om$. We conclude that $f^\infty$ is a random volume integrand.

The existence of the limit in \eqref{DM-ach} follows from (b2) in Definition~\ref{ri} and from 
$(g5)$. Since for every $t>0$ the functions $(\om,x,\zeta,\nu)\mapsto\tfrac{1}{t} g(\om,x,t\,\zeta,\nu)$ 
are $\T\otimes \B^n\otimes \B^m\otimes \B^{n}_S$-measurable by (a2), the same property holds for $g_0$.
Moreover, from Remark~\ref{BV estimate} and from (b2) we deduce that $g_0(\om,\cdot,\cdot,\cdot)\in \mathcal{G}$ 
for every $\om\in \Om$. We conclude that $g_0$ is a random surface integrand.
\end{rem}

\noindent In the study of stochastic homogenisation an important role is played by the notions introduced by the following definitions.

\begin{defn}[$P$-preserving transformation]
A $P$-preserving transformation on $(\Om,\T,P)$ is a map $T \colon \Om \to \Om$ satisfying the following properties:
\begin{itemize}
\item[(a)] (measurability) $T$ is $\T$-measurable;
\item[(b)] (bijectivity) $T$ is bijective;
\item[(c)]  (invariance) $P(T(E))=P(E)$, for every $E\in \T$.
\end{itemize}
\noindent If, in addition, every set $E\in \T$ which satisfies 
$T(E)=E$ (called $T$-invariant set) has probability $0$ or $1$, then $T$ is called {\it ergodic}.
\end{defn}

\begin{defn}[Group of $P$-preserving transformations]\label{P-preserving} 
Let $d$ be a positive integer.
A group of $P$-preserving transformations on $(\Om,\T,P)$ is a family $(\tau_z)_{z\in \Z^d}$ of mappings $\tau_z \colon \Om \to \Om$ satisfying the following properties:
\begin{itemize}
\item[(a)] (measurability) $\tau_z$ is $\T$-measurable for every $z\in \Z^d$;
\item[(b)] (bijectivity) $\tau_z$ is bijective for every $z\in \Z^d$; 
\item[(c)]  (invariance) $P(\tau_z(E))=P(E)$,\, for every $E\in \T$ and every $z\in \Z^d$;
\item[(d)] (group property) $\tau_{0}={\mathrm id}_\Om$ (the identity map on $\Om$) and $\tau_{z+z'}=\tau_{z} \circ \tau_{z'}$ for every $z, z'\in \Z^d$;
\end{itemize}
\noindent If, in addition, every set $E\in \T$ which satisfies 
$\tau_z (E)=E$ for every $z\in \Z^d$
has probability $0$ or $1$, then $(\tau_z)_{z\in\Z^d}$ is called {\it ergodic}.
\end{defn}

We are now in a position to define the notion of stationary random integrand.
\begin{defn}[Stationary random integrand]\label{def:stationary}
A random volume integrand $f$ is \textit{stationary} with respect to a group $(\tau_z)_{z\in \Z^n}$ of $P$-preserving transformations on $(\Om,\T,P)$ if
$$
f(\om,x+z,\xi)=f(\tau_z(\om),x,\xi)
$$
for every $\om\in \Om$, $x\in \R^n$, $z\in \Z^n$, and $\xi\in \R^{m \times n}$.

Similarly, a random surface integrand $g$ is \textit{stationary} with respect to $(\tau_z)_{z\in \Z^{n}}$ if 
\begin{equation*} 
g(\om,x+z,\zeta,\nu)=g(\tau_z(\om),x,\zeta,\nu)
\end{equation*}
for every $\om\in \Om$, $x\in \R^n$, $z\in \Z^{n}$, $\zeta\in \R^m$, and $\nu \in \Sph^{n-1}$.
\end{defn} 

\medskip

We now recall the notion of subadditive stochastic  process  as well as the Subadditive Ergodic Theorem by Akcoglu and Krengel \cite[Theorem 2.7]{AK}.  

\medskip

Let $d$ be a positive integer. For every $a,b\in \R^d$, with $a_i< b_i$ for $i=1,\dots,d$, we define
$$
[a, b):= \{x\in \R^{ d}: a_i\leq x_i<b_i\text{ for } i=1,\dots,d\},
$$
and we set 
\begin{equation}\label{int:int}
\mathcal{I}_d:= \{[a, b): a,b\in \R^{ d}, a_i< b_i\text{ for }i=1,\dots,d\}.
\end{equation}

\begin{defn}[Subadditive process]\label{Def:subadditive} 
A  $d$-dimensional {\em subadditive process} with respect to a group $(\tau_z)_{z\in \Z^{ d}}$, $d\ge 1$, of $P$-preserving transformations on $(\Om,\T,P)$ is  a function $\mu\colon\Om\times \mI_d\to \R$ satisfying the following properties: 
\begin{itemize}
\item[(a)] (measurability) for every $A\in \mI_d$ the function $\om\mapsto \mu(\om,A)$ is $\T$-measurable;
\item[(b)] (covariance) for every $\om\in\Om$, $A\in \mI_d$, and $z\in \Z^d$ we have
$
\mu(\om, A+z)=\mu(\tau_z(\om),A)$;
\item[(c)] (subadditivity) for every $A\in \mI_d$ and for every \emph{finite} family $(A_i)_{i\in I} \subset \mI_d$ of pairwise disjoint sets such that $A= \cup_{i\in I} A_i$, we have
$$
\mu(\om,A)\leq \sum_{i\in I} \mu(\om,A_i)\quad\hbox{for every } \om\in\Om; 
$$
 \item[(d)] (boundedness) there exists $c>0$ such that $0\leq \mu(\om, A) \leq c \mathcal L^d(A)$ for every $\om\in\Om$ and every $A\in \mI_d$. 
\end{itemize}
\end{defn}

\begin{defn}[Regular family of sets] \label{reg-fam}
A family of sets $(A_t)_{t>0}$ in $\mI_d$ is called \emph{regular} (with constant $C>0$) if there exists another family of sets $(A'_t)_{t>0}\subset \mI_d$ such that:
\begin{itemize}
\item[(a)] $A_t \subset A'_t \; \text{for every} \; t>0$;
\item[(b)] $A'_{ s} \subset A'_{ t} \; \text{whenever}\  0<s<t$; 
\item[(c)] $0< \mathcal L^d(A'_t) \leq C \mathcal L^d(A_t) \; \text{for every} \; t>0$.
\end{itemize} 
If the family $(A'_t)_{t>0}$ can be chosen in a way such that $\R^d=\bigcup_{t>0} A'_t$, then we write $\ds\lim_{t\to +\infty}A_t=\R^d$. 
\end{defn}

We now state a variant of the pointwise ergodic Theorem \cite[Theorem 2.7 and Remark p. 59]{AK} which is suitable for our purposes. This variant can be found in \cite[Theorem 4.1]{LiMi}.  

\begin{thm}[Subadditive Ergodic Theorem]\label{ergodic} Let $d\in \{n-1,n\}$ and let $(\tau_z)_{z\in \Z^d}$ be a group of $P$-preserving transformations on $(\Om,\T,P)$. Let $\mu \colon \Om\times \mI_d\to \R$ be a subadditive process  with respect to $(\tau_z)_{z\in \Z^d}$. 
Then there exist a $\T$-measurable function $\varphi \colon \Om \to [0,+\infty)$ and a set $\Om'\in\T$ with $P(\Om')$=1 such that
\begin{equation*}
\lim_{t \to +\infty} \frac{\mu(\om,A_t)}{\mathcal L^d(A_t)}=\varphi(\om), 
\end{equation*}
for every regular family of sets $(A_t)_{t>0} \subset \mI_d$ with $\ds\lim_{t\to +\infty}A_t=\R^d$ and for every $\om\in \Om'$. If in addition $(\tau_z)_{z\in \Z^d}$ is ergodic, then $\varphi$ is constant $P$-a.e. 
\end{thm}

\begin{rem}[Covariance with respect to a continuous group $(\tau_z)_{z\in \R^d}$]\label{rem:contg}
Definitions~\ref{P-preserving}, \ref{def:stationary}, \ref{Def:subadditive} and Theorem~\ref{ergodic} can be adapted also to the case of a continuous group $(\tau_z)_{z\in \R^d}$, see for instance \cite[Section~3.1]{CDMSZ-stoc}.
\end{rem}

\subsection{Statement of the main results}\label{sub:main_results}
In this section we state the main result of the paper, Theorem \ref{G-convE}, which provides a  $\Gamma$-convergence result  for the random functionals $(E_\e(\om))_{\e > 0}$ introduced in \eqref{Eeps}, under the assumption that the volume and surface integrands $f$ and $g$ are stationary. 

The next theorem proves the existence of the  limits in the asymptotic  cell formulas that will be used in the statement of the main result. 

\medskip

When $f$ and $g$ are random integrands it is convenient to introduce the following shorthand notation
\begin{equation} \label{inf for fixed omega}
m^{f,g_0}_\omega:=m^{f(\omega,\cdot,\cdot),g_0(\omega, \cdot, \cdot,\cdot)}\,, \qquad m_\om^{f^\infty,g}:=m^{f^\infty(\omega,\cdot,\cdot),g(\omega, \cdot, \cdot,\cdot)}\,, \qquad
m^{f^\infty,g_0}_\omega:=m^{f^\infty(\omega,\cdot,\cdot),g_0(\omega, \cdot, \cdot,\cdot)},
\end{equation} 
where $m^{f,g_0}$, $m^{f^\infty,g}$, and $m^{f^\infty,g_0}$ are defined as in \eqref{m-phipsi}, with $(f,g)$ replaced by $(f,g_0)$, $(f^\infty,g)$, and $(f^\infty,g_0)$, respectively.

\medskip

\begin{thm}[Homogenisation formulas]\label{en-density_vs}
Let $f$ be a stationary random volume integrand and let $g$ be a stationary random surface integrand  with respect to a group $(\tau_z)_{z\in \Z^n}$ of $P$-preserving transformations on $(\Om,\T,P)$.
Then there exists $\Om'\in \T$, with $P(\Om')=1$,
such that
\begin{itemize}
\item[(a)] for every $\om\in \Om'$, $x\in \R^{n}$, $\xi \in \R^{m\times n}$, $\nu\in \Sph^{n-1}$, and $k\in \N$ the limit  
\begin{equation*}
\lim_{r\to +\infty} \frac{m^{f,g_0}_{\om}(\ell_\xi,Q^{\nu,k}_r(rx))}{k^{n-1}r^{n}}
\end{equation*}
exists and is independent of $x$, $\nu$, and $k$; 
\item[(b)] for every $\om\in \Om'$, $x\in \R^{n}$, $\zeta\in \R^m$, $\nu\in \Sph^{n-1}$ the limit
$$
\lim_{r\to +\infty} \frac{m^{f^\infty,g}_{\om}(u_{r x,\zeta,\nu},Q^\nu_r(rx))}{r^{n-1}}
$$
exists and is independent of $x$. 
\end{itemize}
More precisely, there exist a random volume integrand $f_{\mathrm{hom}} \colon \Om\times \R^{m\times n} \to [0,+\infty)$, and a random surface integrand $g_{\mathrm{hom}} \colon \Om\times \R^m\times \Sph^{n-1} \to [0,+\infty)$ such that for every 
$\om\in\Om'$, $x \in \mathbb{R}^n$,
$\xi \in \mathbb{R}^{m\times n}, \zeta \in \mathbb{R}^m$, and $\nu \in \mathbb{S}^{n-1}$
\begin{gather}\label{psi0}
f_{\mathrm{hom}}(\om,\xi)=\lim_{r\to +\infty} \frac{m^{f,g_0}_{\om}(\ell_\xi,Q^{\nu,k}_r(rx))}{k^{n-1}r^{n}}=  \lim_{ r \to +\infty} \frac{m^{f,g_0}_{\om}(\ell_\xi,  Q_r)}{r^n},
\\
\label{phi0}
g_{\mathrm{hom}}(\om,\zeta,\nu)=\lim_{r\to +\infty} \frac{m^{f^\infty,g}_{\om}(u_{r x,\zeta,\nu},Q^\nu_r(rx))}{r^{n-1}}= \lim_{r\to +\infty} \frac{m^{f^\infty,g}_{\om}(u_{0,\zeta,\nu}, Q^\nu_r)}{r^{n-1}},
\end{gather}
 where $Q_r:=Q_r(0)$ and $Q^\nu_r=Q^\nu_r(0)$.

For  every $\om\in \Om'$ and $\xi \in \R^{m\times n}$  let 
$$
f^\infty_{\rm hom}(\omega,\xi):=\lim_{t\to+\infty}\frac{f_{\rm hom}(\omega,t\xi)}{t}
$$
(since $f_{\mathrm{hom}}(\om,\cdot)\in \mathcal{F}$, the existence of the limit is guaranteed by $(f5)$). Then   for every 
$\om\in \Om'$, $x\in \R^{n}$, $\xi \in \R^{m\times n}$, $\nu\in \Sph^{n-1}$, and $k\in \N$ we have 
\begin{equation}\label{hhhrrr}
f^\infty_{\rm hom}(\omega,\xi)=\lim_{r\to +\infty} \frac{m^{f^\infty,g_0}_{\om}(\ell_\xi,Q^{\nu,k}_r(rx))}{k^{n-1}r^{n}}=\lim_{r\to +\infty} \frac{m^{f^\infty,g_0}_{\om}(\ell_\xi, Q_r)}{r^{n}}.
\end{equation}
If, in addition, $(\tau_z)_{z\in \Z^n}$ is ergodic, then $f_{\mathrm{hom}}$ and $g_{\mathrm{hom}}$ are independent of $\om$ and
\begin{align}
f_{\mathrm{hom}}(\xi)&=\lim_{r\to +\infty}\, \frac{1}{r^n} {\int_\Om m^{f,g_0}_{\om}(\ell_\xi,  Q_r)\,dP(\om)}, 
\label{DM-8401}
\\
g_{\mathrm{hom}}(\zeta,\nu)&= \lim_{r\to +\infty} \frac{1}{r^{n-1}} 
\int_\Om m^{f^\infty,g}_{\om}(u_{0,\zeta,\nu}, Q^\nu_r) \, dP(\om),
\label{DM-8402}
\\
f^\infty_{\mathrm{hom}}(\xi)&=\lim_{r\to +\infty}\frac{1}{r^{n}} \int_\Om {m^{f^\infty,g_0}_{\om}(\ell_\xi, Q_r)\,dP(\om)}.
\label{DM-8403}
\end{align}
\end{thm}

We are now ready to state the main result of this paper, namely the almost sure $\Gamma$-convergence of the sequence of random functionals $(E_\e(\om))_{\e>0}$ introduced in \eqref{Eeps}.

\begin{thm}[Almost sure $\Gamma$-convergence]\label{G-convE}
Let $f$ and $g$ be stationary random  volume and surface integrands with respect to a group $(\tau_z)_{z\in \Z^n}$ of $P$-preserving transformations on $(\Om,\T,P)$, and for every $\e>0$ and $\om\in \Omega$ let $E_\e(\om)$ be as in \eqref{Eeps}. Let $\Om'\in \T$ (with $P(\Om')=1$), $f_{\mathrm{hom}}$,  $f^\infty_{\rm hom}$,  and $g_{\mathrm{hom}}$ be as in Theorem \ref{en-density_vs}, and for every $\om\in\Omega$ let $E_{\mathrm{hom}}(\om) \colon L^1_{\rm loc}(\R^n, \R^m)\times \A  \longrightarrow [0,+\infty]$ be the functional defined by
$$
E_{\mathrm{hom}}(\om)(u, A):= \int_A f_{\mathrm{hom}}(\om,\nabla u)\dx + \int_{S_u\cap A}g_{\mathrm{hom}}(\om,[u],\nu_u)\,d \mathcal{H}^{n-1}+\int_A f^\infty_{\rm hom}\Big(\om,\frac{dC(u)}{d|C(u)|}\Big)\,d|C(u)|,
$$
if $u|_A\in BV(A,\R^m)$, and by $E_{\mathrm{hom}}(\om)(u, A):= +\infty$, if $u|_A\notin BV(A,\R^m)$. 
Then for every $\om\in \Om'$ and every $A\in \A$ the functionals $E_\e(\om)(\cdot,A)$ $\Gamma$-converge  to $E_{\mathrm{hom}}(\om)(\cdot,A)$ in $L^1_{\rm loc}(\R^n,\R^m)$, as $\e\to 0+$. 

If, in addition, $(\tau_z)_{z\in \Z^n}$ is ergodic, then $E_{\mathrm{hom}}$ is a deterministic functional\ie it does not depend on~$\om$.
\end{thm}

Thanks to Theorem \ref{G-convE} we can also characterise the asymptotic behaviour of some minimisation problems 
involving $E_\e(\om)$. An example is shown in the corollary below.  Since for every $A\in\A$ the values of $E_\e(\om)(u,A)$ and $E_{\mathrm{hom}}(\om)(u,A)$ depend only on the restriction of $u$ to $A$, in the corollary we regard $E_\e(\om)(u,\cdot)$ and $E_{\mathrm{hom}}(\om)(u,\cdot)$ as functionals defined on $L^1(A,\R^m)$. 

\begin{cor}[Convergence of  minima and mininisers] 
Let $f$ and $g$ be stationary random  volume and surface 
 integrands with respect to a group $(\tau_z)_{z\in \Z^n}$ of $P$-preserving transformations on $(\Om,\T,P)$, and for every $\e>0$ and $\om\in \Omega$  let $E_\e(\om)$ be as in \eqref{Eeps}. Let $\Om'\in \T$ (with $P(\Om')=1$) be as in Theorem \ref{en-density_vs},  and let $E_{\mathrm{hom}}(\om)$ be as in Theorem \ref{G-convE}.
Given  $\om\in\Om'$, $A\in \A$,  and  $h\in L^1(A,\R^m)$, we have 
\begin{equation}\label{arb0}
\inf_{u\in SBV(A,\R^m)} \big(E_\e(\om)(u,A)+ \|u-h\|_{L^1(A,\R^m)}\big) \longrightarrow \hskip-10pt \min_{u\in BV(A,\R^m)}\big(E_{\rm hom}(\om)(u,A)+ \|u-h\|_{L^1(A,\R^m)}\big)
\end{equation}
as $\e\to 0+$.
Moreover, if $(u_\e) \subset SBV(A,\R^m)$ is a sequence such that
\begin{equation}\label{diag-min-seq}
E_\e(\om)(u_\e,A)+ \|u_\e-h\|_{L^1(A,\R^m)}
\le \inf_{u\in SBV(A,\R^m)} \big(E_\e(\om)(u,A)+ \|u-h\|_{L^1(A,\R^m)}\big)+\eta_\e
\end{equation}
for some $\eta_\e\to 0+$, then  there exists a  sequence  $\e_j\to 0+$ such that $(u_{\e_j})_{j \in \mathbb{N}}$ converges in $L^1(A,\R^m)$, as $j\to +\infty$, to a solution of the  minimisation  problem
\begin{equation}\label{arf}
\min_{u\in BV(A,\R^m)}\Big(E_{\rm hom}(\om)(u,A)+ \|u-h\|_{L^1(A,\R^m)}\Big).
\end{equation}
\end{cor}
\begin{proof}
 If $A$ has a Lipschitz boundary, then the functionals $E_\e(\om)(\cdot, A)+\|\cdot-h\|_{L^1(A,\R^m)}$ are equi-coercive in $L^1(A,\R^m)$ thanks to Remark \ref{BV estimate}.
Since we have $\Gamma$-convergence in $L^1(A,\R^m)$ by virtue of Theorem \ref{G-convE},   the proof readily follows from the fundamental property of $\Gamma$-convergence (see,  \textit{e.g.,} \cite[Corollary 7.20]{DM93}).  

 We now show that the convergence of minimum values and of minimisers can be obtained even if $\partial A$ is not regular. 
Let us fix $\om\in\Om'$, $A\in \A$, and $h\in L^1(A,\R^m)$.
By Theorem \ref{G-convE} for every $A'\in\A$ the functional $E_{\rm hom}(\om)(\cdot,A')$ is a $\Gamma$-limit in $L^1_{\rm loc}(\R^n,\R^m)$, hence it is lower semicontinuous in $L^1_{\rm loc}(\R^n,\R^m)$ (see \cite[Proposition 6.8]{DM93}). This implies that $E_{\rm hom}(\om)(\cdot,A')$, considered as a functional on $L^1(A',\R^m)$, is  lower semicontinuous. Since 
$$
E_{\rm hom}(\om)(\cdot,A)=\sup \{E_{\rm hom}(\om)(\cdot,A'): A'\in\A,\ A'\subset\subset A\}, 
$$
the functional $E_{\rm hom}(\om)(\cdot,A)$, defined on $L^1(A,\R^m)$, is lower semicontinuous with respect to the convergence in $L^1_{\rm loc}(A,\R^m)$. 

Since $f_{\mathrm{hom}}(\om,\cdot)$ and $f_{\mathrm{hom}}^\infty(\om,\cdot)$ belong to $\mathcal{F}$, while $g_{\mathrm{hom}}(\om,\cdot,\cdot)$ belongs to $\G$, it follows from the definition  of $E_{\rm hom}(\om)(\cdot,A)$ that
$c_2|Du|(A) + \|u\|_{L^1(A,\R^m)}\le E_{\rm hom}(\om)(u,A)+ \|u-h\|_{L^1(A,\R^m)}+\|h\|_{L^1(A,\R^m)}$ for every $u\in BV(A,\R^m)$. This shows that the functional $u\mapsto  E_{\rm hom}(\om)(u,A)+ \|u-h\|_{L^1(A,\R^m)}$ is coercive in $BV(A,\R^m)$ with respect to the convergence in $L^1_{\rm loc}(A,\R^m)$. Therefore it attains a minimum value in $BV(A,\R^m)$, which we denote by $\mu_0$.

Let $u_0$ be a minimum point in $BV(A,\R^m)$. We extend $u_0$ to a function of $L^1_{\rm loc}(\R^n,\R^m)$, still denoted by $u_0$. By $\Gamma$-convergence, for every sequence $(\e_j)$ of positive numbers converging to $0$ there exists a sequence $u_j$ converging to $u_0$ in  $L^1_{\rm loc}(\R^n,\R^m)$ such that
$E_{\e_j}(\om)(u_j,A)\to E_{\rm hom}(\om)(u_0,A)<+\infty$. By the definition of $E_{\e_j}(\om)$ we have $u_j\in SBV(A,\R^m)$ for $j$ large enough, hence
$$
\inf_{u\in SBV(A,\R^m)} \big(E_{\e_{j}}(\om)(u,A)+ \|u-h\|_{L^1(A,\R^m)}\big) \le E_{\e_j}(\om)(u_j,A)+ \|u_j-h\|_{L^1(A,\R^m)}.
$$
This implies that
$$
\limsup_{j\to+\infty}\inf_{u\in SBV(A,\R^m)} \big(E_{\e_j}(\om)(u,A)+ \|u-h\|_{L^1(A,\R^m)}\big) \le \mu_0.
$$
Since the sequence $\e_j\to0$ is arbitrary, we obtain
\begin{equation}\label{arc}
\limsup_{\e\to0+}\inf_{u\in SBV(A,\R^m)} \big(E_\e(\om)(u,A)+ \|u-h\|_{L^1(A,\R^m)}\big) \le \mu_0.
\end{equation}

To prove the opposite inequality for the liminf, as well as the last statement of the corollary, we fix a sequence $(u_\e) \subset SBV(A,\R^m)$ satisfying \eqref{diag-min-seq}. For every sequence $(\e_j)$ of positive numbers converging to $0$, by  Remark \ref{BV estimate} and by \eqref{arc} the sequence $(u_{\e_j})$ is bounded in $BV(A,\R^m)$. Therefore a subsequence, not relabelled, converges in $L^1_{\rm loc}(A,\R^m)$ to a function $u_*\in BV (A,\R^m)$.

Given $A'\in\A$, with $A'\subset\subset A$, we can consider the functions $v_j$, defined by $v_j:=u_{\e_j}$ in $A'$ and 
$v_j:=0$ in $\R^n\setminus A'$, which converge in $L^1(\R^n,\R^m)$ to the function $v_*$, defined by $v_*:=u_*$ in $A'$ and 
$v_*:=0$ in $\R^n\setminus A'$. Since $E_{\e_j}(\om)(\cdot,A')$ $\Gamma$-converges to $E_{\rm hom}(\om)(\cdot,A')$ in  $L^1_{\rm loc}(\R^n,\R^m)$, we have
$$
E_{\rm hom}(\om)(u_*,A')=E_{\rm hom}(\om)(v_*,A')\le \liminf_{j\to+\infty}E_{\e_j}(\om)(v_j,A')\le  \liminf_{j\to+\infty}E_{\e_j}(\om)(u_{ \e_j} ,A),
$$
which implies that
$$
E_{\rm hom}(\om)(u_*,A')+\|u_*-h\|_{L^1(A',\R^m)}\le \liminf_{j\to+\infty}\big(E_{\e_j}(\om)(u_{\e_j},A) + \|u_{\e_j}-h\|_{L^1(A,\R^m)}\big).
$$
Taking the supremum for $A'\subset\subset A$ in the previous inequalities we obtain
\begin{align}
E_{ \rm hom}(\om)(u_*,A)\le  \liminf_{j\to+\infty}E_{\e_j}(\om)(u_{\e_j} ,A),
\label{ard0}
\end{align}
and 
\begin{align}
\mu_0&\le E_{\rm hom}(\om)(u_*,A)+\|u_*-h\|_{L^1(A,\R^m)}\le  \liminf_{j\to+\infty}\big(E_{\e_j}(\om)(u_{\e_j},A) + \|u_{\e_j}-h\|_{L^1(A,\R^m)}\big)
\nonumber
\\
&\le
\liminf_{j\to+\infty}\inf_{u\in SBV(A,\R^m)} \big(E_{\e_j}(\om)(u,A)+ \|u-h\|_{L^1(A,\R^m)}\big).
\label{arg}
\end{align}
By the arbitrariness of the sequence $\e_j\to0$, this chain of inequalities, together with \eqref{arc}, gives \eqref{arb0}  and shows that $u_*$ is a solution of the minimisation problem \eqref{arf}. 

In turn, \eqref{arb0}, \eqref{ard0}, and \eqref{arg} imply that  $\|u_{\e_j}-h\|_{L^1(A,\R^m)}\to \|u_*-h\|_{L^1(A,\R^m)}$. Since $u_{\e_j}$ converges to $u_*$ in $L^1_{\rm loc}(A,\R^m)$, from the general version of the Dominated Convergence Theorem we obtain that $u_{\e_j}$ converges to $u_*$ in $L^1(A,\R^m)$. This concludes the proof of the last statement of the corollary.
\end{proof}


\section{Deterministic homogenisation: properties of the homogenised integrands}\label{determ-hom}

Let $f\in \mathcal F$ and $g\in \G$. For $\e>0$ consider the functionals $E_\e \colon L^1_{\rm loc}(\R^n,\R^m)\times \A \longrightarrow [0,+\infty]$ defined by
\begin{equation}\label{E-det}
E_\e(u,A) :=\begin{cases}
\ds\int_A\!\! f(\tfrac{x}{\e}, \nabla u)\dx + \!\! \int_{S_u\cap A} \!\!\!\!\!  g(\tfrac{x}\e,[u],\nu_u)\,d \mathcal{H}^{n-1}& \!\!\text{if} \; u|_A\!\in SBV(A,\R^m),\cr
+\infty & \!\! \text{otherwise\! in}\,L^{1}_{\rm loc}(\R^n,\R^m).
\end{cases}
\end{equation}
In this section we prove the $\Gamma$-convergence of $E_\e$  under suitable assumptions on  $f$ and $g$, which are more general than the periodicity with respect to~$x$.

The main result of this section is  the following theorem.

\begin{thm}[Homogenisation]\label{T:det-hom}  
Let $f\in \mathcal F$, $g\in \G$, and let $m^{f,g_0}$  and  $m^{f^\infty,g}$ be defined as in \eqref{m-phipsi}  with $(f,g)$ replaced by  $(f,g_0)$  and  $(f^\infty,g)$, respectively. Assume that
\begin{itemize}
\item[(a)] for every $x\in \R^{n}$, $\xi \in \R^{m\times n}$, $\nu\in \Sph^{n-1}$, and $k\in \N$  the limit  
\begin{equation}\label{f-hom}
\lim_{r\to +\infty} \frac{m^{f,g_0}(\ell_\xi,Q_r^{\nu,k}(rx))}{ k^{n-1} r^{n}}=:f_{\rm hom}(\xi)
\end{equation}
exists and is independent of $x$,  $\nu$, and $k$; 
\item[(b)] for every $x\in \R^{n}$, $\zeta\in \R^m$, and $\nu\in \Sph^{n-1}$ the limit  
\begin{equation}\label{g-hom}
\lim_{r\to +\infty} \frac{m^{f^\infty,g}(u_{r x,\zeta,\nu},Q^\nu_r(rx))}{r^{n-1}}=:g_{\rm hom}(\zeta,\nu)
\end{equation} 
exists and is independent of $x$.
\end{itemize}
Then $f_{\rm hom}\in \mathcal F$  and  $g_{\rm hom}\in \G$.  Let $f_{\rm hom}^\infty$ be the recession function of $f_{\rm hom}$ and let  $E_{\rm hom}\colon L^1_{\rm loc}(\R^n,\R^m) \times \A \longrightarrow [0,+\infty]$  be the functional defined by
\begin{equation}\label{det-Glim}
E_{\mathrm{hom}}(u, A):= \int_A f_{\mathrm{hom}}(\nabla u)\dx + \int_{S_u\cap A}g_{\mathrm{hom}}([u],\nu_u)d \mathcal{H}^{n-1}+\int_A  f_{\rm hom}^\infty\Big(\frac{dC(u)}{d|C(u)|}\Big)\,d|C(u)|
\end{equation}
if $u|_A\in BV(A,\R^m)$, while
$E_{\mathrm{hom}}(u, A):=+\infty$ if $u|_A\notin BV(A,\R^m)$.
 Then, for every $A\in \A$ the functionals $E_\e(\cdot, A)$ defined as in \eqref{E-det} 
 $\Gamma$-converge to $E_{\rm hom}(\cdot, A)$ in $L^1_{\rm loc}(\R^n,\R^m)$,  as $\e\to 0+$, meaning that  for  every  sequence $(\e_j)$ of positive numbers converging to zero the sequence $(E_{\e_j}(\cdot, A))$ $\Gamma$-converges to $E_{\rm hom}(\cdot, A)$ in $L^1_{\rm loc}(\R^n,\R^m)$.
\end{thm}

The proof of the homogenisation result Theorem \ref{T:det-hom} will be carried out in three main steps. In the first step (Lemmas \ref{l:f-hom} and \ref{l:g-hom})  we show that $f_{\rm hom}\in \mathcal F$  and  $g_{\rm hom}\in \G$. In the second step (Theorem \ref{thm:Gamma-conv})  we prove that, up to subsequences, for every $A\in \A$ the functionals $E_\e(\cdot,A)$ $\Gamma$-converge to some functional $\widehat E(\cdot,A)$, whose domain is $BV(A,\R^m)$. Further, we prove that $\widehat E$ satisfies some suitable properties both as a functional and as a set-function.  In particular $\widehat E(u,\cdot)$  is the restriction to $\A$ of a Borel measure. 

In the third and last step we show that \eqref{f-hom}  and  \eqref{g-hom} imply, respectively, that the following identities hold true for every $A \in \A$ and for every $u\in L^1_{\rm loc}(\R^n,\R^m)$ with $u|_A\in BV(A,\R^m)$: 
\begin{gather}\label{derivata-lebesgue}
\frac{d\widehat E(u,\cdot)}{d \mathcal L^n}(x)=f_{\rm hom}(\nabla u(x))\qquad\hbox{for }\mathcal L^n\hbox{-a.e.\ }x\in A, 
\\
\label{derivata-hausdorff}
\frac{d \widehat E(u,\cdot)}{d\mathcal H^{n-1}\LLL{S_u}}(x) = g_{\rm hom}([u](x),\nu_{u}(x))\qquad \hbox{for }\mathcal H^{n-1}\hbox{-a.e.\ } x\in S_u \cap A,
\\
\label{derivata-cantor}
\frac{d\widehat E(u,\cdot)}{d|C(u)|}(x)= f_{\rm hom}^\infty\Big(\frac{dC(u)}{d|C(u)|}(x)\Big)\qquad\hbox{for }|C(u)|\hbox{-a.e.\ }x\in A
\end{gather}
(see Propositions \ref{p:homo-vol}, \ref{p:homo-sur}, and \ref{p:homo-Can}). 
Moreover,  thanks to \eqref{derivata-lebesgue}-\eqref{derivata-cantor} we deduce that $\widehat E$ coincides with the functional $E_{\rm hom}$ defined in \eqref{det-Glim}; as a consequence, the $\Gamma$-convergence result  proved in the second step actually holds true for the whole sequence $(E_\e)$.

 In the next lemmas  we prove that the homogenised integrands $f_{\rm hom}$  and  $g_{\rm hom}$ belong to the classes $\mathcal F$ and $\G$, respectively.

\begin{lem}\label{l:f-hom}
Let $f \in \F$ and $g\in\G$. Assume that hypothesis (a) of Theorem \ref{T:det-hom} is satisfied and let $f_{\rm hom}$ be defined as in \eqref{f-hom}. Then $f_{\rm hom} \in \mathcal F$. 
\end{lem}
 
\begin{proof}
To prove $(f2)$ we fix $\xi_1$, $\xi_2 \in \R^{m{\times}n}$ and set $\xi:=\xi_2-\xi_1$. 
We claim that  for every $r>0$ 
\begin{equation}\label{aea}
|m^{f,g_0}(\ell_{\xi_1}, Q_r)- m^{f,g_0}(\ell_{\xi_2}, Q_r)| \leq \sigma_1(|\xi|)(m^{f,g_0}(\ell_{\xi_1}, Q_r)+m^{f,g_0}(\ell_{\xi_2}, Q_r)) +  c_1 |\xi|r^n,
\end{equation} 
 where $Q_r:=Q_r(0)$. 
Indeed, by $(f2)$, for every $u\in SBV(Q_r,\R^m)$  we have 
$$
E^{f,g_0}(u+\ell_\xi,Q_r)\le E^{f,g_0}(u,Q_r)+\sigma_1(|\xi|)\big(E^{f,g_0}(u+\ell_\xi,Q_r)+E^{f,g_0}(u,Q_r)\big)+ c_1|\xi|r^n. 
$$
By rearranging the terms we get 
\begin{equation*}
(1-\sigma_1(|\xi|))E^{f,g_0}(u+\ell_\xi,Q_r)\le (1+\sigma_1(|\xi|))E^{f,g_0}(u,Q_r)+c_1|\xi|r^n.
\end{equation*}
If $\sigma_1(|\xi|)< 1$, we minimise over all functions $ u\in SBV(Q_r,\R^m)$  such that  $u=\ell_{\xi_1}$ near $\partial Q_r$  and, using \eqref{m-phipsi}, we  obtain 
$$
(1-\sigma_1(|\xi|)) m^{f,g_0}(\ell_{\xi_2}, Q_r)\le (1+\sigma_1(|\xi|))m^{f,g_0}(\ell_{\xi_1}, Q_r)+ 
c_1|\xi|r^n.
$$
This inequality is trivial if $\sigma_1(|\xi|)\ge 1$.
Exchanging the roles of $\xi_1$ and $\xi_2$ we obtain \eqref{aea}. We now divide both sides of this inequality by $r^n$ and, passing to the limit as $r\to +\infty$, from \eqref{f-hom} we obtain that $f_{\rm hom}$ satisfies $(f2)$.

Property $(f1)$ for  $f_{\rm hom}$ follows from the continuity estimate $(f2)$, since $f_{\rm hom}$ does not depend on~$x$. The lower bound $(f3)$ for  $f_{\rm hom}$ follows from the lower bound in Remark \ref{BV estimate}, which gives
$$
c_2\inf |Du|(Q_r)\le m^{f,g_0}(\ell_{\xi}, Q_r)
$$
for every $\xi \in \R^{m{\times}n}$, where the infimum is over all functions $u\in SBV(Q_r,\R^m)$, and such that  $u=\ell_\xi$ near $\partial Q_r$. By Jensen's inequality the left-hand side is equal to $ c_2|\xi|r^n$. Using \eqref{f-hom} we conclude that that $f_{\rm hom}$ satisfies $(f3)$.

Property $(f4)$ for  $f_{\rm hom}$ follows from the fact that for every $\xi \in \R^{m{\times}n}$ we have
$$
\frac{1}{r^n} m^{f,g_0}(\ell_\xi, Q_r)\leq \frac{1}{r^n} E^{f,g_0}(\ell_\xi,Q_r) =\frac{1}{r^n}\int_{Q_r} f( x ,\xi)\dx  \leq c_3|\xi|+c_4.
$$
Passing to the limit as $r\to+\infty$, from \eqref{f-hom} we obtain that $f_{\rm hom}$ satisfies $(f4)$.
 
We now prove that $f_{\rm hom}$  satisfies $(f5)$. Fix $\xi \in \R^{m\times n}$, $s>0$, $t>0$,  and $\eta\in(0,1)$. By \eqref{m-phipsi} for every $r>0$ there exists
$u_r\in SBV(Q_r,\R^m)$,  with $u_r=\ell_\xi$ near $\partial Q_r$, such that
\begin{equation}\label{abj-0}
 \int_{Q_r}f( x ,t\nabla u_r)\dx +\int_{S_{u_r}\cap Q_r}g_0( x ,t[u_r],\nu_{u_r})\,d\mathcal H^{n-1}\le m^{f,g_0}(\ell_{t\xi},Q_r)+\eta r^n.
\end{equation} 
By \eqref{abd} for every $ x \in \R^n$ and $\xi \in \R^{m\times n}$ we have
\begin{equation*}
 \tfrac{1}{s}f( x ,s\xi) - \tfrac{c_5}{s} -  \tfrac{c_5}{s}f( x ,s\xi)^{1-\alpha} \leq   \tfrac{1}{t}f( x ,t\xi) + \tfrac{c_5}{t}+\tfrac{c_5}{t}f( x ,t\xi)^{1-\alpha},
\end{equation*} 
hence, using the positive $1$-homogeneity of $g_0$,
\begin{eqnarray*}
&&  \frac{1}{s} \frac{1}{r^n} \int_{Q_r}f( x ,s\nabla u_r)\dx - \frac{c_5}{s}- \frac{c_5}{s}\frac{1}{r^n}\int_{Q_r}f( x ,s\nabla u_r)^{1-\alpha} dx  + \frac{1}{s} \frac{1}{r^n} \int_{S_{u_r}\cap Q_r}g_0( x ,s[u_r],\nu_{u_r})\,d\mathcal H^{n-1}
\nonumber
\\
&& \le  \frac{1}{t}\frac{1}{r^n}\int_{Q_r}f( x ,t\nabla u_r)\dx  + \frac{c_5}{t}+ \frac{c_5}{t}\frac{1}{r^n}\int_{Q_r}f( x ,t\nabla u_r)^{1-\alpha} dx  +  \frac{1}{t} \frac{1}{r^n} \int_{S_{u_r}\cap Q_r}g_0( x ,t[u_r],\nu_{u_r})\,d\mathcal H^{n-1}.
\end{eqnarray*}
By H\"older's inequality we obtain
\begin{eqnarray*}
\hskip-12pt&&  \frac{1}{s} \frac{1}{r^n} \int_{Q_r}f( x ,s\nabla u_r)\dx  + \frac{1}{s} \frac{1}{r^n} \int_{S_{u_r}\cap Q_r}g_0( x ,s[u_r],\nu_{u_r})\,d\mathcal H^{n-1} - 
\frac{c_5}{s}- \frac{c_5}{s}\Big(\frac{1}{r^n}\int_{Q_r}f( x ,s\nabla u_r)\dx \Big)^{1-\alpha} 
\nonumber
\\
\hskip-12pt&& \le  \frac{1}{t}\frac{1}{r^n}\int_{Q_r}f( x ,t\nabla u_r)\dx   +  \frac{1}{t} \frac{1}{r^n} \int_{S_{u_r}\cap Q_r}g_0( x ,t[u_r],\nu_{u_r})\,d\mathcal H^{n-1} + \frac{c_5}{t}+ \frac{c_5}{t}\Big(\frac{1}{r^n}\int_{Q_r}f( x ,t\nabla u_r)\dx \Big)^{1-\alpha}.
\end{eqnarray*}
By \eqref{abj-0}  this inequality implies that
\begin{eqnarray}
&&  \frac{1}{s} \Big( \frac{1}{r^n} \int_{Q_r}f( x ,s\nabla u_r)\dx  + \frac{1}{r^n} \int_{S_{u_r}\cap Q_r}g_0( x ,s[u_r],\nu_{u_r})\,d\mathcal H^{n-1} \Big) - 
\frac{c_5}{s}
\nonumber
\\
&&- \frac{c_5}{s}\Big(\frac{1}{r^n}\int_{Q_r}f( x ,s\nabla u_r)\dx +
\frac{1}{r^n} \int_{S_{u_r}\cap Q_r}g_0( x, s[u_r],\nu_{u_r})\,d\mathcal H^{n-1}\Big)^{1-\alpha} 
\nonumber
\\
&& \le  \frac{1}{t}\Big(\frac{1}{r^n}m^{f,g_0}(\ell_{t\xi},Q_r) +\eta\Big) + \frac{c_5}{t}+ \frac{c_5}{t}\Big(\frac{1}{r^n}m^{f,g_0}(\ell_{t\xi},Q_r) +\eta\Big)^{1-\alpha}.
\label{abp}
\end{eqnarray}
If 
\begin{equation}\label{abl}
 \frac{1}{r^n} m^{f,g_0}(\ell_{s\xi},Q_r)-c_5\Big( \frac{1}{r^n} m^{f,g_0}(\ell_{s\xi},Q_r)\Big)^{1-\alpha} \le 0,
 \end{equation}
 then we have 
 \begin{eqnarray}
 && \frac{1}{s} \frac{1}{r^n} m^{f,g_0}(\ell_{s\xi},Q_r)- \frac{c_5}{s} - \frac{c_5}{s} \Big( \frac{1}{r^n} m^{f,g_0}(\ell_{s\xi},Q_r)\Big)^{1-\alpha}
  \nonumber
  \\
&& \le  \frac{1}{t}\Big(\frac{1}{r^n}m^{f,g_0}(\ell_{t\xi},Q_r) +\eta\Big) + \frac{c_5}{t}+ \frac{c_5}{t}\Big(\frac{1}{r^n}m^{f,g_0}(\ell_{t\xi},Q_r) +\eta\Big)^{1-\alpha},
\label{abq}
 \end{eqnarray}
 just because the left-hand side is negative and the right-hand side is positive. Since the function $\tau\mapsto \tau-c_5\tau^{1-\alpha}$,
 defined for $\tau>0$, is increasing in the half-line where it is positive, from the inequality 
 $$
 m^{f,g_0}(\ell_{s\xi},Q_r)\le  \int_{Q_r}f( x ,s\nabla u_r)\dx  +\int_{S_{u_r}\cap Q_r}g_0( x ,s[u_r],\nu_{u_r})\,d\mathcal H^{n-1} 
 $$
and from \eqref{abp} we deduce that \eqref{abq} is  satisfied even if \eqref{abl} is not.

Passing to the limit first as $r\to+\infty$ and then as $\eta\to0+$, from \eqref{f-hom} and \eqref{abq} we obtain
$$
 \frac{1}{s} f_{\rm hom}(s\xi) - \frac{c_5}{s} - \frac{c_5}{s} f_{\rm hom}(s\xi)^{1-\alpha}
  \le  \frac{1}{t} f_{\rm hom}(t\xi) + \frac{c_5}{t}+ \frac{c_5}{t} f_{\rm hom}(t\xi) ^{1-\alpha}.
$$
By exchanging the roles of $s$ and $t$ we obtain \eqref{abd}. 
Recalling that $ f_{\rm hom}$ satisfies $(f4)$, we can apply Remark \ref{abc} and we obtain that  $ f_{\rm hom}$ satisfies $(f5)$. 
\end{proof}

To prove that $g_{\rm hom} \in \mathcal G$ we need the truncation result given by the following lemma, which will be used
several times in this paper. The proof is given in \cite[Lemma 3.7]{BBBF} (see also \cite[Lemma 2.8]{BFM}).

\begin{lem}\label{truncation}
Let $C_1>0$, $C_2>0$, and $\eta>0$. Then there exists a constant $M=M(C_1,C_2,\eta)>0$ such that for every $f\in\F$ and $g\in\G$,  for every $A\in\A$, for every $w\in SBV(A,\R^m)\cap L^\infty(A,\R^m)$, with $\|w\|_{L^\infty(A,\R^m)}\le C_1$, and for every $u\in BV(A,\R^m)$, with
$\|u\|_{L^1(A,\R^m)}+|Du|(A)\le C_2$ and $u=w$ near $\partial A$, there exists $\tilde u\in SBV(A,\R^m)\cap L^\infty(A,\R^m)$ such that
\begin{itemize}
\item[(a)]$\|\tilde u\|_{L^\infty(A,\R^m)}\le M$,
\item[(b)]$E^{f,g}(\tilde u,A)\le E^{f,g}(u,A)+\eta$,
\item[(c)]$\|\tilde u-w\|_{L^1(A,\R^m)}\le \|u-w\|_{L^1(A,\R^m)}$,
\item[(d)]$\tilde u=w$ near $\partial A$.
\end{itemize}
\end{lem}

\begin{rem}\label{rem:truncation0}
A careful inspection of the proof of \cite[Lemma 3.7]{BBBF} shows that 
the lemma also applies if $u$ only attains the boundary conditions on a subset of $\partial A$, as defined 
after \eqref{m-phipsi}. 
More precisely, if $\Lambda \subset \partial A$ is a relatively open subset of $\partial A$,  $U$ is a neighbourhood of $\Lambda$ in $\R^n$,
$w\in SBV(A,\R^m)\cap L^\infty(U \cap A,\R^m)$,
and $u=w$ $\mathcal{L}^n$-a.e.\ in $U \cap A$, then
the conclusion still holds true, with (d) replaced by
\begin{itemize}
\item[(d$'$)] $\tilde{u}=w$ $\mathcal{L}^n$-a.e.\ in $U \cap A$,
\end{itemize}
and in this case $M=M(\widetilde{C}_1,C_2,\eta)>0$, where $\| w \|_{L^\infty(U \cap A,\R^m)} \leq \tilde{C}_1.$
\end{rem}

%
%

We are now ready to prove that $g_{\rm hom} \in \mathcal G$.

\begin{lem}\label{l:g-hom}
Let $f \in \F$ and $g \in \G$. Assume that hypothesis (b) of Theorem \ref{T:det-hom} is satisfied and let $g_{\rm hom}$ be defined as in \eqref{g-hom}. Then $g_{\rm hom} \in \mathcal G$. 
\end{lem}

\begin{proof}
To prove $(g2)$ we fix $\zeta_1$, $\zeta_2 \in \R^m$ and $\nu\in\Sph^{n-1}$, and we set $\zeta:=\zeta_2-\zeta_1$. 
We claim that  for every $r>0$ 
\begin{equation}\label{aek}
|m^{f^\infty,g}(u_{0,\zeta_1,\nu}, Q^\nu_r)- m^{f^\infty,g}(u_{0,\zeta_2,\nu}, Q^\nu_r)| \leq \sigma_2(|\zeta|)(m^{f^\infty,g}(u_{0,\zeta_1,\nu}, Q^\nu_r)+m^{f^\infty,g}(u_{0,\zeta_2,\nu}, Q^\nu_r)),
\end{equation} 
 where $Q^\nu_r:=Q^\nu_r(0)$. Indeed,  for every $u\in SBV(Q^\nu_r,\R^m)$, by $(g2)$  we have
$$
E^{f^\infty,g}(u+u_{0,\zeta,\nu},Q^\nu_r)\le E^{f^\infty,g}(u,Q^\nu_r)+ \sigma_2(|\zeta|)\big(E^{f^\infty,g}(u+u_{0,\zeta,\nu},Q^\nu_r)+E^{f^\infty,g}(u,Q^\nu_r)\big).
$$
By rearranging the terms we get 
\begin{equation*}
(1-\sigma_2(|\zeta|))E^{f^\infty,g}(u+u_{0,\zeta,\nu},Q^\nu_r)\le (1+\sigma_2(|\zeta|))E^{f^\infty,g}(u,Q^\nu_r).
\end{equation*}
If  $\sigma_2(|\zeta|)< 1$, we minimise over all functions  $u\in SBV(Q^\nu_r,\R^m)$  such that  $u=u_{0,\zeta_1,\nu}$ near $\partial Q^\nu_r$  and by \eqref{m-phipsi} we  obtain 
$$
(1-\sigma_2(|\zeta|)) m^{f^\infty,g}(u_{0,\zeta_2,\nu}, Q^\nu_r)\le (1+\sigma_2(|\zeta|))m^{f^\infty,g}(u_{0,\zeta_1,\nu}, Q^\nu_r).
$$
This inequality is trivial if $\sigma_2(|\zeta|)\ge 1$.
Exchanging the roles of $\zeta_1$ and $\zeta_2$ we obtain \eqref{aek}. We now divide both sides of this inequality by $r^{n-1}$ and, passing to the limit as $r\to+\infty$, from \eqref{g-hom} we obtain that $g_{\rm hom}$ satisfies $(g2)$.

In view of  $(g2)$, to prove  $(g1)$ for  $g_{\rm hom}$ it is enough to show that for every $\zeta\in\R^m$ the restriction of the function $\nu\mapsto g_{\rm hom}(\zeta,\nu)$
 to the sets $\widehat\Sph^{n-1}_+$ and $\widehat\Sph^{n-1}_-$ is continuous.
We only prove this property for $\widehat\Sph^{n-1}_+$, 
the other case being analogous.
To this end, let us fix $\zeta\in\R^m$, $\nu\in \widehat\Sph^{n-1}_+$, 
and  a sequence  $(\nu_j) \subset \widehat\Sph^{n-1}_+$  such that $\nu_j \to \nu$ as $j\to +\infty$. 
Since the function $\nu\mapsto R_\nu$ is continuous on 
$\widehat\Sph^{n-1}_+$, for every $\delta \in(0,\frac12)$ there exists an integer $j_\delta$ such that
\begin{equation}\label{afa}
|\nu_j-\nu|<\delta\quad\hbox{and}\quad Q^{\nu_j}_{(1-\delta)r}  \subset\subset  Q^\nu_{r}  \subset\subset  Q^{\nu_j}_{(1+\delta)r},
\end{equation}
for every $j\geq j_\delta$ and every $r>0$.
Fix $j\geq j_\delta$, $r>0$, and $\eta>0$. 
 By \eqref{m-phipsi}  there exists  $u\in  SBV(Q^\nu_r,\R^m)$, with $u=u_{0,\zeta,\nu}$ near 
$\partial Q^\nu_r$, such that
\begin{equation}\label{minmin}
\int_{Q^\nu_r}f^{\infty}( x ,\nabla u)\dx + \int_{S_{u}\cap Q^\nu_r}
g( x ,[u],\nu_{u})\,d\mathcal H^{n-1} \leq m^{f^{\infty},g}(u_{ 0 , \zeta, \nu},Q^\nu_r) + \eta.  
\end{equation}
We define $v \in SBV_{\rm loc} (Q^{\nu_j}_{(1+\delta)r}, \R^m)$ as
$$
v( x ) := 
\begin{cases}
u( x ) \,\, &\textrm{if} \;  x \in Q^\nu_r,\\
u_{0,\zeta,\nu_j}( x ) \,\, &\textrm{if} \;  x \in Q^{\nu_j}_{(1+\delta)r}\setminus Q^\nu_r.
\end{cases} 
$$
Then $v=u_{0,\zeta,\nu_j}$  near  $\partial Q^{\nu_j}_{(1+\delta)r}$
and $S_{v}\subset S_{u}\cup\Sigma$, where 
$$
 \Sigma:=\big\{ x \in \partial Q^\nu_r: ( x  \cdot \nu) ( x  \cdot \nu_j) < 0
  \big\}  \cup \big( \Pi^{\nu_j}_0 \cap(Q^{\nu_j}_{(1+\delta)r}\setminus Q^\nu_r)\big).
$$
By \eqref{afa} there exists $\varsigma(\delta)>0$, independent of $j$ and $r$, with $\varsigma(\delta)\to 0$  as $\delta \to 0+$, such that $\hs^{n-1}(\Sigma)\le  \varsigma(\delta)r^{n-1}$. Thanks to $(g4)$, \eqref{m-phipsi} and \eqref{minmin} we then have 
\begin{align*}
&m^{f^{\infty},g} (u_{ 0 , \zeta, \nu_j},Q^{\nu_j}_{(1 + \delta) r} )
\leq \int_{Q^{\nu_j}_{(1 + \delta) r} }f^{\infty}( x ,\nabla v)\dx 
+ \int_{S_{v}\cap Q^{\nu_j}_{(1 + \delta) r} }
g(  x ,[v],\nu_{v})\,d\mathcal H^{n-1} \\
&\quad\leq \int_{Q^\nu_r}f^{\infty}(  x ,\nabla u)\dx + \int_{S_{u}\cap Q^\nu_r}
g(  x ,[u],\nu_{u})\,d\mathcal H^{n-1} 
+ c_3  |\zeta|  \varsigma(\delta)  r^{n-1} \\
&\quad \leq m^{f^{\infty},g} (u_{ 0 , \zeta, \nu},Q^\nu_r) + \eta
+ c_3  |\zeta|  \varsigma(\delta)  r^{n-1},
\end{align*}
where we used the fact that $f^{\infty} ( x , 0) = 0$ for every $ x  \in \R^n$.
Dividing by $r^{n-1}$ and passing to the limit as $r\to +\infty$, recalling \eqref{g-hom} we obtain
\begin{align*}
g_{\rm hom}( \zeta,\nu_j)(1+\delta)^{n-1} \leq g_{\rm hom}( \zeta,\nu) + c_3  |\zeta|  \varsigma(\delta).
\end{align*}
Letting $j\to +\infty$ and then $\delta \to 0+$  we  deduce that 
\begin{equation*}
\limsup_{j \to +\infty }g_{\rm hom}( \zeta,\nu_j) \leq g_{\rm hom}( \zeta,\nu).
\end{equation*}
An analogous argument,  now using the cubes $Q^{\nu_j}_{(1-\delta)r}$, implies that  
\begin{equation*}
g_{\rm hom}( \zeta,\nu) \leq \liminf_{j \to +\infty} g_{\rm hom}( \zeta,\nu_j),
\end{equation*}
hence the restriction of the function $\nu\mapsto g_{\rm hom}(\zeta,\nu)$
 to $\widehat\Sph^{n-1}_+$ is continuous. This concludes the proof of $(g1)$ for  $g_{\rm hom}$. 

The lower bound $(g3)$ for  $g_{\rm hom}$ can be obtained from the lower bound in Remark \ref{BV estimate}, which gives
$$
c_2\inf |Du|(Q^\nu_r)\le m^{f^\infty,g}(u_{0,\zeta,\nu}, Q^\nu_r)
$$
for every $\zeta \in \R^m$ and $\nu\in\Sph^{n-1}$, where the infimum is over all functions $u\in SBV(Q^\nu_r,\R^m)$ and 
 such that  $u=u_{0,\zeta,\nu}$ near $\partial Q^\nu_r$.  In turn, this infimum is larger than or equal to
\begin{equation}\label{arb}
c_2\inf |Dv|(Q^\nu_r),
\end{equation}
where the infimum is now over all scalar functions $v\in SBV(Q^\nu_r)$ and such that $v=u_{0,|\zeta|,\nu}$ near $\partial Q^\nu_r$.
Using \eqref{g-hom}, property $(g3)$ for  $g_{\rm hom}$ follows from these inequalities and from the fact that the value of \eqref{arb} is  $c_2|\zeta|r^{n-1}$. This is a well kown fact, which can be proved, for instance, using a slicing argument based on \cite[Theorem 3.103]{AFP}. 

Property $(g4)$ for  $g_{\rm hom}$ follows from the fact that for every $\zeta \in \R^m$ and $\nu\in \Sph^{n-1}$ we have
$$
\frac{1}{r^{n-1}} m^{f^\infty,g}(u_{0,\zeta,\nu}, Q^\nu_r)\leq \frac{1}{r^{n-1}} E^{f^\infty,g}(u_{0,\zeta,\nu},Q^\nu_r) =\frac{1}{r^{n-1}}\int_{\Pi^\nu_0\cap Q^\nu_r} g( x ,\zeta,\nu)\dx  \leq c_3|\zeta|.
$$
Passing to the limit as $r\to+\infty$, from \eqref{g-hom} we obtain that $g_{\rm hom}$ satisfies $(g4)$.

We now prove that $g_{\rm hom}$ satisfies $(g5)$. Fix $\zeta\in \Sph^{m-1}$, $\nu\in \Sph^{n-1}$, $s>0$, $t>0$, and $\eta\in (0,1)$. By \eqref{m-phipsi}  for every $r>0$ there exists
$v_r\in SBV(Q^\nu_r,\R^m)$, with $v_r=u_{0,\zeta,\nu}$ near $\partial Q^\nu_r$, such that
\begin{equation*}
 \int_{Q^\nu_r}f^\infty( x ,t\nabla v_r)\dx +\int_{S_{v_r}\cap Q^\nu_r}g( x ,t[v_r],\nu_{v_r})\,d\mathcal H^{n-1}\le m^{f^\infty,g}(u_{0,t\zeta,\nu},Q^\nu_r)+\eta t r^{n-1},
\end{equation*} 
hence
\begin{equation*}
 \int_{Q^\nu_r}f^\infty( x ,\nabla v_r)\dx +\frac1t\int_{S_{v_r}\cap Q^\nu_r}g( x ,t[v_r],\nu_{v_r})\,d\mathcal H^{n-1}\le \frac1t m^{f^\infty,g}(u_{0,t\zeta,\nu},Q^\nu_r)+\eta r^{n-1},
\end{equation*} 
where we used the positive $1$-homogeneity of $f^\infty(x,\cdot)$.  

Let $Q^\nu:=Q^\nu_1(0)$ and let $w_r\in SBV(Q^\nu,\R^m)$ be the rescaled function, defined by $w_r( x ):=v_r(r x )$ for every $ x \in Q^\nu$. Then $w_r=u_{0,\zeta,\nu}$ near $\partial Q^\nu$ and, by a change of variables,
\begin{equation}\label{abj}
 \int_{Q^\nu}f^\infty(r x ,\nabla w_r)\dx +\frac1t\int_{S_{w_r}\cap Q^\nu}g(r x ,t[w_r],\nu_{w_r})\,d\mathcal H^{n-1}\le \frac1t \frac{1}{r^{n-1}}m^{f^\infty,g}(u_{0,t\zeta,\nu},Q^\nu_r)+\eta,
\end{equation} 
where we used the positive $1$-homogeneity of $f^\infty(x,\cdot)$.  Since the function $g_{r,t}$ defined by $g_{r,t}( x ,\zeta,\nu):=\frac1t g(r x , t\zeta,\nu)$ satisfies $(g3)$ with the constant $c_2$ independent of $r$ and $t$, and by $(g4)$
$$
\frac{1}{tr^{n-1}}m^{f^\infty,g}(u_{0,t\zeta,\nu},Q^\nu_r)\le\frac{1}{tr^{n-1}} E^{f^\infty,g}(u_{0,t\zeta,\nu},Q^\nu_r)=
\frac{1}{tr^{n-1}}\int_{\Pi_0^\nu\cap Q^\nu_r}g( x ,t\zeta,\nu)\,d\mathcal H^{n-1}\le c_3|\zeta|=c_3,
$$
from  Remark \ref{BV estimate} and \eqref{abj} we deduce that  there exists a constant $C$ such that $|Dw_r|(Q^\nu)\le C$, for every $r>0$, $t>0$, and $\eta\in (0,1)$. In addition, since $ w_r$ coincides with $u_{0,\zeta,\nu}$ near $\partial Q^\nu$, we can apply Poincar\'e's inequality and from the bound on its total variation we deduce that $w_r$ is  bounded in $BV(Q^\nu,\R^m)$, uniformly with respect to $r$, by a constant that we still denote with $C$; namely,
$\|w_r\|_{L^1(Q^\nu,\R^m)}+|Dw_r|(Q^\nu)\le C$. 

By Lemma \ref{truncation} for every $\eta\in(0,1)$ there exists a constant $M_\eta>0$, depending on $C$ but not on $t>0$, $r>0$, $\zeta\in \Sph^{m-1}$, and $\nu\in \Sph^{n-1}$, such that for every $r>0$ there exists a function $\tilde w_r\in SBV(Q^\nu,\R^m) \cap  L^\infty(Q^\nu,\R^m)$ with the following properties: $\tilde w_r=u_{0,\zeta,\nu}$ near $\partial Q^\nu$,  $ \|\tilde w_r\|_{L^\infty(Q^\nu,\R^m)}\le M_\eta$, and
\begin{eqnarray}
 &&\ds\int_{Q^\nu}f^\infty(r x ,\nabla \tilde w_r)\dx +\frac1t\int_{S_{\tilde w_r}\cap Q^\nu}g(r x ,t[\tilde w_r],\nu_{\tilde w_r})\,d\mathcal H^{n-1}
 \nonumber
 \\
&&\ds\le \int_{Q^\nu}f^\infty(r x ,\nabla w_r)\dx +\frac1t\int_{S_{w_r}\cap Q^\nu}g(r x ,t[w_r],\nu_{w_r})\,d\mathcal H^{n-1}+\eta,
 \nonumber
\end{eqnarray}
where we used the fact that $f^\infty\in\F$ and  $g_{r,t}\in\G$.  
By \eqref{abj} this implies that
\begin{equation}\label{ade}
\int_{Q^\nu}f^\infty(r x ,\nabla \tilde w_r)\dx +\frac1t\int_{S_{\tilde w_r}\cap Q^\nu}g(r x ,t[\tilde w_r],\nu_{\tilde w_r})\,d\mathcal H^{n-1}\le \frac1t \frac{1}{r^{n-1}}m^{f^\infty,g}(u_{0,t\zeta,\nu},Q^\nu_r)+2\eta.
\end{equation}

Let $\tilde v_r\in SBV(Q^\nu_r,\R^m)\cap L^\infty( Q^\nu_r,\R^m)$ be the function defined by $\tilde v_r( x ):=\tilde w_r(\frac{ x}{r})$ for every $ x \in Q^\nu_r$. Then $\tilde v_r=u_{0,\zeta,\nu}$ near $\partial Q^\nu_r$, $ \|\tilde v_r\|_{L^\infty(Q^\nu_r,\R^m)}\le M_{\eta}$, and, by a change of variables,
\begin{equation}\label{adf}
 \int_{Q^\nu_r}f^\infty( x ,\nabla \tilde v_r)\dx +\frac1t\int_{S_{\tilde v_r}\cap Q^\nu_r}g( x ,t[\tilde v_r],\nu_{\tilde v_r})\,d\mathcal H^{n-1}\le \frac1t m^{f^\infty,g}(u_{0,t\zeta,\nu},Q^\nu_r)+2\eta r^{n-1},
\end{equation} 
where we used \eqref{ade} and the positive $1$-homogeneity of $f^\infty(x,\cdot)$. 
Since $ \|\tilde v_r\|_{L^\infty(Q^\nu_r,\R^m)}\le M_{\eta}$, by \eqref{acg} we have
\begin{eqnarray}
&&\Big(1-\frac1{c_2}\lambda( 2 sM_{\eta})\Big)\Big(\int_{Q^\nu_r}f^\infty( x ,\nabla \tilde v_r)\dx +\frac1s\int_{S_{\tilde v_r}\cap Q^\nu_r}g( x ,s[\tilde v_r],\nu_{\tilde v_r})\,d\mathcal H^{n-1}\Big)
\nonumber
\\
&&\le \Big(1+\frac1{c_2}\lambda(2 tM_{\eta})\Big)\Big(\int_{Q^\nu_r}f^\infty( x ,\nabla \tilde v_r)\dx +\frac1t\int_{S_{\tilde v_r}\cap Q^\nu_r}g( x ,t[\tilde v_r],\nu_{\tilde v_r})\,d\mathcal H^{n-1}\Big).
\label{adg}
 \end{eqnarray}
 
 Assume that
  \begin{equation}\label{adi}
 1-\tfrac1{c_2}\lambda( 2 sM_{\eta})>0.
 \end{equation} 
 Since $s\tilde v_r=u_{0,s\zeta,\nu}$ near $\partial Q^\nu_r$, using the positive $1$-homogeneity of $f^\infty$ we obtain that
 \begin{equation*}
 \frac1s m^{f^\infty,g}(u_{0,s\zeta,\nu},Q^\nu_r)\le \int_{Q^\nu_r}f^\infty( x ,\nabla \tilde v_r)\dx +\frac1s\int_{S_{\tilde v_r}\cap Q^\nu_r}g( x ,s[\tilde v_r],\nu_{\tilde v_r})\,d\mathcal H^{n-1}.
\end{equation*} 
Hence from \eqref{adf} and \eqref{adg} we have
\begin{equation}\label{adh}
\big(1-\tfrac1{c_2}\lambda( 2  sM_{\eta})\big) \tfrac1s m^{f^\infty,g}(u_{0,s\zeta,\nu},Q^\nu_r)\le
 \big(1+\tfrac1{c_2}\lambda( 2  tM_{\eta})\big)\big( \tfrac1t m^{f^\infty,g}(u_{0,t\zeta,\nu},Q^\nu_r)+2\eta r^{n-1}\big).
\end{equation} 
This inequality holds also when \eqref{adi} is not satisfied, since in that case the left-hand side is nonpositive.

Since $M_{\eta}$ does not depend on $t$, $s$, and $r$, we can divide \eqref{adh} by $r^{n-1}$ and, passing to the limit as
$r\to +\infty$, by \eqref{g-hom} we obtain
$$
\big(1-\tfrac1{c_2}\lambda( 2 sM_{\eta})\big) \tfrac1s g_{\rm hom}(s\zeta,\nu)\le  \big(1+\tfrac1{c_2}\lambda( 2 tM_{\eta})\big)\big(\tfrac1t g_{\rm hom}(t\zeta,\nu)+2\eta\big),
$$
which gives
$$
 \tfrac1s g_{\rm hom}(s\zeta,\nu) - \tfrac1t g_{\rm hom}(t\zeta,\nu)
 \le \tfrac1{c_2}\lambda( 2  sM_{\eta})\tfrac1s g_{\rm hom}(s\zeta,\nu)
 +
\tfrac1{c_2}\lambda( 2 tM_{\eta})\tfrac1t g_{\rm hom}(t\zeta,\nu)
+2\eta \big(1+\tfrac1{c_2}\lambda( 2  tM_{\eta})\big).
$$
Since $g_{\rm hom}$ satisfies $(g4)$ and $|\zeta|=1$, from the previous inequality we deduce that
$$
 \tfrac1s g_{\rm hom}(s\zeta,\nu) - \tfrac1t g_{\rm hom}(t\zeta,\nu)
 \le \tfrac{c_3}{c_2}\lambda( 2  sM_{\eta})
 +
\tfrac{c_3}{c_2}\lambda( 2  tM_{\eta})
+2\eta \big(1+\tfrac1{c_2}\lambda( 2  tM_{\eta})\big).
$$
Exchanging the roles of $s$ and $t$ we obtain
\begin{equation}\label{adl}
| \tfrac1s g_{\rm hom}(s\zeta,\nu) - \tfrac1t g_{\rm hom}(t\zeta,\nu)|
\le \tfrac{c_3}{c_2}\lambda( 2 sM_{\eta})
 +
 \tfrac{c_3}{c_2}\lambda( 2  t M_{\eta})
+2\eta \big(2+ \tfrac1{c_2}\lambda( 2 sM_{\eta}) +\tfrac1{c_2}\lambda( 2  tM_{\eta})\big)
\end{equation}
for every $s>0$, $t>0$, $\zeta\in\Sph^{m-1}$, and $\nu\in\Sph^{n-1}$.

Given $\tau>0$, we fix $\eta>0$ such that $4\eta<\frac\tau5$. Then, using \eqref{acc}, we find $\delta>0$ such that for every $t\in (0,\delta)$ we have $\frac{c_3}{c_2}\lambda( 2  tM_{\eta})<\frac\tau5$ and $2\eta\frac1{c_2}\lambda( 2  tM_{\eta})<\frac\tau5$. From \eqref{adl} we obtain
\begin{equation}\label{adk}
| \tfrac1s g_{\rm hom}(s\zeta,\nu) - \tfrac1t g_{\rm hom}(t\zeta,\nu)|
\le \tau
\end{equation}
for every $s,\, t\in(0,\delta)$, $\zeta\in\Sph^{m-1}$, and $\nu\in\Sph^{n-1}$. This shows that the function $t\mapsto  \tfrac1t g_{\rm hom}(t\zeta,\nu)$ satisfies the Cauchy condition for $t\to0+$, hence the limit
$$
g_{\rm hom, 0}(\zeta,\nu):=\lim_{t\to0+} \tfrac1t g_{\rm hom}(t\zeta,\nu)
$$
exists and is finite.
This limit is uniform with respect to $\zeta\in\Sph^{m-1}$ and $\nu\in\Sph^{n-1}$ thanks to \eqref{adk}.
This concludes the proof of $(g5)$.

Property $(g6)$ for $g_{\rm hom}$ follows from \eqref{g-hom} and from the fact that 
 $u_{0,-\zeta,-\nu}=u_{0,\zeta,\nu}-\zeta$ and $Q^\nu_r=Q^{-\nu}_r$ (see (h) in Section~\ref{Notation}). 
\end{proof}


\section{$\Gamma$-convergence of a subsequence of $(E_\e)$}\label{Sect:abstract}

In this short section we show that, up to a subsequence, the functionals $E_\e$ defined in \eqref{E-det} $\Gamma$-converge to some functional $\widehat E$ as $\e\to 0+$, and study the main properties of this functional.  

\begin{thm}[Properties of the $\Gamma$-limit]\label{thm:Gamma-conv}
Let $f\in\F$, let $g\in\G$, and for $\e>0$ let $E_\e \colon L^{1}_{\rm loc}(\R^n,\R^m)\times \mathscr{A} \longrightarrow [0,+\infty]$ be the functionals defined in \eqref{E-det}. Then, for  every  sequence  of positive numbers converging to zero, there exist a subsequence $(\e_{j})$ and a functional $\widehat E \colon L^{1}_{\rm loc}(\R^n,\R^m)\times \mathscr{A} \longrightarrow [0,+\infty]$
such that for every $A\in\A$ the functionals $E_{\e_{j}}(\cdot,A)$ $\Gamma$-converge to $\widehat E(\cdot,A)$ in $L^1_{\rm loc}(\R^n,\R^m)$, as $j\to +\infty$. Moreover, $\widehat E$ satisfies the following properties:
\begin{itemize}
\item[(a)] {\rm (locality)} $\widehat E$ is local\ie $\widehat E(u, A)=\widehat E(v, A)$ for every $A\in \A$ and every $u$, $v\in L^1_{\rm loc}(\R^n,\R^m)$ such that $u=v$ $\mathcal{L}^n$-a.e.\ in $A$; 
\item[(b)] {\rm (semicontinuity)} for every $A\in \A$ the functional $\widehat E(\cdot, A)$ is lower semicontinuous in $L^1_{\rm loc}(\R^n,\R^m)$;  
\item[(c)]{\rm (bounds)} for every $A\in \A$ and every $u\in L^1_{\rm loc}(\R^n, \R^m)$ we have
 \begin{gather*}
 c_2  |Du|(A)\le \widehat E(u,A)\le  c_3 |Du|(A) + c_4 \mathcal L^n(A) \quad\hbox{if }u_{|_A}\in BV(A,\R^m),
\\
 \widehat E(u,A)=+\infty \quad\hbox{otherwise};
\end{gather*}
\item[(d)]{\rm (measure property)} for every $u\in L^1_{\rm loc}(\R^n, \R^m)$ the set function $\widehat E(u,\cdot)$ is the restriction to $\A$ of a Borel measure defined  on $\B^n$, which we still denote by $\widehat E(u,\cdot)$; 
\item[(e)]{\rm (translation invariance in $u$)} for every $A\in \A$ and every $u\in L^1_{\rm loc}(\R^n, \R^m)$ we have 
$$
\widehat E(u+s,A)= \widehat E(u,A)\quad \text{for every}\; s\in \R^m.
$$
\end{itemize}
\end{thm}

\begin{proof}  Given an infinitesimal sequence 
 $(\e_j)$   of positive numbers,  let $\widehat E'$, $\widehat E''\colon L^{1}_\textrm{loc}(\R^n,\R^m)\times \mathscr{A}\longrightarrow [0,+\infty]$ be the functionals defined as 
$$
\widehat E'(\cdot,A) :=\Gamma\hbox{-}\liminf_{j\to +\infty} E_{\e_j}(\cdot,A) \quad\hbox{and}\quad \widehat E''(\cdot,A):=\Gamma\hbox{-}\limsup_{j\to +\infty} E_{\e_j}(\cdot,A).
$$
In view of the  bounds $(f3)$,  $(f4)$, $(g3)$, and $(g4)$  satisfied by $f$ and $g$, it immediately follows that 
\begin{eqnarray}
&  c_2 |Du|(A)\le \widehat E'(u,A)\le \widehat E''(u,A)\le   c_3 |Du|(A) + c_4 \mathcal L^n(A) \qquad\hbox{if  }u|_{A}\in BV(A,\R^m),
\label{4.14bis}
\\
&\widehat E'(u,A)= \widehat E''(u,A)=+\infty\qquad\hbox{if  }u|_{A}\notin BV(A,\R^m).
\label{4.14ter}
\end{eqnarray}

By the definition of $E_{\e_j}$ and the general properties of $\Gamma$-convergence, we can also deduce that the functionals $\widehat E'$ and $\widehat E''$ are local  \cite[Proposition 16.15]{DM93}, lower semicontinuous  (in $u$)  \cite[Proposition 6.8]{DM93}, and increasing  (in $A$)  \cite[Proposition 6.7]{DM93}. Moreover $\widehat E'$ is superadditive  (in $A$)  \cite[Proposition 16.12]{DM93}. 
Since it is not obvious that $\widehat E'$ and $\widehat E''$ are inner regular (in $A$), at this stage of the proof we  consider their inner regular envelopes\ie the functionals $\widehat E'_-, \widehat E''_- \colon L^{1}_\textrm{loc}(\R^n,\R^m)\times \mathscr{A}\longrightarrow [0,+\infty]$ defined as 
$$
\widehat E'_-(u,A) := \sup_{\substack{A'\subset\subset A\\A'\in \mathscr{A}}} \widehat E'(u,A')\quad\hbox{and}\quad \widehat E''_-(u,A) := \sup_{\substack{A'\subset\subset A\\A'\in \mathscr{A} }} \widehat E''(u,A').
$$
Also $\widehat E'_-$ and $\widehat E''_-$ are increasing, lower semicontinuous \cite[Remark 15.10]{DM93}, and local \cite[Remark 15.25]{DM93}. 
Moreover, by \cite[Theorem 16.9]{DM93} we can find a subsequence of $(\e_j)$ (not relabelled) 
such that the corresponding functionals $\widehat E'$ and $\widehat E''$ satisfy
\begin{equation}\label{Def-whE}
\widehat E'_-=\widehat E''_-=:\widehat E.
\end{equation}
The functional $\widehat E$ defined in \eqref{Def-whE} is inner regular \cite[Remark 15.10]{DM93} and superadditive \cite[Proposition~16.12]{DM93}. 

By virtue of \cite[Proposition 3.1]{BDfV} applied with $p=1$ we can immediately deduce that the functionals $E_\e$ satisfy the so-called fundamental estimate uniformly in $\e$. Therefore \cite[ Proposition 18.4]{DM93} yields the subadditivity of $\widehat E(u,\cdot)$. Therefore, invoking the measure-property criterion of De Giorgi and Letta \cite[Theorem 14.23]{DM93}, we can deduce that, for every $u\in L^1_{\rm loc}(\R^n,\R^m)$,   the set function $\widehat E(u,\cdot)$ is the restriction to $\A$ of a Borel measure defined on $\B^n$. 

Moreover \cite[Proposition 18.6]{DM93} and \eqref{4.14bis} imply that $\widehat E(u,A)=\widehat E'(u,A)=\widehat E''(u,A)$ whenever $u|_{A}\in BV(A,\R^m)$. Finally, it follows from \eqref{4.14bis} and \eqref{4.14ter} that $\widehat E(u,A)=\widehat E'(u,A)=\widehat E''(u,A)=+\infty$ if $u{|_A}\notin BV(A,\R^m)$. We then conclude that  $\widehat E=\widehat E'=\widehat E''$ in $L^{1}_\textrm{loc}(\R^n,\R^m)\times \mathscr{A}$, 
hence  that $\widehat E(\cdot,A)$  is the $\Gamma$-limit  of $E_{\e_j}(\cdot,A)$  in $L^{1}_\textrm{loc}(\R^n,\R^m)$ for every $A\in\A$.  

Eventually, the translation invariance in $u$ of $\widehat E(\cdot, A)$ can be easily checked arguing as in \cite[Lemma~3.7]{BDfV}. 
\end{proof}

For later use we need to introduce the following notation.
Let $A\in \A$ and  $w\in BV(A,\R^m)$,  we set
\begin{equation*}
m_{\widehat E}(w; A):=\inf\{\widehat E(u,A) : u\in L^{1}_\textrm{loc}(\R^n,\R^m),\ u|_A\in BV(A,\R^m),\,  u=w\; \text{near}\; \partial A\}.
\end{equation*}

\section{Identification of the volume term}\label{Sect:volume}

In Proposition \ref{p:homo-vol} below we characterise the derivative of $\widehat E$ with respect to the Lebesgue measure $\mathcal L^n$.  In order to prove this result we need the estimate established in the following lemma, whose proof is an immediate consequence of~\eqref{acd}.

\begin{lem}\label{estimate g g0}
Let $g\in \mathcal G$, $A\in\A$, and  $u \in BV(A,\R^m)\cap L^\infty(A,\R^m)$. Then for every $t>0$
$$
\int_{S_{ u}\cap A}\big|g_0(x,[{ u}],\nu_{ u})-\tfrac{1}{t}g(x,t [{ u}],\nu_{ u})\big|\,d\mathcal H^{n-1}
\leq \lambda(t\|[{u}]\|_{L^\infty(S_{ u}\cap A,\R^m)}) \int_{S_{ u}\cap A}|[ u]| \,d\mathcal H^{n-1},
$$
where $\lambda$ is the function defined in \eqref{acb}. 
\end{lem}

\begin{prop}[Homogenised volume integrand]\label{p:homo-vol} Let $f$, $g$, $E_\e$, $(\e_j)$, and 
 $\widehat E$ be as in Theorem \ref{thm:Gamma-conv}. Assume that (a) of Theorem~\ref{T:det-hom} holds, and let $f_{\rm hom}$ be as in \eqref{f-hom}. 
 Then for every $A\in \A$ and every $u\in L^1_{\rm loc}(\R^n,\R^m)$, with $ u|_A\in BV(A,\R^m)$, we have that 
\begin{equation*}
\frac{d\widehat E(u,\cdot)}{d \mathcal L^n}(x)=f_{\rm hom}\big(\nabla u(x)\big)\quad \hbox{for }\mathcal L^n\hbox{-a.e.\ }x\in A.
\end{equation*}    
\end{prop}

\begin{proof} Let us fix $A\in \A$ and $u\in L^1_{\rm loc}(\R^n,\R^m)$, with $ u|_A\in BV(A,\R^m)$.
We divide the proof into two steps.

\medskip

\textit{Step 1:}  We claim that
\begin{equation}\label{aqc}
\frac{d\widehat E(u,\cdot)}{d \mathcal L^n}(x) \leq f_{\rm hom}\big(\nabla u(x)\big)\quad \hbox{for }\mathcal L^n\hbox{-a.e.\ }x\in A.
\end{equation}

By (a)-(e) of Theorem \ref{thm:Gamma-conv} and by \cite[Lemmas 3.1 and 3.5]{BFM}, arguing as in the proof of (3.16) in  \cite[Theorem 3.7]{BFM}, for $\mathcal L^n$-a.e. $x\in A$ we have 
\begin{equation}\label{rep-lebesgue-BFM}
\frac{d\widehat E(u,\cdot)}{d \mathcal L^n}(x)=\lim_{\rho \to 0+} \frac{m_{\widehat E}(\ell_{\xi(x)}, Q_\rho(x))}{\rho^n},
\end{equation}
where $\xi(x):=\nabla u(x)$. Fix $x\in A$ such that \eqref{rep-lebesgue-BFM} holds and  set $\xi:=\xi(x)=\nabla u(x)$. 

For every $\rho>0$ we have 
\begin{equation}\label{ass-f} 
f_{\rm hom}(\xi)=\lim_{r\to +\infty} \frac{m^{f,g_0}(\ell_{\xi},Q_r(\frac{r x}\rho))}{r^{n}},
\end{equation}
 since  the above identity directly follows from \eqref{f-hom} by replacing $ x  $ with $\frac{ x }\rho$.

Let us fix $\eta\in(0,1)$.  By \eqref{m-phipsi} for every $\rho>0$ and $r>0$ there exists $ v^\rho_r\in SBV(Q_r( \tfrac{ r  x}\rho),\R^m)$, with  $ v^\rho_r=\ell_\xi$ near $\partial Q_r( \tfrac{ r  x}\rho)$, such that  
\begin{equation}\label{quasi-min-f}
 E^{f,g_0}( v^\rho_r,Q_r( \tfrac{r  x}\rho))
 \leq m^{f,g_0}(\ell_{\xi},Q_r(\tfrac{r  x}\rho)) +\eta r^n \le \int_{Q_r(\frac{r x}\rho)}f(y,\xi)\dy +\eta r^n  \le  (c_3|\xi| + c_4 +1)r^n,
\end{equation}
where we used that $g_0(\cdot,0,\cdot)=0$ and $(f4)$. We extend $ v^\rho_r$ to $\R^n$ by setting
$ v^\rho_r(y)=\ell_\xi(y)$ for every $y\in \R^n\setminus Q_r(\tfrac{r  x}\rho)$.

 For every $y\in\R^n$  let $ w^\rho_r(y):=\frac{1}{r} v^\rho_r(\frac{rx}{\rho}+ry)-\frac1\rho\ell_\xi(x)$. Clearly $ w^\rho_r\in SBV_{ \rm loc}(\R^n,\R^m)$ and $ w^\rho_r =\ell_\xi$  near $\partial Q$ and in $\R^n\setminus Q$, where $Q:=Q_1(0)$.
Moreover, by a change of variables we obtain
\begin{equation}\label{scaling-f}
\int_{Q}f\big(\tfrac{rx}{\rho}+ry,\nabla w^\rho_r(y)\big)\dy +\int_{S_{ w^\rho_r}\cap Q} g_0\big(\tfrac{rx}{\rho}+ry,[ w^\rho_r](y),\nu_{ w^\rho_r}(y)\big)\,d\mathcal H^{n-1}(y)= \frac1{r^n}E^{f,g_0}(v^\rho_r,Q_r( \tfrac{r  x}\rho)),
\end{equation}
 where we used the $1$-homogeneity of $g_0$ in the second variable. 
By the lower bounds  $(f3)$ and $(g3)$, from \eqref{quasi-min-f} and \eqref{scaling-f} we deduce that  there exists a constant $K$, depending on $|\xi|$, such that $|Dw^\rho_r|(Q)\le K$ for every $\rho>0$ and $r>0$. In addition, since $ w^\rho_r$ coincides with $\ell_\xi$ in $\R^n\setminus Q$, we can apply Poincar\'e's inequality and from the bound on its total variation we deduce that the sequence $(w^\rho_r)$ is bounded in $BV_{\rm loc}(\R^n,\R^m)$. In particular it is  bounded in $BV(Q,\R^m)$, uniformly with respect to $\rho$ and $r$, by a constant that we still denote with $K$.

By  Lemma \ref{truncation} and by \eqref{quasi-min-f} and \eqref{scaling-f} for every $\eta\in(0,1)$ there exists a constant $M_{\eta}$, depending also on $|\xi|$ and $K$, such that for every $\rho>0$ and $r>0$ there exists $\tilde w^\rho_r\in SBV(Q,\R^m)\cap L^\infty(Q,\R^m)$ with the following properties: $\tilde w^\rho_r=\ell_\xi$ near $\partial Q$,  $\|\tilde w^\rho_r\|_{L^\infty(Q,\R^m)}\le M_{\eta}$, and
$$
\int_{Q}f\big(\tfrac{rx}{\rho}+ry,\nabla \tilde w^\rho_r(y)\big)\dy +\int_{S_{\tilde w^\rho_r}\cap Q} g_0\big(\tfrac{rx}{\rho}+ry,[\tilde w^\rho_r](y),\nu_{\tilde w^\rho_r}(y)\big)\,d\mathcal H^{n-1}(y)
\le \frac{m^{f,g_0}(\ell_{\xi},Q_r(\tfrac{r x}\rho))}{r^{n}}+2\eta.
$$

Let $\tilde v^\rho_r\in SBV_{\rm loc}(\R^n,\R^m)$ be defined by 
$\tilde v^\rho_r(y):=r\tilde w^\rho_r(\frac{y}{r}-\frac{x}{\rho})+\frac{r}\rho\ell_\xi(x)$. Then $\tilde v^\rho_r=\ell_\xi$ near $\partial Q_r( \tfrac{ r  x}\rho)$ and, by a change of variables, 
\begin{gather}
 \|[\tilde v^\rho_r]\|_{L^\infty(S_{\tilde v^\rho_r}\cap Q_r(\frac{rx}\rho),\R^m)}\leq 2 M_{\eta} r,
 \label{v bounded r}
 \\
 E^{f,g_0}(\tilde v^\rho_r,Q_r(\tfrac{r  x}\rho)) \leq m^{f,g_0}(\ell_{\xi},Q_r(\tfrac{r  x}\rho)) +2\eta\, r^n.
  \label{quasi-min-f-bis}
\end{gather}
Moreover, by combining \eqref{quasi-min-f} and \eqref{quasi-min-f-bis} with the lower bound $(g3)$ we immediately deduce the existence of a constant $C>0$, depending on $|\xi|$, such that
\begin{equation}\label{f:unif-bd}
\frac{1}{r^n}\int_{S_{\tilde v^\rho_r}\cap Q_r(\frac{r x}{\rho})}|[\tilde v^\rho_r]| \,d\mathcal H^{n-1}\leq C.
\end{equation}

By Lemma \ref{estimate g g0}, applied  with $t=\rho/r$, using \eqref{v bounded r} and \eqref{f:unif-bd} we obtain 
\begin{equation*} 
\frac{1}{r^n}\int_{S_{\tilde v^\rho_r}\cap Q_r(\tfrac{r x}\rho)}\Big|g_0(y,[\tilde v^\rho_r],\nu_{\tilde v^\rho_r})-\tfrac{r}{\rho}g( y,\tfrac{\rho}{r} [\tilde v^\rho_r],\nu_{\tilde v^\rho_r})\Big|\,d\mathcal H^{n-1}
  \leq C \lambda(2\rho M_{\eta}).
\end{equation*}
This inequality, together with  \eqref{ass-f} and \eqref{quasi-min-f-bis}, gives
$$
\limsup_{r \to +\infty}\frac{1}{r^n}\Big(\int_{Q_r( \tfrac{r  x}\rho)}f(y,\nabla \tilde v^\rho_r)\dy +\int_{S_{\tilde v^\rho_r}\cap Q_r( \tfrac{r  x}\rho)}\tfrac{r}{\rho}g(y,\tfrac{\rho}{r} [\tilde v^\rho_r],\nu_{\tilde v^\rho_r})\,d\mathcal H^{n-1}\Big)
\le
f_{\rm hom}(\xi) + 2  \eta  +  C \lambda(2\rho M_{\eta}).
$$
 Given $\e>0$ and $\rho>0$, we choose $r=\rho/\e$ and  for every $y\in\R^n$ we define  $ u^\rho_\e(y):= \e\,  \tilde v^\rho_r (\frac{y}{\e})=\frac{\rho}{r} \tilde v^\rho_r (\frac{r y}{\rho})$. Then $ u^\rho_\e\in SBV_{\rm loc}(\R^n,\R^m)$, $ u^\rho_\e =\ell_\xi$ near $\partial Q_\rho( x  )$ and in $\R^n\setminus Q_\rho( x  )$. By a change of variables, from the previous inequality we get that for every $\rho>0$
\begin{equation} \label{f:stima-1}
\limsup_{\e \to 0+} \frac{ E_\e( u^\rho_\e,Q_\rho(x))}{\rho^n}
 \le 
f_{\rm hom}(\xi) + 2  \eta  +  C \lambda(2\rho M_{\eta}).
\end{equation}

Since the functions $ u^\rho_\e$ coincide with $\ell_\xi$ in $Q_{(1+\eta)\rho}(x)\setminus Q_{(1-\delta_\e)\rho}(x)$ for some $\delta_\e\in(0,1)$, by $(f4)$ we have
$E_\e( u^\rho_\e,Q_{(1+\eta)\rho}( x ))\le E_\e( u^\rho_\e,Q_\rho( x )) + (c_3|\xi|+c_4) 2^n \rho^n\eta$,  
which, together with \eqref{f:stima-1}, gives
\begin{equation}\label{stima 1+sigma}
\limsup_{\e \to 0+} 
\frac{ E_\e( u^\rho_\e,Q_{(1+\eta)\rho}(x))}{\rho^n} \le f_{\rm hom}(\xi) +   C \lambda(2\rho M_{\eta}) +\tilde K \eta,
\end{equation}
 where $\tilde K:=2+ (c_3|\xi|+c_4) 2^n$. 
Since $ u^\rho_\e$ coincides with $\ell_\xi$ in $\R^n\setminus Q_\rho(x)$, using 
Poincar\'e's inequality and the lower bounds  $(f3)$ and $(g3)$ we deduce from \eqref{stima 1+sigma} that for every $\rho>0$  there exists $\e(\rho)>0$ such that  the sequence  
$(u^\rho_\e)$ is bounded in $BV_{\rm loc}(\R^n,\R^m)$. In particular it is  bounded in $BV(Q_\rho(x),\R^m)$ uniformly with respect to $\e\in(0,\e(\rho))$. Note that the bound on the $L^1$-norm can be obtained by e.g. \cite[Theorem 3.47]{AFP}.
Then there exists a subsequence, not relabelled, of the sequence $(\e_j)$ considered in Theorem \ref{thm:Gamma-conv}, such that 
 $( u^\rho_{\e_j})$ converges  in $L^1_{\rm loc}(\R^n,\R^m)$  to  some  $ u^\rho \in BV_{\rm loc}(\R^n,\R^m)$  with  $u^\rho=\ell_\xi$  in   $Q_{(1+\eta)\rho}( x  )\setminus Q_{\rho}( x  )$. As a consequence  of the
$\Gamma$-convergence of $E_{\e_j}(\cdot, Q_{(1+\eta)\rho}(x))$ to  $\widehat E(\cdot, Q_{(1+\eta)\rho}(x))$, from \eqref{stima 1+sigma}  we obtain
$$
\frac{m_{\widehat E}(\ell_\xi,Q_{(1+\eta)\rho}( x  ))}{\rho^n}
\le
 \frac{\widehat E( u^\rho,Q_{(1+\eta)\rho}( x  ))}{\rho^n}
 \le
\limsup_{ j \to +\infty}  \frac{E_{\e_j}( u^\rho_{\e_j},Q_{(1+\eta)\rho}( x  ))}{\rho^n}
\le
f_{\rm hom}(\xi)  +   C \lambda(2\rho M_{\eta}) + \tilde K\eta. 
$$
Finally, passing to the limit  as $\rho \to 0+$, from \eqref{acc} and \eqref{rep-lebesgue-BFM} we get
$$
(1+\eta)^n \frac{d\widehat E(u,\cdot)}{d \mathcal L^n}(x) \le f_{\rm hom}(\xi) + \tilde K\eta.
$$
Since $\xi:= \nabla u(x)$, this gives \eqref{aqc} by the arbitrariness of $\eta>0$.

\medskip

\textit{Step 2:} We claim that 
\begin{equation}\label{aqs}
\frac{d\widehat E(u,\cdot)}{d \mathcal L^n}(x) \geq f_{\rm hom}\big(\nabla u(x)\big)
\quad \hbox{for }\mathcal L^n\hbox{-a.e.\ }x\in A.
\end{equation}
We extend $u$ to $\R^n$ by setting $u=0$ on $\R^n\setminus A$. By  $\Gamma$-convergence  there exists $(u_\e)\subset L^1_{\rm loc}(\R^n,\R^m)$,  with ${u_\e}|_A\in SBV(A,\R^m)$,  such that
\begin{equation}\label{f-hom:rec-seq}
u_\e \to u \quad \text{in} \quad L^1_{\rm loc}(\R^n,\R^m) \quad \text{and}\quad \lim_{\e \to 0+}E_\e(u_\e,A)=\widehat E(u,A),
\end{equation}
along the sequence $(\e_j)$ considered in Theorem \ref{thm:Gamma-conv}. Passing to a further subsequence we have
\begin{equation}\label{leb-00}
\lim_{\e \to 0+}u_\e(x)=u(x)
\end{equation}
 for $\mathcal{L}^n$- a.e.\ $ x  \in A$.
In the rest of the proof $\e$ will always be an element of this sequence. 

By (j) of Section \ref{Notation} and  by \cite[Definition 3.70]{AFP}, for $\mathcal{L}^n$-a.e.\ $ x  \in A$ we have
\begin{gather}
\label{leb-0}
\lim_{\rho \to 0+}\frac{1}{\rho^{n}}|Du|(Q_\rho(x))= |\nabla u(x)|<+\infty, 
\\
\label{leb-1}
\lim_{\rho \to 0+}\frac{1}{\rho^{n+1}}\int_{Q_\rho(x)} |u(y)-u(x) -\nabla u(x)(y-x)|\dx =0, 
\\
\label{leb-2}
\lim_{\rho \to 0+}\frac{\widehat E(u,Q_{\rho}(x))}{\rho^n}=\frac{d\widehat E(u,\cdot)}{d \mathcal L^n}(x).
\end{gather} 
 Let us fix $x\in A$ such that \eqref{leb-00}-\eqref{leb-2}  hold true.

Recalling \eqref{4.14bis} we have that
$
\widehat E(u, Q_{\rho}(x))\leq c_3|Du|(Q_{\rho}(x)) + c_4 \rho^n
$,
hence by \eqref{leb-0} and \eqref{leb-2} there exists $\rho_0\in(0,1)$  such that
\begin{equation}\label{finitezza}
Q_\rho( x  ) \subset\subset A\quad\hbox{and}\quad\frac{\widehat E(u,Q_\rho(x))}{\rho^n} \leq  c_3|\nabla u(x)| 
+ c_4 + 1
\end{equation}
 for every $\rho\in(0,\rho_0)$. 
Since $\widehat E(u,\cdot)$ is a Radon measure, there exists a set $B\subset  (0,\rho_0)$, with $(0,\rho_0)\setminus B$ at most countable, such that $\widehat E(u,\partial Q_\rho( x  ))=0$ for every $\rho\in B$. Then 
we have
$$
\widehat E(u,A)=\widehat E(u,Q_\rho( x  ))+\widehat E(u,A\setminus \overline Q_\rho( x  ))
$$ 
for every $\rho\in B$. By   $\Gamma$-convergence   we also have
$$
\liminf_{\e \to 0+} E_\e(u_\e,Q_\rho( x  ))\geq  \widehat E(u,Q_\rho( x  ))\quad\hbox{and}\quad
\liminf_{\e \to 0+} E_\e(u_\e,A\setminus \overline Q_\rho( x  ))\geq  \widehat E(u,A\setminus \overline Q_\rho( x  )),
$$
so that by \eqref{f-hom:rec-seq} it follows that  for every  $\rho\in B$ 
\begin{equation}\label{4.20bis}
\lim_{\e \to 0+}E_\e(u_\e,Q_{\rho}( x  ))=\widehat E(u,Q_{\rho}( x  )).
\end{equation}
Note that for every  $\rho\in B$ there exists $\e(\rho)>0$  such that for every $\e\in (0,\e(\rho))$
\begin{equation}\label{f-hom:stima1}
 \frac{E_\e(u_\e,Q_\rho( x  ))}{\rho^n}\leq
\frac{\widehat E(u,Q_\rho( x  ))}{\rho^n} + \rho \le c_3|\nabla u(x)| 
+ c_4 + 2,
\end{equation}
where in the last inequality we used \eqref{finitezza}. 

The rest of this proof is devoted to  modifying $u_\e$ in order to construct  a competitor for the minimisation  problems  appearing in  \eqref{f-hom}, which defines  $f_{\rm hom}$. To this end, for every $\rho\in B$ and $\e>0$ we consider the blow-up functions defined for $y\in Q:=Q_1(0)$ by 
$$
w_\e^\rho(y):=\frac{u_\e( x  +\rho y)-u_\e( x  )}{\rho} \quad\hbox{and}\quad  w^\rho(y):=\frac{u( x  +\rho y)-u( x  )}{\rho}. 
$$
Then $ w_\e^\rho\in SBV(Q,\R^m)$ and $w^\rho\in BV(Q,\R^m)$. 
Since $u_\e \to u$ in $L^1(Q_\rho( x  ), \R^m)$ by \eqref{f-hom:rec-seq}, using \eqref{leb-00} for every $\rho\in B$ 
we obtain 
\begin{equation}\label{aqo}
 w_\e^\rho\to w^\rho\quad\hbox{in }L^1(Q,\R^m) \hbox{ as }\e\to0+.
\end{equation}
Moreover, from \eqref{leb-1} we can deduce that 
\begin{equation}\label{aqq}
 w^\rho\to \ell_\xi \quad\hbox{in }L^1(Q,\R^m)\hbox{ as }\rho \to 0+,
\end{equation} 
where we set $\xi:=\nabla u(x)$. Therefore, by possibly reducing the values of $\rho_0$ and $\e(\rho)$, we may assume that
\begin{equation}\label{aqj}
\|w_\e^\rho -\ell_\xi \|_{L^1(Q,\R^m)}\le 1,
\end{equation} 
for every $\rho\in B$ and $\e\in(0,\e(\rho))$.
By the  definition of $w_\e^\rho$, a change of variables gives
\begin{equation}\label{aqd}
\frac{E_\e(u_\e,Q_\rho( x  ))}{\rho^n}= E^\rho_\e( w_\e^\rho,Q),
\end{equation}
where $E^\rho_\e$ is the functional corresponding to the integrands 
$ f_\e^\rho(y,\xi):=f(\tfrac{ x  +\rho y}{\e},\xi)$ and $g_\e^\rho(y,\zeta,\nu):=\tfrac{1}{\rho}g(\tfrac{ x  +\rho y}{\e},\rho \zeta,\nu)$\ie
$$
E^\rho_\e(w,Q):= \int_{Q}f(\tfrac{ x  +\rho y}{\e},\nabla w(y))\dy+\int_{S_{w}\cap Q} \tfrac{1}{\rho}g(\tfrac{ x  +\rho y}{\e},\rho [w](y),\nu_{w}(y))d\mathcal H^{n-1}(y)
$$
for every $w\in SBV(Q,\R^m)$. Note that $f_\e^\rho$ satisfies $(f3)$ and $(f4)$, while $g_\e^\rho$ satisfies $(g3)$ and $(g4)$.

We now modify $w_\e^\rho$ in a way such that it attains the linear boundary datum $\ell_\xi$ near $\partial Q$. To this end we apply the Fundamental Estimate \cite[Proposition 3.1]{BDfV} to the functionals $E^\rho_\e$. Thus for $\eta \in (0,\frac12)$ fixed there exist a constant $L_\eta>0$ with the following property: for every $\rho\in B$ and $\e\in(0,\e(\rho))$ there exists a cut-off function $\varphi^\rho_\e \in C^\infty_c(Q)$, with $0\leq\varphi^\rho_\e \leq 1$ in $Q$, ${\rm supp}(\varphi^\rho_\e)\subset Q_{1-\eta}:=Q_{1-\eta}(0)$, and $\varphi^\rho_\e=1$ in $Q_{1-2\eta}:=Q_{1-2\eta}(0)$, such that, setting
$\hat w^\rho_\e:=\varphi^\rho_\e w^\rho_\e+ (1-\varphi^\rho_\e)\ell_{\xi}$, we have 
\begin{equation}\label{aqh}
E^\rho_\e (\hat w^\rho_\e,Q)\leq (1+\eta)\big(E^\rho_\e ( w_\e^\rho,Q)+ E^\rho_\e (\ell_\xi,Q\setminus \overline Q_{1-2\eta})\big)+ L_\eta\| w_\e^\rho-\ell_\xi\|_{L^1(Q,\R^m)}.
\end{equation}
We note that $\hat w^\rho_\e=\ell_\xi$ in $Q\setminus Q_{1-\eta}$, as desired. Moreover in view of $(f4)$ we have 
\begin{equation}\label{aqn}
E^\rho_\e (\ell_\xi,Q\setminus \overline Q_{1-2\eta}))\leq (c_3|\xi| +c_4) \mathcal L^n(Q\setminus  Q_{1-2\eta}))\leq (c_3|\xi| +c_4) 2n \eta. 
\end{equation}
From \eqref{f-hom:stima1}, \eqref{aqj}, \eqref{aqd}, \eqref{aqh} and \eqref{aqn} we obtain
\begin{equation}\label{aqi}
E^\rho_\e (\hat w^\rho_\e,Q)\leq \tfrac32 (c_3 |\xi|
+ c_4 + 2 + (c_3|\xi| +c_4) n)+ L_\eta
\end{equation}
 for every $\rho\in B$ and $\e\in(0,\e(\rho))$. 
 By the lower bounds $(f3)$ and $(g3)$, we deduce that the total variation of $\hat w_\e^\rho$ is bounded uniformly with respect to $\rho\in B$ and $\e\in(0,\e(\rho))$. Note moreover that also the $L^1$-norm of $\hat w_\e^\rho$ is bounded uniformly, by \eqref{aqj}. In conclusion, the sequence  $(\hat w_\e^\rho)$ is bounded in $BV(Q,\R^m)$ uniformly with respect to $\rho\in B$ and $\e\in(0,\e(\rho))$.

 By Lemma \ref{truncation} there exist a constant $M_{\eta}>0$ with the following property:  for every $\rho\in B$ and $\e\in(0,\e(\rho))$ there exists $\tilde w^\rho_\e\in SBV(Q,\R^m)\cap L^\infty(Q,\R^m)$, with $\tilde w^\rho_\e=\ell_\xi$ near $\partial Q$,  such that 
 \begin{equation}\label{aqk}
 \|\tilde w^\rho_\e\|_{L^\infty(Q,\R^m)}\le M_{\eta}, \quad
 \|\tilde w^\rho_\e-\ell_\xi\|_{L^1(Q,\R^m)}\le  \|\hat w^\rho_\e-\ell_\xi\|_{L^1(Q,\R^m)}, \ \hbox{and }
 E^\rho_\e(\tilde w^\rho_\e,Q)\le E^\rho_\e(\hat w^\rho_\e,Q)+\eta.
 \end{equation}
 We now set $r:=\frac\rho\e$  and
$ v_\e^\rho( y):=
 r \tilde w_\e^\rho (\tfrac{y}{r}- \tfrac{x}{\rho})+\tfrac{r}{\rho}\ell_\xi( x)$;
clearly $v_\e^\rho\in SBV(Q_r(\frac{r x  }{\rho}),\R^m)$, $ v_\e^\rho = \ell_\xi$ near $\partial Q_r(\frac{r x  }{\rho})$, and, by a change of variables
\begin{gather}
 \|[v^\rho_\e]\|_{L^\infty(S_{v^\rho_\e}\cap Q_r(\frac{rx}\rho),\R^m)}\leq 2M_{\eta} r,
 \label{aqf}
 \\
 \frac{1}{r^n}\Big(\int_{Q_r( \tfrac{r  x}\rho)}f(y,\nabla v^\rho_{\e})\dy 
 +\int_{S_{v^\rho_{\e}}\cap Q_r( \tfrac{r  x}\rho)}\tfrac{r}{\rho}g(y,\tfrac{\rho}{r} [v^\rho_{\e}],\nu_{v^\rho_{\e}})\,d\mathcal H^{n-1}\Big)
=E^\rho_\e(\tilde w^\rho_\e,Q).
  \label{aqg}
\end{gather}
Moreover, by combining  \eqref{aqi}, \eqref{aqk},  and  \eqref{aqg} with the lower bound $(g3)$ we immediately deduce the existence of a constant $C>0$, depending on $|\xi|$, such that
\begin{equation}\label{aqm}
\frac{1}{r^n}\int_{S_{v^\rho_\e}\cap Q_r(\frac{r x}{\rho})}|[v^\rho_\e]| \,d\mathcal H^{n-1}\leq C.
\end{equation}

By Lemma \ref{estimate g g0}, applied  with $t=\rho/r$, using \eqref{aqf} and \eqref{aqm} we  deduce that 
\begin{equation*} 
\frac{1}{r^n}\int_{S_{v^\rho_\e}\cap Q_r(\tfrac{r x}\rho)}\Big|g_0(y,[v^\rho_\e],\nu_{v^\rho_\e})-\tfrac{r}{\rho}g(y,\tfrac{\rho}{r} [v^\rho_\e],\nu_{v^\rho_\e})\Big|\,d\mathcal H^{n-1}
  \leq C \lambda(2\rho M_{\eta})
\end{equation*}
for every $\rho\in B$ and $\e\in(0,\e(\rho))$.
 From this inequality and from \eqref{aqd}, \eqref{aqh}, \eqref{aqn}, \eqref{aqk}, and \eqref{aqg} we obtain
 $$
 \frac{E^{f,g_0}(v^\rho_\e,Q_r(\tfrac{r x}\rho))}{r^n} \le (1+\eta) \frac{E_\e(u_\e,Q_\rho(x))}{\rho^n}
 +K\eta +C\lambda(2\rho M_{\eta}) + L_{\eta}\| w_\e^\rho-\ell_\xi\|_{L^1(Q,\R^m)}, 
 $$
where $K:=(c_3|\xi| +c_4)3n +1$. Recalling that $ v_\e^\rho = \ell_\xi$ near $\partial Q_r(\frac{r x  }{\rho})$, we get 
$$
 \frac{m^{f,g_0}(\ell_\xi,Q_r(\tfrac{r x}\rho))}{r^n} \le (1+\eta) \frac{E_\e(u_\e,Q_\rho(x))}{\rho^n}
 +K\eta +C\lambda(2\rho M_{\eta}) + L_{\eta}\| w_\e^\rho-\ell_\xi\|_{L^1(Q,\R^m)}.
 $$
 Since $r=\frac\rho\e$, by \eqref{f-hom} with $x$ replaced by $\frac{x}\rho$, the left-hand side of the previous inequality converges to $ f_{\rm hom}(\xi)$ as $\e\to0+$. By \eqref{4.20bis} and \eqref{aqo} we can pass to the limit in the right-hand side as $\e\to0+$ and we obtain
 $$
 f_{\rm hom}(\xi) \le (1+\eta) \frac{\widehat E(u,Q_\rho(x))}{\rho^n}
 +K\eta + C\lambda(2\rho M_{\eta}) + L_{\eta}\|w^\rho-\ell_\xi\|_{L^1(Q,\R^m)}.
 $$
 By \eqref{acc}, \eqref{leb-2}, and \eqref{aqq}, passing to the limit as $\rho\to0+$ we get
 $$
 f_{\rm hom}(\xi) \le (1+\eta) \frac{d\widehat E(u,\cdot)}{d \mathcal L^n}(x)
 +K\eta.
 $$
 Since $\xi= \nabla u(x)$, this inequality gives \eqref{aqs} by the arbitrariness of $\eta>0$.
\end{proof}

\section{Identification of the surface term}\label{Sect:surface}

In Proposition \ref{p:homo-sur} below we characterise the derivative of $\widehat E(u,\cdot)$ with respect to the measure $\mathcal H^{n-1}\LLL{S_u}$ for a given $BV$-function $u$.  In order to prove this result we need the estimate established in the following lemma.

\begin{lem}\label{estimate f finfty}
Let $f\in \mathcal F$, $A\in\A$, $v\in BV(A,\R^m)$. Then for every $t>0$
\begin{equation*}
\int_A|f^\infty( x ,\nabla v)-\tfrac{1}{t}f( x ,t\nabla v)|\dx\le 
 \frac1t c_5 (1+ c_4^{1-\alpha})\mathcal{L}^n(A)+ \frac{1}{t^\alpha}c_5c_3^{1-\alpha} (\mathcal{L}^n(A))^\alpha \|\nabla v\|^{1-\alpha}_{L^1(A,\R^m)}.
\end{equation*}
\end{lem}

 \begin{proof}
 Let $t>0$. By $(f4)$  and \eqref{abf}, using  H\"older's inequality, we  obtain  that 
 \begin{align*}
 &\int_A|f^\infty( x ,\nabla v)-\tfrac{1}{t}f( x ,t\nabla v)|\dx\le  
 \frac{c_5}t \mathcal{L}^n(A)+ \frac{c_5}{t} \int_A f( x ,t\nabla v)^{1-\alpha} \dx\\
 &\le  \frac{c_5}t \mathcal{L}^n(A)+ \frac{c_5}{t} (\mathcal{L}^n(A))^\alpha\Big(\int_A f( x ,t\nabla v)\dy\Big)^{1-\alpha} \\
 &\le  \frac1t c_5 (1+ c_4^{1-\alpha})\mathcal{L}^n(A)+ \frac{1}{t^\alpha}c_5c_3^{1-\alpha} (\mathcal{L}^n(A))^\alpha \|\nabla v\|^{1-\alpha}_{L^1(A,\R^m)}.
 \end{align*}

 This concludes the proof. \end{proof}

\begin{prop}[Homogenised surface integrand]\label{p:homo-sur} Let $f$, $g$, $E_\e$, $(\e_j)$, and 
 $\widehat E$ be as in Theorem \ref{thm:Gamma-conv}. Assume that (b) of Theorem~\ref{T:det-hom} holds, and let $g_{\rm hom}$ be as in \eqref{g-hom}. 
 Then for every $A\in \A$ and every $u\in L^1_{\rm loc}(\R^n,\R^m)$, with $ u|_A\in BV(A,\R^m)$, we have that 
\begin{equation*}
\frac{d\widehat E(u,\cdot)}{d\mathcal H^{n-1}\LLL{S_u}}(x)=g_{\rm hom}\big([u](x),\nu_{u}(x)\big)\quad \hbox{for }\mathcal H^{n-1}\hbox{-a.e.\ }x\in S_u \cap A.
\end{equation*}    
\end{prop}

\begin{proof} Let us fix $A\in \A$ and $u\in L^1_{\rm loc}(\R^n,\R^m)$, with $ u|_A\in BV(A,\R^m)$.
We divide the proof into two steps. 

\medskip

\textit{Step 1:} We claim that
\begin{equation}\label{aqcc}
\frac{d \widehat E(u,\cdot)}{d\mathcal H^{n-1}\LLL{S_u}}(x) \leq g_{\rm hom}([u](x),\nu_{u}(x)) \quad \hbox{for }\mathcal H^{n-1}\hbox{-a.e.\ }x\in S_u \cap A.
\end{equation}

By (a)-(e) of Theorem \ref{thm:Gamma-conv} and by \cite[Lemmas 3.1 and 3.5]{BFM}, arguing as in the proof of (3.17) in  \cite[Theorem 3.7]{BFM}, for $\mathcal H^{n-1}$-a.e. $x\in S_u \cap A$ we have 
\begin{equation}\label{rep-hausdorff-BFM}
\frac{d \widehat E(u,\cdot)}{d\mathcal H^{n-1}\LLL{S_u}}(x)=\lim_{\rho \to 0+} \frac{m_{\widehat E}(u_{x,[u](x),\nu_u(x)}, Q^{\nu_u(x)}_\rho(x))}{\rho^{n-1}}.
\end{equation}
Fix $x\in S_u \cap A$ such that \eqref{rep-hausdorff-BFM} holds and set $\zeta:=[u](x)$ and $\nu:=\nu_u(x)$. 

For every $\rho>0$ we have 
\begin{equation}\label{ass-g} 
g_{\rm hom}(\zeta,\nu)=\lim_{r\to +\infty} \frac{ m^{f^\infty,g}\big(u_{\frac{r  x}\rho,\zeta,\nu},Q^\nu_r(\tfrac{r x  }{\rho})\big)}{r^{n-1}},
\end{equation}
since the above identity directly follows from \eqref{g-hom} by replacing $ x  $ with $ \frac{x}{\rho}$. 

Let us fix $\eta\in (0,1)$. By \eqref{m-phipsi} for every $\rho\in (0,1)$ and $r>0$ there exists $ v^\rho_r\in SBV(Q^\nu_r( \tfrac{ r  x}\rho),\R^m)$, with  $ v^\rho_r=u_{\frac{r x  }{\rho},\zeta,\nu}$ near $\partial Q^\nu_r( \tfrac{ r  x}\rho)$, such that 
\begin{align}\label{quasi-min-g}
 E^{f^\infty,g}( v^\rho_r,Q^\nu_r( \tfrac{r  x}\rho))
 &\leq m^{f^\infty,g}\big(u_{\frac{r x  }{\rho},\zeta,\nu},Q^\nu_r(\tfrac{r x  }{\rho})\big) +\eta\, r^{n-1} \nonumber\\
 & \le \int_{\Pi^\nu_{\frac{rx}{\rho}}\cap Q^\nu_r(\frac{r x  }{\rho})}g(y,\zeta,\nu)\,d\mathcal H^{n-1} + \eta r^{n-1} \leq 
 (c_3|\zeta|+1) r^{n-1},
\end{align}
where we used that $f^\infty(\cdot,0)=0$ and $(g4)$. We extend $ v^\rho_r$ to $\R^n$ by setting
$ v^\rho_r(y)= u_{\frac{r x  }{\rho},\zeta,\nu}(y)$ for every $y\in \R^n\setminus Q^\nu_r(\tfrac{r  x}\rho)$.

By combining  \eqref{quasi-min-g} with the lower bound $(f3)$ we immediately deduce the existence of a constant $C>0$, depending on $\zeta$, such that 
\begin{equation}\label{g:unif-bd0}
\frac{1}{r^{n-1}}\int_{Q^\nu_r(\frac{r x  }{\rho})}|\nabla v_r^\rho| \dy\leq C.
\end{equation}
By Lemma \ref{estimate f finfty}, applied with $t=r/\rho$, using \eqref{g:unif-bd0}, we deduce that
\begin{equation*}
\frac{1}{r^{n-1}}\int_{Q^{\nu}_r(\tfrac{r x }{\rho})}\big|f^\infty(x,\nabla v_r^\rho)-\tfrac{\rho}{r}f(x, \tfrac{r}{\rho} \nabla v_r^\rho)\big|\dy
\le  c_5 (1+ c_4^{1-\alpha})\rho+ c_5c_3^{1-\alpha} \rho^\alpha C^{1-\alpha}.
\end{equation*}
This inequality, together with \eqref{ass-g} and \eqref{quasi-min-g}, gives 
\begin{equation*}
\limsup_{r \to +\infty}\frac{1}{r^{n-1}}\Big(\int_{Q^\nu_r(\frac{r x  }{\rho})}\tfrac{\rho}{r}f( y,\tfrac{r}{\rho}\nabla v_r^\rho)\dy +\int_{S_{v_r^\rho}\cap Q^\nu_r(\frac{r x  }{\rho})}g(y,[v_r^\rho],\nu_{v_r^\rho})\,d\mathcal H^{n-1}\Big) \leq g_{\rm hom}(\zeta,\nu) +\eta + K\rho^\alpha,
\end{equation*}
 where $K:= c_5 (1+ c_4^{1-\alpha})+c_5c_3^{1-\alpha} C^{1-\alpha}$. 
 
Given $\e>0$ and $\rho \in (0,1)$, for every $y\in\R^n$ we define  $ u^\rho_\e(y):= v^\rho_r (\frac{y}{\e})=v^\rho_r (\frac{r y}{\rho})$, with $r:=\rho/\e$. Then $ u^\rho_\e\in SBV_{\rm loc}(\R^n,\R^m)$, $ u^\rho_\e =u_{x,\zeta,\nu}$ near $\partial Q^\nu_\rho( x  )$ and in $\R^n\setminus Q^\nu_\rho( x  )$. By a change of variables, from the previous inequality we get that for every $\rho\in (0,1)$
\begin{equation}\label{scaling-g}
\limsup_{\e\to 0+}  \frac{E_\e(u^\rho_\e,Q^\nu_\rho(x))}{\rho^{n-1}} \leq g_{\rm hom}(\zeta,\nu) +\eta + K\rho^\alpha.
 \end{equation}

 Since the functions $ u^\rho_\e$ coincide with $u_{x,\zeta,\nu}$ in $Q^\nu_{(1+\eta)\rho}(x)\setminus Q^\nu_{(1-\delta_\e)\rho}(x)$ for some $\delta_\e\in(0,1)$, by $(g4)$ we have
$E_\e( u^\rho_\e,Q^\nu_{(1+\eta)\rho}( x ))\le E_\e( u^\rho_\e,Q^\nu_\rho( x )) + c_3|\zeta|2^{n-1} \rho^{n-1}\eta$, 
which, together with \eqref{scaling-g}, gives
\begin{equation}\label{scaling-g2}
\limsup_{\e \to 0+} 
\frac{ E_\e( u^\rho_\e,Q^\nu_{(1+\eta)\rho}(x))}{\rho^{n-1}} \le g_{\rm hom}(\zeta,\nu) + K\rho^\alpha +\tilde K\eta,
\end{equation}
 where $\tilde K:=1+ c_3|\zeta|2^{n-1}$. 
Since $ u^\rho_\e$ coincides with $u_{x,\zeta,\nu}$ in $\R^n\setminus Q^\nu_\rho(x)$, using the lower bounds  $(f3)$ and $(g3)$ and Poincar\'e's inequality we deduce from \eqref{scaling-g2} that for every $\rho>0$  there exists $\e(\rho)>0$ such that  the functions 
$u^\rho_\e$ are bounded in $BV_{\rm loc}(\R^n,\R^m)$ uniformly with respect to $\e\in(0,\e(\rho))$. Then there exists a subsequence, not relabelled, of the sequence $(\e_j)$ considered in Theorem \ref{thm:Gamma-conv}, such that 
 $( u^\rho_{\e_k})$ converges  in $L^1_{\rm loc}(\R^n,\R^m)$  to  some  $ u^\rho \in BV_{\rm loc}(\R^n,\R^m)$  with  $u^\rho=u_{x,\zeta,\nu}$  in   $Q^\nu_{(1+\eta)\rho}( x  )\setminus Q^\nu_{\rho}( x  )$. As a consequence  of the
$\Gamma$-convergence of $E_{\e_j}(\cdot, Q^\nu_{(1+\eta)\rho}(x))$ to  $\widehat E(\cdot, Q^\nu_{(1+\eta)\rho}(x))$, from \eqref{scaling-g2} we obtain
$$
\frac{m_{\widehat E}(u_{x,\zeta,\nu},Q^\nu_{(1+\eta)\rho}( x  ))}{\rho^{n-1}} \leq \frac{\widehat E(u^\rho,Q^\nu_{(1+\eta)\rho}( x  ))}{\rho^{n-1}}
\leq \limsup_{j\to +\infty}\frac{ E_{\e_j}( u^\rho_{\e_j},Q^\nu_{(1+\eta)\rho}(x))}{\rho^{n-1}} \le g_{\rm hom}(\zeta,\nu) + K\rho^\alpha +\tilde K\eta.
$$
Finally, passing to the limit  as $\rho \to 0+$, from \eqref{rep-hausdorff-BFM} we get
$$
(1+\eta)^{n-1}\frac{d \widehat E(u,\cdot)}{d\mathcal H^{n-1}\LLL{S_u}}(x) \le g_{\rm hom}(\zeta,\nu) + \tilde K\eta.
$$
Since $\zeta:= [u](x)$ and $\nu = \nu_u(x)$, this gives \eqref{aqcc} by the arbitrariness of $\eta>0$.

\medskip

\textit{Step 2:} We claim that 
\begin{equation}\label{aqs-g}
\frac{d \widehat E(u,\cdot)}{d\mathcal H^{n-1}\LLL{S_u}}(x) \geq g_{\rm hom}([u](x),\nu_{u}(x))  \quad \hbox{for }\mathcal H^{n-1}\hbox{-a.e.\ }x\in S_u \cap A.
\end{equation}
We extend $u$ to $\R^n$ by setting $u=0$ on $\R^n\setminus A$. By  $\Gamma$-convergence  there exists $(u_\e)\subset L^1_{\rm loc}(\R^n,\R^m)$,  with ${u_\e}|_A\in SBV(A,\R^m)$,  such that
\begin{equation}\label{g-hom:rec-seq}
u_\e \to u \quad \text{in} \quad L^1_{\rm loc}(\R^n,\R^m) \quad \text{and}\quad \lim_{\e \to 0+}E_\e(u_\e,A)=\widehat E(u,A),
\end{equation}
along the sequence $(\e_k)$ considered in Theorem \ref{thm:Gamma-conv}.

By \cite[Definition 3.67 and Step 2 in the proof of Theorem~3.77]{AFP} and thanks to a generalised version of the Besicovitch Differentiation Theorem (see \cite{Mor} and \cite[Sections~1.2.1-1.2.2]{FonLeo}), for $\hs^{n-1}$-a.e.\ $x\in S_u\cap A$ we have 
\begin{eqnarray}
& \ds \lim_{\rho \to 0+}\frac{1}{\rho^{n-1}}|Du|(Q^{\nu_u(x)}_\rho(x))=|[u](x)| \neq 0, 
\label{z neq 0}
\\
\label{u:salto}
&\ds 
 \lim_{\rho \to 0+} \frac{1}{\rho^n}\int_{Q_\rho^{\nu_u(x)}( x  )}|u(y)-u_{ x  ,[u](x),\nu_u(x)}(y)|dy = 0 ,
\\
\label{estimate:0}
&\ds \lim_{\rho \to 0+} \frac{\widehat E(u, Q_\rho^{\nu_u(x)}( x  ))}{\rho^{n-1}} = \frac{d \widehat E(u,\cdot)}{d\mathcal H^{n-1}\LLL{S_u}}(x).
\end{eqnarray}
Let us fix $ x  \in S_u\cap A$ such that \eqref{z neq 0}-\eqref{estimate:0} are satisfied, and set $\zeta:=[u]( x  )$ and $\nu:=\nu_u( x  )$. 

Recalling \eqref{4.14bis} we have that
$
\widehat E(u, Q^\nu_{\rho}(x))\leq c_3|Du|(Q^\nu_{\rho}(x)) + c_4 \rho^n
$,
hence by \eqref{z neq 0} 
there exists $\rho_0\in(0,1)$  such that
\begin{equation}\label{finitezza-g}
Q^\nu_\rho( x  ) \subset\subset A\quad\hbox{and}\quad\frac{\widehat E(u,Q^\nu_\rho(x))}{\rho^{n-1}} \leq  c_3|\zeta| + 1
\end{equation}
 for every $\rho\in(0,\rho_0)$. 
Since $\widehat E(u,\cdot)$ is a Radon measure, there exists a set $B\subset  (0,\rho_0)$, with $(0,\rho_0)\setminus B$ at most countable, such that $\widehat E(u,\partial Q^\nu_\rho( x  ))=0$ for every $\rho\in B$. Proceeding as in the proof of \eqref{4.20bis}, by \eqref{g-hom:rec-seq}   we can show that for every  $\rho\in B$
\begin{equation}\label{4.20bis-g}
\lim_{\e \to 0+}E_\e(u_\e,Q^\nu_{\rho}( x  ))=\widehat E(u,Q^\nu_{\rho}( x  )).
\end{equation}
Hence, for every  $\rho\in B$ there exists $\e(\rho)>0$  such that for every $\e\in (0,\e(\rho))$
\begin{equation}\label{g-hom:stima1}
 \frac{E_\e(u_\e,Q^\nu_\rho( x  ))}{\rho^{n-1}}\leq
\frac{\widehat E(u,Q^\nu_\rho( x  ))}{\rho^{n-1}} + \rho \le c_3|\zeta| 
+ 2,
\end{equation}
where in the last inequality we used \eqref{finitezza-g}.

The rest of this proof is devoted to  modifying $u_\e$ in order to construct  a competitor for the minimisation  problems  appearing in  \eqref{g-hom}, which defines  $g_{\rm hom}$. To this end, for every $\rho\in B$ and $\e>0$ we consider the blow-up functions defined for $y\in Q^\nu:=Q^\nu_1(0)$ by 
$$
w_\e^\rho(y):=u_\e( x  +\rho y) \quad\hbox{and}\quad  w^\rho(y):= u( x  +\rho y). 
$$
Then $ w_\e^\rho\in SBV(Q^\nu,\R^m)$ and $w^\rho\in BV(Q^\nu,\R^m)$. 
Since $u_\e \to u$ in $L^1(Q^\nu_\rho( x  ), \R^m)$ by \eqref{g-hom:rec-seq}, for every $\rho\in B$ 
we obtain 
\begin{equation}\label{4.31bis}
w_\e^\rho  \to w^\rho \qquad \text{in} \quad L^1(Q^{\nu},\R^m)\; \text{ as }\; \e \to 0+;
\end{equation}
moreover, from \eqref{u:salto} we have  
\begin{equation}\label{4.31ter}
w^\rho \to u_{0,\zeta,\nu} \quad \text{in} \quad L^1(Q^{\nu},\R^m)\; \text{ as }\; \rho \to 0+. 
\end{equation}
Up to possibly reducing the values of $\rho_0$ and $\e(\rho)$, we may then assume that
\begin{equation}\label{aqj-g}
\|w_\e^\rho -  u_{0,\zeta,\nu}\|_{L^1(Q^\nu,\R^m)}\le 1,
\end{equation} 
for every $\rho\in B$ and $\e\in(0,\e(\rho))$.  By a change of variables, we obtain the relation
\begin{equation}\label{change of variables}
\frac{E_\e(u_\e,Q^{\nu}_\rho( x  ))}{\rho^{n-1}} =  E^\rho_\e( w_\e^\rho,Q^\nu),
\end{equation}
where $E^\rho_\e$ is the functional corresponding to the integrands 
$ f_\e^\rho(y,\xi):=\rho f(\tfrac{ x  +\rho y}{\e},\tfrac{\xi}{\rho})$  and $g_\e^\rho(y,\zeta,\nu):=g(\tfrac{ x  +\rho y}{\e},\zeta,\nu)$\ie
$$
E^\rho_\e(w,Q^\nu):=  \int_{Q^{\nu}}\rho f(\tfrac{ x  +\rho y}{\e},\tfrac{1}{\rho}\,\nabla w(y))\dy+\int_{S_{ w}\cap Q^{\nu}} g(\tfrac{ x  +\rho y}{\e}, [ w](y),\nu_{ w}(y))d\mathcal H^{n-1}(y)
$$
for every $w\in SBV(Q^\nu,\R^m)$. Note that $f_\e^\rho$ satisfies $(f3)$ and $(f4)$ (recall that $\rho<1$), while $g_\e^\rho$ satisfies $(g3)$ and $(g4)$. 

We now modify $ w_\e^\rho$ in a way such that it attains the boundary datum $u_{0,\zeta,\nu}$ near $\partial Q^\nu$. This  will be done by applying the Fundamental Estimate \cite[Proposition 3.1]{BDfV} to the functionals $E^\rho_\e$. Thus for $\eta \in (0,\frac12)$ fixed there exist a constant $L_\eta>0$ with the following property: for every $\rho\in B$ and $\e\in(0,\e(\rho))$ there exists a cut-off function $\varphi^\rho_\e \in C^\infty_c(Q^\nu)$, with $0\leq\varphi^\rho_\e \leq 1$ in $Q^\nu$, ${\rm supp}(\varphi^\rho_\e)\subset Q^\nu_{1-\eta}:=Q^\nu_{1-\eta}(0)$, and $\varphi^\rho_\e=1$ in $Q^\nu_{1-2\eta}:=Q^\nu_{1-2\eta}(0)$, such that, setting
$\hat w^\rho_\e:=\varphi^\rho_\e w^\rho_\e+ (1-\varphi^\rho_\e)u_{0,\zeta,\nu}$, we have 
\begin{equation}\label{stima fondamentale superficie}
E^\rho_\e (\hat w^\rho_\e,Q^\nu)\leq (1+\eta)\big( E^\rho_\e ( w_\e^\rho,Q^\nu)+ E^\rho_\e (u_{0,\zeta,\nu},Q^\nu\setminus \overline Q^\nu_{1-2\eta})\big)+ L_\eta\| w_\e^\rho-u_{0,\zeta,\nu}\|_{L^1(Q^\nu,\R^m)}.
\end{equation}
By definition we clearly have $\hat w^\rho_\e=u_{0,\zeta,\nu}$ in  $Q^\nu \setminus Q^\nu_{1-\eta}$, as desired. Moreover, from $(f4)$ and $(g4)$ we obtain the bound 
\begin{align}\label{aqn-g}
E^\rho_\e (u_{0,\zeta,\nu},Q^\nu\setminus \overline Q^\nu_{1-2\eta})) &\leq 
c_4 \mathcal{L}^n(Q^\nu\setminus  Q^\nu_{1-2\eta})  + c_3 |\zeta|\,\mathcal{H}^{n-1}(\Pi^\nu_0\cap(Q^{\nu}\setminus\overline Q^\nu_{1-2\eta}))\nonumber\\
 &\leq  2 c_4 n \eta + 2c_3 |\zeta| (n-1)\eta,
\end{align}
and hence, from \eqref{stima fondamentale superficie} and \eqref{aqn-g} we have that 
\begin{equation}\label{724}
E^\rho_\e (\hat w^\rho_\e,Q^\nu)\leq (1+\eta)E^\rho_\e ( w_\e^\rho,Q^\nu)+ K \eta + L_\eta\| w_\e^\rho-u_{0,\zeta,\nu}\|_{L^1(Q^\nu,\R^m)},
\end{equation}
for every $\rho\in B$ and every $\e\in (0,\e(\rho))$, where $K = 3(c_4 \rho n + c_3 |\zeta| (n-1))$.
Note that $(f3)$, $(g3)$, \eqref{g-hom:stima1},  \eqref{aqj-g}, \eqref{change of variables} and \eqref{724} imply that 
\begin{equation}\label{g2:unif-bd}
c_2\int_{Q^\nu} |\nabla \hat w_\e^\rho(y)|\dy \leq (1+\eta)(c_3|\zeta|+1) + K \eta + L_\eta,
\end{equation}
for every $\rho\in B$ and $\e\in(0,\e(\rho))$.

We finally set $r:=\frac\rho\e$  and $v_\e^\rho(y):=\hat w_\e^\rho(\tfrac{y}r - \tfrac{x  }{\rho})$; clearly $v_\e^\rho \in SBV(Q^\nu_r( \tfrac{r x}\rho),\R^m)$, $v_\e^\rho = u_{\frac{rx}{\rho},\zeta,\nu}$ near $\partial Q^\nu_r( \tfrac{r x}\rho)$, and via a change of variables, using also \eqref{g2:unif-bd}, we have that 
\begin{gather}\label{g2:unif-bd2}
\frac{1}{r^{n-1}}\int_{Q^\nu_r( \tfrac{r x}\rho)} |\nabla v_\e^\rho(y)|\dy \leq C,\\\label{aqg-g}
\frac{1}{r^{n-1}}\Big(\int_{Q^\nu_r( \tfrac{r  x}\rho)}\tfrac{\rho}r f(y,\tfrac{r}{\rho}\nabla v^\rho_\e)\dy 
 +\int_{S_{v^\rho_\e}\cap Q^\nu_r( \tfrac{r  x}\rho)}g(y, [v^\rho_r],\nu_{v^\rho_\e})\,d\mathcal H^{n-1}\Big)
 = E^\rho_\e (\hat w^\rho_\e,Q^\nu),
\end{gather}
where $C$ depends only on $|\zeta|$. Hence by Lemma \ref{estimate f finfty}, applied  with $t=r/\rho$, using \eqref{g2:unif-bd2},  we have
\begin{align}\label{f-f-infty}
\frac{1}{r^{n-1}}\int_{Q^\nu_r( \tfrac{r  x}\rho)}
\big|f^\infty(y,\nabla v^\rho_\e)-\tfrac{\rho}r f(y,\tfrac{r}{\rho}\nabla v^\rho_\e)\big|\dy 
\leq c_5 (1+ c_4^{1-\alpha})\rho+ c_5c_3^{1-\alpha}\rho^\alpha  C^{1-\alpha}
 \le \tilde K\rho^\alpha
\end{align}
for every $\rho\in B$ and $\e\in(0,\e(\rho))$, where $\tilde K:=   c_5 (1+ c_4^{1-\alpha})+ c_5c_3^{1-\alpha}  C^{1-\alpha}$. From \eqref{change of variables}, \eqref{724}, \eqref{aqg-g}, and \eqref{f-f-infty} we obtain 
$$
 \frac{E^{f^\infty,g}(v^\rho_\e,Q^\nu_r(\tfrac{r x}\rho))}{r^{n-1}} \le (1+\eta) \frac{E_\e(u_\e,Q^{\nu}_\rho( x  ))}{\rho^{n-1}}
 +K\eta + L_{\eta}\| w_\e^\rho-u_{0,\zeta,\nu}\|_{L^1(Q^\nu,\R^m)} + \tilde K\rho^\alpha.
 $$
Since $v_\e^\rho = u_{\frac{rx}{\rho},\zeta,\nu}$ near $\partial Q^\nu_r( \tfrac{r x}\rho)$, we have that 
$$
\frac{m^{f^\infty, g}(u_{\frac{rx}{\rho},\zeta,\nu}, Q^\nu_r( \tfrac{r x}\rho))}{r^{n-1}}\le (1+\eta) \frac{E_\e(u_\e,Q^{\nu}_\rho( x  ))}{\rho^{n-1}}
 +K\eta + L_{\eta}\| w_\e^\rho-u_{0,\zeta,\nu}\|_{L^1(Q^\nu,\R^m)} + \tilde K\rho^\alpha.
$$
Since $r=\frac\rho\e$, by \eqref{g-hom} with $x$ replaced by $\frac{x}\rho$, the left-hand side converges to $ g_{\rm hom}(\xi)$ as $\e\to0+$. By \eqref{4.20bis-g} and \eqref{4.31bis} we can pass to the limit in the right-hand side as $\e\to0+$ and we obtain
$$
g_{\rm hom}(\zeta,\nu) \le (1+\eta) \frac{\widehat E(u,Q^{\nu}_\rho( x  ))}{\rho^{n-1}} +K\eta + L_{\eta}\| w^\rho-u_{0,\zeta,\nu}\|_{L^1(Q^\nu,\R^m)} + \tilde K\rho^\alpha.
$$
 By \eqref{estimate:0} and  \eqref{4.31ter}, passing to the limit as $\rho\to0+$ we get
 $$
g_{\rm hom}(\zeta,\nu) \le (1+\eta) \frac{d \widehat E(u,\cdot)}{d\mathcal H^{n-1}\LLL{S_u}}(x) +K\eta.
$$
Since $\zeta= [u](x)$ and $\nu = \nu_u(x)$, this inequality gives \eqref{aqs-g} by the arbitrariness of $\eta>0$.

\end{proof}

\section{ Identification of the Cantor term}\label{Sect:Cantor}

In Proposition \ref{p:homo-Can} below we characterise the derivative of $\widehat E(u,\cdot)$ with respect to $|C(u)|$, the variaton of the Cantor part of the measure $Du$. At the end of this section we conclude the proof of 
Theorem~\ref{T:det-hom}. 

 The following proposition shows that the recession function of $f_{\rm hom}$, which is defined in terms of the minimum values $m^{f,g_0}$, can be obtained from suitably rescaled limits of the minimum values $m^{f^\infty,g_0}$, involving now the recession function $f^\infty$ of $f$.

\begin{prop}\label{p:h}
Let $f\in \mathcal F$ and $g\in \mathcal G$, and let  $m^{f,g_0}$ and $m^{f^\infty,g_0}$ be as in \eqref{m-phipsi},  with  $(f,g)$ replaced by  $(f,g_0)$ and $(f^\infty,g_0)$, respectively. Assume that (a) of Theorem~\ref{T:det-hom} holds, and let $f_{\rm hom}$ be as in \eqref{f-hom}.  Let $f_{\rm hom}^\infty$ be the recession function of $f_{\rm hom}$ (whose existence is guaranteed by the fact that $f_{\rm hom} \in \mathcal F$). Then
\begin{equation}\label{asa}
f^\infty_{\rm hom}(\xi)=\lim_{r\to +\infty} \frac{m^{f^\infty,g_0}(\ell_\xi,Q^{\nu,k}_r(rx))}{k^{n-1}r^{n}}
\end{equation}
 for every $x\in \R^{n}$, $\xi \in \R^{m\times n}$, $\nu\in \Sph^{n-1}$, and $k\in \N$.
\end{prop}

\begin{proof}
 Let $x\in \R^{n}$, $\xi \in \R^{m\times n}$, $\nu\in \Sph^{n-1}$, $k\in \N$, and $\eta\in(0,1)$ be fixed. By \eqref{m-phipsi} for every $r>0$ there exists $ v_r\in SBV(Q^{\nu,k}_r(rx),\R^m)$, with  $ v_r=\ell_\xi$ near $\partial Q_r^{\nu,k}({ r  x})$, such that  
\begin{equation}
E^{f^\infty,g_0}(v_r, Q^{\nu,k}_r(rx))\leq m^{f^\infty,g_0}(\ell_{\xi},Q^{\nu,k}_r(rx)) +\eta k^{n-1} r^n.
\label{4.6bis}
\end{equation} 
Note that, by $(f3)$ and $(f4)$, this implies that 
\begin{equation}\label{4.6biss}
c_2\int_{Q^{\nu,k}_r(rx)}|\nabla v_r|dy \leq m^{f^\infty,g_0}(\ell_{\xi},Q^{\nu,k}_r(rx)) +\eta k^{n-1} r^n
\leq (c_3|\xi|+ 1) k^{n-1} r^n,
\end{equation}
where we used the fact that $f^\infty$ satisfies  $(f4)$ with $c_4=0$. Let $t>1$; by Lemma \ref{estimate f finfty} and by \eqref{4.6biss}, recalling that $\alpha\in (0,1)$,  we have 
\begin{align*}
&\int_{Q^{\nu,k}_r(rx)}\big|f^\infty(y,\nabla v_r)-\tfrac1t f(y,t\nabla v_r)\big|\dy\\
& \leq  \frac1t c_5 (1+ c_4^{1-\alpha})k^{n-1}r^n+ \frac{1}{t^\alpha}c_5c_3^{1-\alpha} (k^{n-1}r^n)^\alpha \|\nabla v_r\|^{1-\alpha}_{L^1(Q^{\nu,k}_r(rx),\R^m)}
\leq \frac{1}{t^\alpha} K k^{n-1}r^n,
\end{align*}
where  $K=c_5 (1+ c_4^{1-\alpha})+c_5(c_3/c_2)^{1-\alpha}(c_3|\xi|+ 1)^{1-\alpha}$. Hence 
$$
E^{f_t,g_0}(v_r, Q^{\nu,k}_r(rx)) \leq E^{f^\infty,g_0}(v_r, Q^{\nu,k}_r(rx)) +   \frac{1}{t^\alpha} K k^{n-1}r^n,
$$
where $f_t(y,\xi):=\frac1t  f(y,t \xi)$.

Since $v_r\in SBV(Q^{\nu,k}_r(rx),\R^m)$ and $ v_r=\ell_\xi$ near $\partial Q_r^{\nu,k}({ r  x})$, the previous inequality, together with \eqref{m-phipsi} and \eqref{4.6bis}, gives  
\begin{align}\label{mm2}
\frac{m^{f_t,g_0}(\ell_{\xi},Q^{\nu,k}_r(rx))}{k^{n-1}r^n} \leq \frac{m^{f^\infty,g_0}(\ell_{\xi},Q^{\nu,k}_r(rx))}{k^{n-1}r^n} +\eta +   \frac{1}{t^\alpha} K.
\end{align}
We now let $r\to +\infty$ in the previous estimate. For the left-hand side we note that, by the definition of $f_{\rm hom}$, \eqref{m-phipsi}, and the  positive  $1$-homogeneity of $g_0$ with respect to its second variable, by a change of variables we have
\begin{eqnarray}\label{h=f-infty}
\lim_{r\to +\infty} \frac{m^{f_t,g_0}(\ell_{\xi},Q^{\nu,k}_r(rx))}{k^{n-1}r^{n}}
=\lim_{r\to +\infty} \frac{m^{f,g_0}(\ell_{t\xi},Q^{\nu,k}_r(rx))}{t\,k^{n-1}r^{n}} = \frac{f_{\rm hom}(t \xi)}{t}.
\end{eqnarray}
Hence, from \eqref{mm2} and \eqref{h=f-infty} we have that 
\begin{equation*}
\frac{f_{\rm hom}(t \xi)}{t} \leq \liminf_{r\to +\infty}\frac{m^{f^\infty,g_0}(\ell_{\xi},Q^{\nu,k}_r(rx))}{k^{n-1}r^n} +\eta +   \frac{1}{t^\alpha} K.
\end{equation*}
By letting $t\to +\infty$, since $\eta\in (0,1)$ is arbitrary, we obtain the inequality 
$$
f^\infty_{\rm hom}(\xi) \leq \liminf_{r\to +\infty}\frac{m^{f^\infty,g_0}(\ell_{\xi},Q^{\nu,k}_r(rx))}{k^{n-1}r^n}.
$$
Exchanging the roles of $f_t$ and $f^\infty$,  an analogous argument yields the inequality 
$$
\limsup_{r\to +\infty}\frac{m^{f^\infty,g_0}(\ell_{\xi},Q^{\nu,k}_r(rx))}{k^{n-1}r^n}\leq f^\infty_{\rm hom}(\xi) ,
$$
and hence \eqref{asa} follows.
\end{proof}

For later purposes it is convenient to prove that $f^\infty_{\rm hom}$ can be equivalently expressed in terms 
of a (double) limit involving minimisation problems where the Dirichlet conditions are 
prescribed only on a part of the boundary. 
We recall that the definitions of  $\partial^\perp_\nu  Q^{\nu,k}_r(rx)$ and $\partial^\parallel_\nu  Q^{\nu,k}_r(rx)$
are given in (i) in
Section \ref{Notation}, while the meaning of the boundary 
condition on a part of the boundary is explained after \eqref{m-phipsi}.

\begin{lem}
Under the assumptions of Proposition \ref{p:h} we have that 
\begin{equation}\label{h-hom-2}
f_{\rm hom}^\infty(a\otimes \nu)=\lim_{k\to +\infty}\liminf_{r\to +\infty} \frac{\widetilde m^{f^\infty,g_0}(\ell_{a \otimes \nu},Q^{\nu,k}_r(rx))}{k^{n-1}r^{n}}=\lim_{k\to +\infty}\limsup_{r\to +\infty} \frac{\widetilde m^{f^\infty,g_0}(\ell_{a \otimes \nu},Q^{\nu,k}_r(rx))}{k^{n-1}r^{n}}
\end{equation}
for every $x\in \R^{n}$, $a \in \R^{m}$, and $\nu\in \Sph^{n-1}$, where 
\begin{equation}\label{min-neumann}
\widetilde m^{f^\infty,g_0}(\ell_{a \otimes \nu},Q^{\nu,k}_r(rx)):=\inf\{ E^{f^\infty,g_0}(v, Q^{\nu,k}_r(rx)): 
 v \in \mathcal U^{\nu,k}_{a,r}(x)\},
\end{equation}
with
$
\mathcal U^{\nu,k}_{a,r}(x):=\{v \in SBV(Q^{\nu,k}_r(rx),\R^m) \colon
v=\ell_{a\otimes \nu} \textrm{ near }\partial^\perp_\nu  Q^{\nu,k}_r(rx)\}
$.
\end{lem}

\begin{proof} Let $x\in\R^n$, $a\in \R^m$, and $\nu\in \Sph^{n-1}$ be fixed, and for every $r>0$ and $k>0$ let
$Q^{\nu,k}_{x,r}:=Q^{\nu,k}_r(rx)$. Since 
$\partial^\perp_\nu Q^{\nu,k}_{x,r}\subset \partial Q^{\nu,k}_{x,r}$, 
we have 
$\widetilde m^{f^\infty,g_0}(\ell_{a \otimes \nu}, Q^{\nu,k}_{x,r})\le m^{f^\infty,g_0}(\ell_{a \otimes \nu}, Q^{\nu,k}_{x,r})$.
Due to \eqref{asa},  to obtain \eqref{h-hom-2} we only need to prove the inequality
\begin{equation}\label{h-hom-2-proof}
f^\infty_{\rm hom}(a\otimes \nu)\leq \liminf_{k\to +\infty}\liminf_{r\to +\infty}  \frac{\widetilde m^{f^\infty,g_0}(\ell_{a \otimes \nu}, Q^{\nu,k}_{x,r})}{k^{n-1}r^{n}}.
\end{equation}
To this aim, let us fix $h\in\N$. For every  $r\ge 1$ and $k\in \N$, with $k\ge h$, there exists  $v_r^k\in \mathcal U^{\nu,k}_{a,r}(x)$ such that 
\begin{equation}\label{tildem}
E^{f^\infty, g_0}(v_r^k, Q^{\nu,k}_{x,r}) \leq 
\widetilde m^{f^\infty,g_0}(\ell_{a \otimes \nu}, Q^{\nu,k}_{x,r}) + 1  \le  (c_3|a|+1)k^{n-1} r^n,
\end{equation}
where we used the fact that $f^{\infty} (y,\xi) \leq c_3 |\xi|$ and $1\le k^{n-1} r^n$.
By  $(f3)$ and $(g3)$, inequality \eqref{tildem} implies that
\begin{equation}\label{8.100}
c_2 |Dv_r^k|(Q^{\nu,k}_{x,r})\le (c_3|a|+1) k^{n-1} r^n.
\end{equation}
Changing $v_r^k$ in an $\mathcal L^n$-negligible set, we may assume that $v_r^k(y)$ coincides with the approximate limit of $v_r^k$ at $y$ for every $y\in Q^{\nu,k}_{x,r}\setminus S_{v_r^k}$ (see \cite[Definition~3.63]{AFP}). By Fubini's theorem we have
\begin{equation*}
\int_{k-h}^k\int_{\partial^\parallel_\nu Q^{\nu,\lambda}_{x,r}}|v_r^k-\ell_{a \otimes \nu}|d\mathcal H^{n-1}d\lambda=
\frac{2}{r} \int_{Q^{\nu,k}_{x,r}\setminus Q^{\nu,k-h}_{x,r}}|v_r^k-\ell_{a \otimes \nu}|\dy
\le \frac{2}{r} \int_{Q^{\nu,k}_{x,r}}|v_r^k-\ell_{a \otimes \nu}|\dy.
\end{equation*}

Since $v_r^k-\ell_{a \otimes \nu}=0$ near $\partial^\perp_\nu Q^{\nu,k}_{x,r}$, by Poincar\'e's inequality on strips we have
\begin{equation}\label{8.103}
\frac1r \int_{Q^{\nu,k}_{x,r}}|v_r^k-\ell_{a \otimes \nu}|\dy
\le  |Dv_r^k-{a \otimes \nu}|(Q^{\nu,k}_{x,r})
\le |Dv_r^k|(Q^{\nu,k}_{x,r})+|a|k^{n-1}r^n.
\end{equation}
Since $\mathcal H^{n-1}$ is $\sigma$-finite on $S_{v_r^k}$, we have 
$\mathcal H^{n-1}(S_{v_r^k}\cap\partial^\parallel_\nu Q^{\nu,\lambda}_{x,r})=0$ for $\mathcal L^1$-a.e.\ $\lambda\in (k-h,k)$.
From \eqref{8.100}-\eqref{8.103} we deduce that there exists $\lambda^k_r\in(k-h,k)$ such that
$\mathcal H^{n-1}(S_{v_r^k}\cap\partial^\parallel_\nu Q^{\nu,\lambda^k_r}_{x,r})=0$ and
\begin{equation}\label{8.104}
\int_{\partial^\parallel_\nu Q^{\nu,\lambda^k_r}_{x,r}}|v_r^k-\ell_{a \otimes \nu}|d\mathcal H^{n-1}
\le\frac{2}{h} |Dv_r^k|(Q^{\nu,k}_{x,r})+\frac{2|a|}h k^{n-1}r^n 
\le
\frac{C}{h}k^{n-1}r^n.
\end{equation}
where $C:= 2(c_3|a|+1)/c_2 +2 |a|$. 

To prove \eqref{h-hom-2-proof} we need to modify $v_r^k$ so that it attains the affine boundary datum $\ell_{a\otimes \nu}$ near the whole boundary $\partial Q^{\nu,k}_{x,r}$, and hence is a competitor for the minimisation problem in the definition of $f^{\infty}_{\rm hom}$. 
The modified function is defined by 
$$
\hat{v}_r^k := 
\begin{cases}
\smash{v_r^k} \quad & \text{in } \smash{Q^{\nu,\lambda^k_r}_{x,r}},\\
\ell_{a\otimes \nu} \quad & \text{in } \R^n\setminus  Q^{\nu,\lambda^k_r}_{x,r}.
\end{cases}
$$
Then $\hat{v}_r^k  \in SBV(Q^{\nu,k}_{x,r}, \R^m)$ and $\hat{v}_r^k = \ell_{a\otimes \nu}$ near $\partial Q^{\nu,k}_{x,r}$. Moreover,  since $\mathcal H^{n-1}(S_{v_r^k}\cap\partial^\parallel_\nu Q^{\nu,\lambda^k_r}_{x,r})=0$, by $(f4)$ and $(g4)$ we have
\begin{equation}\label{8.105}
E^{f^\infty,g_0}(\hat{v}_r^k, Q^{\nu,k}_{x,r}) \leq E^{f^\infty,g_0}(v_r^k, Q^{\nu,k}_{x,r})
+ c_3|a|\mathcal L^n(Q^{\nu,k}_{x,r}\setminus Q^{\nu,\lambda^k_r}_{x,r}) + 
 c_3 \int_{\partial^\parallel_\nu Q^{\nu,\lambda^k_r}_{x,r}}|v_r^k-\ell_{a \otimes \nu}|d\mathcal H^{n-1}.
 \end{equation}
 Since $k^{n-1}-(\lambda^k_r)^{n-1}\le k^{n-1}-(k-h)^{n-1}\le (n-1)k^{n-2}h$, from \eqref{tildem}, \eqref{8.104}, and  \eqref{8.105} we obtain
 \begin{equation*}
m^{f^\infty,g_0}(\ell_{a \otimes \nu},Q^{\nu,k}_{x,r})\le E^{f^\infty,g_0}(\hat{v}_r^k, Q^{\nu,k}_{x,r}) \leq \widetilde m^{f^\infty,g_0}(\ell_{a \otimes \nu}, Q^{\nu,k}_{x,r}) + (n-1)c_3 |a|k^{n-2}hr^n+ 
\frac{c_3C}{h}k^{n-1}r^n+1.
\end{equation*}
We then divide both sides of the previous inequality by $k^{n-1} r^n$, to obtain 
$$
\frac{m^{f^\infty,g_0}(\ell_{a \otimes \nu}, Q^{\nu,k}_{x,r})}{k^{n-1}r^n} \leq \frac{\widetilde m^{f^\infty,g_0}(
\ell_{a \otimes \nu}, Q^{\nu,k}_{x,r})}{k^{n-1}r^n}+  
(n-1)c_3 |a|\frac{h}{k}+\frac{c_3C}{h}+\frac1{k^{n-1}r^n}.
$$
Taking the limit first as $r\to +\infty$, then as $k\to +\infty$, and finally as $h\to +\infty$, from \eqref{asa} we obtain \eqref{h-hom-2-proof}, and hence  \eqref{h-hom-2}.
\end{proof}

In the next proposition we characterise the derivative of $\widehat E(u,\cdot)$ with respect to $|C(u)|$, for any $BV$-function $u$.

\begin{prop}[Homogenised Cantor integrand]\label{p:homo-Can}  Let $f$, $g$, $E_\e$, $(\e_k)$, and 
 $\widehat E$ be as in Theorem \ref{thm:Gamma-conv}. Assume that (a) of Theorem~\ref{T:det-hom} holds, let $f_{\rm hom}$ be as in \eqref{f-hom}, and let $f_{\rm hom}^\infty$ denote its recession function (whose existence is guaranteed by the fact that $f_{\rm hom} \in \mathcal F$).  
 Then for every $A\in \A$ and every $u\in L^1_{\rm loc}(\R^n,\R^m)$, with $ u|_A\in BV(A,\R^m)$, we have that 
\begin{equation*}
\frac{d\widehat E(u,\cdot)}{d|C(u)|}(x)=f^\infty_{\rm hom}\bigg(\frac{dC(u)}{d|C(u)|}(x)\bigg)
\end{equation*} 
for $|C(u)|$-a.e. $x\in A$.    
\end{prop}
\begin{proof} Let us fix $A\in \A$ and $u\in L^1_{\rm loc}(\R^n,\R^m)$, with $ u|_A\in BV(A,\R^m)$.
We divide the proof into two steps.

\medskip

\textit{Step 1:} We claim that
\begin{equation}\label{aqccc}
\frac{d\widehat E(u,\cdot)}{d|C(u)|}(x)\leq f^\infty_{\rm hom}\bigg(\frac{dC(u)}{d|C(u)|}(x)\bigg) \quad \textrm{for } |C(u)|\textrm{-a.e. } x\in A.
\end{equation}
By Alberti's  rank-one theorem  \cite{Al93} we know that  for $|C(u)|$-a.e.\ $x\in A$ we have 
\begin{equation}\label{C-rank-one}
\frac{dC(u)}{d|C(u)|}(x) = a(x) \otimes \nu(x)
\end{equation}
for a suitable pair $(a(x),\nu(x)) \in \R^m \times \Sph^{n-1}$. 
Moreover, by (a)-(e) of Theorem \ref{thm:Gamma-conv} and by \cite[Lemma 3.9]{BFM} we have that for $|C(u)|$-a.e.\ $x\in A$ there exists a 
doubly indexed positive sequence $( t_{\rho,k} )$, with $\rho>0$ and $k\in \N$, such that
\begin{gather}
\text{for every }k\in\N\quad  t_{\rho,k}  \to +\infty\quad\text{and}\quad \rho\,  t_{\rho,k}  \to {0+} \quad \text{as } \rho \to {0+},
\label{t-rho}
\\
\frac{d\widehat E(u,\cdot)}{d|C(u)|}(x)=\lim_{k \to +\infty} \limsup_{\rho \to 0+}\frac{m_{\widehat E}(\ell_{ t_{\rho,k} a(x)\otimes \nu(x)}, Q_\rho^{\nu(x),k}(x))}{k^{n-1}\rho^n\,  t_{\rho,k}}. 
\label{C-deriv-form}
\end{gather}
Let $ x \in A$ be fixed such that  \eqref{C-rank-one}-\eqref{C-deriv-form}  hold true and set $a:=a( x )$ and $\nu:=\nu( x )$. 

For every $\rho>0$ and every $k\in \N$ we have 
\begin{equation}\label{ass-h} 
f^\infty_{\rm hom}(a\otimes\nu)=\lim_{r\to +\infty} \frac{m^{f^\infty,g_0}(\ell_{a\otimes \nu},Q^{\nu,k}_r(\tfrac{r x }{\rho}))}{k^{n-1}r^{n}} ,
\end{equation}
since the above identity directly follows from \eqref{asa} by replacing $x$ with $\frac{ x}{\rho}$. 

Let  us fix $\eta\in(0,\frac12)$. By \eqref{m-phipsi} for every $k\in\N$, $\rho\in (0,1)$ and  $r>0$ there exists a function  $v^{\rho,k}_r\in SBV\big(Q^{\nu,k}_r\big(\frac{r x }{\rho}\big),\R^m\big)$ with $v^{\rho,k}_r= \ell_{a\otimes \nu}$  near  $\partial Q^{\nu,k}_r(\frac{r x }{\rho})$ and such that
\begin{equation}\label{quasi-min-h}
 E^{f^\infty,g_0}(v^{\rho,k}_r, Q^{\nu,k}_r(\tfrac{r x }{\rho}) )
 \leq m^{f^\infty,g_0}(\ell_{a\otimes \nu},Q^{\nu,k}_r(\tfrac{r x }{\rho})) +\eta\, k^{n-1} r^n  \leq (c_3|a|+1) k^{n-1} r^n,
\end{equation}
where we used the fact that $f^{\infty} (y, \xi) \leq c_3 |\xi|$. We extend $ v^{\rho,k}_r$ to $\R^n$ by setting
$ v^{\rho,k}_r(y)= \ell_{a\otimes \nu}(y)$ for every $y\in \R^n\setminus Q^{\nu,k}_r(\tfrac{r  x}\rho)$.

For every $y\in \R^n$ we define $w^{\rho,k}_r(y):=\frac1r v^{\rho,k}_r(\frac{rx}\rho + ry) -\ell_{a\otimes\nu}(\frac{x}\rho)$. Clearly $ w^{\rho,k}_r\in SBV_{ \rm loc}(\R^n,\R^m)$ and $ w^{\rho,k}_r =\ell_{a\otimes\nu}$ near $\partial Q^{\nu,k}$ and in $\R^n\setminus Q^{\nu,k}$, where $Q^{\nu,k}:=Q^{\nu,k}_1(0)$. Moreover, by a change of variables, using the $1$-homogeneity of $g_0$ in the second variable we have 
\begin{align}\label{scaling-hh}
E^{f^\infty_{r, \rho}, g_{0}^{r, \rho}} (w^{\rho,k}_r, Q^{\nu,k}) 
=  \frac{1}{r^n}E^{f^\infty,g_0}(v^{\rho,k}_r, Q^{\nu,k}_r(\tfrac{r x }{\rho}) ), 
\end{align}
 where $E^{f^\infty_{r, \rho}, g_{0}^{r, \rho}}$ is the functional with integrands $ f^\infty_{r, \rho} (y,\xi):=f^\infty(\tfrac{r x}{\rho} +r y ,\xi)$ and $g_{0}^{r, \rho} (y,\zeta,\nu):=
g_0(\tfrac{r x}{\rho} + r y,  \zeta,\nu)$\ie
 $$
E^{f^\infty_{r, \rho}, g_{0}^{r, \rho}} (w, Q^{\nu,k})=  \int_{Q^{\nu,k}}f^\infty(\tfrac{r x}{\rho} +r y ,\nabla w(y))\dy +\int_{S_{w}\cap Q^{\nu,k}}g_0(\tfrac{r x}{\rho} + r y,[w](y),\nu_{w}(y))\,d\mathcal H^{n-1}(y) \nonumber
$$
for every $w\in SBV(Q^{\nu,k},\R^m)$. Note that $f^\infty_{r, \rho} \in \mathcal{F}$ and $g_{0}^{r, \rho} \in \mathcal{G}$.  
By the lower bounds  $(f3)$ and $(g3)$, from \eqref{quasi-min-h} and \eqref{scaling-hh}, using also
Poincar\'e's inequality, we deduce that $\|w^{\rho,k}_r\|_{L^1(Q^{\nu,k},\R^m)}+|Dw^{\rho,k}_r|(Q^{\nu,k})\le C k^{n-1}$ for every $\rho\in (0,1)$ and $r>0$, with $C:=2(c_3|a|+1)/c_2+2|a|$.

By Lemma \ref{truncation} there exist a constant $ M_{\eta,k} >0$, depending also on $|a|$ and $C$, such that for every $\rho\in(0,1)$ and $r>0$ there exists $\tilde w^{\rho,k}_r\in SBV(Q^{\nu,k},\R^m)\cap L^\infty(Q^{\nu,k},\R^m)$ with the following properties: $\tilde w^{\rho,k}_r=\ell_{a\otimes \nu}$ near $\partial Q^{\nu,k}$,  $\|\tilde w^{\rho,k}_r\|_{L^\infty(Q^{\nu,k},\R^m)}\le M_{\eta,k} $, and
\begin{equation}\label{8.300}
E^{f^\infty_{r, \rho}, g_{0}^{r, \rho}} (\tilde w^{\rho,k}_r, Q^{\nu,k}) \leq E^{f^\infty_{r, \rho}, g_{0}^{r, \rho}} (w^{\rho,k}_r, Q^{\nu,k}) + \eta\, k^{n-1}  
 \leq \frac{m^{f^\infty,g_0}(\ell_{a\otimes \nu},Q^{\nu,k}_r(\tfrac{r x}{\rho}))}{r^n}  +2\eta\, k^{n-1},
\end{equation}
where the last inequality follows from \eqref{ass-h} and \eqref{quasi-min-h}.
Let $\tilde v^{\rho,k}_r\in SBV_{\rm loc}(\R^n,\R^m)$ be defined by 
$\tilde v^{\rho,k}_r(y):=r\tilde w^{\rho,k}_r(\frac{y}{r}-\frac{x}{\rho})+ \ell_{a\otimes \nu}(\frac{rx}\rho)$. 
Then $\tilde v^{\rho,k}_r=\ell_{a\otimes \nu}$ near $\partial Q^{\nu,k}_r( \tfrac{ r  x}\rho)$ and, 
by a change of variables,
\begin{gather}
 \|[\tilde v^{\rho,k}_r]\|_{L^\infty(S_{\smash{\tilde v^{\rho,k}_r}}\cap Q_r(\frac{rx}\rho),\R^m)}\leq 2 M_{\eta,k}\, r,
 \label{101}\\
 \label{222}
E^{f^\infty,g_0} (\tilde v^{\rho,k}_r, Q^{\nu,k}_r( \tfrac{ r  x}\rho)) = r^n E^{f^\infty_{r, \rho}, g_{0}^{r, \rho}} (\tilde w^{\rho,k}_r, Q^{\nu,k}) \leq m^{f^\infty,g_0}(\ell_{a\otimes \nu},Q^{\nu,k}_r(\tfrac{r x}{\rho})) + 2\eta\, k^{n-1} r^n,
\end{gather}
where the last inequality follows from \eqref{8.300}.
Moreover, by combining \eqref{quasi-min-h} and \eqref{222} with the lower bounds $(f3)$ and $(g3)$ we immediately deduce the existence of a constant $C>0$, depending on $|a|$, such that
\begin{gather}
\frac{1}{k^{n-1}r^n}\int_{Q^{\nu,k}_r( \tfrac{ r  x}\rho)} |\nabla \tilde v^{\rho,k}_r|\dy \leq C,\label{111}
\\
\label{121}
\frac{1}{k^{n-1}r^n}\int_{S_{\smash{\tilde v^{\rho,k}_r}}\cap Q^{\nu,k}_r( \tfrac{ r  x}\rho)} |[\tilde v^{\rho,k}_r]| d\mathcal{H}^{n-1}  \leq C.
\end{gather}

Let $\rho\in (0,1)$ and $k \in \N$, with $t_{\rho,k}>1$. By  Lemma \ref{estimate f finfty}, applied with $t=t_{\rho,k}$, and \eqref{111}, recalling that $\alpha \in (0,1)$, we obtain that
\begin{gather}\label{Claim-f-infty}
 \frac{1}{k^{n-1} r^n}\int_{Q^{\nu,k}_r(\frac{r x }{\rho})}\Big|f^\infty( y,\nabla\tilde v_{r}^{\rho,k})-\tfrac{1}{ t_{\rho,k} } 
f( y, t_{\rho,k}\nabla\tilde v_{r}^{\rho,k})\Big|\dy \nonumber\\
 \leq  \frac1{t_{\rho,k}} c_5 (1+ c_4^{1-\alpha})+ \frac{1}{t_{\rho,k}^\alpha}c_5c_3^{1-\alpha} C^{1-\alpha}
 \le \frac{K_1}{t_{\rho,k}^\alpha},
\end{gather}
 where $K_1:= c_5 (1+ c_4^{1-\alpha})+c_5c_3^{1-\alpha} C^{1-\alpha}$.
 
For every $r>0$ and  $\rho\in(0,1)$ we can apply Lemma~\ref{estimate g g0} with $t:= \frac{\rho\, t_{\rho,k}}{r}$. By \eqref{101} and \eqref{121} we obtain
\begin{equation}\label{Claim-g0-2}
 \frac{1}{k^{n-1} r^n}\int_{S_{\smash{\tilde v_{r}^{\rho,k}}}\cap Q^{\nu,k}_r(\frac{r x}{\rho})}
\!\big| g_0( y,[\tilde v_{r}^{\rho,k}],\nu_{\tilde v_{r}^{\rho,k}})-\tfrac{r}{\rho\,  t_{\rho,k} }g(y,\tfrac{\rho\, t_{\rho,k} }{r} [\tilde v_{r}^{\rho,k}],\nu_{\tilde v_{r}^{\rho,k}})\big|\,d\mathcal H^{n-1}
\leq  \lambda (2 M_{\eta,k}\,\rho\,  t_{\rho,k} ) C  
\end{equation}
for every $r>0$ and $\rho\in(0,1)$.

Estimates \eqref{Claim-f-infty} and \eqref{Claim-g0-2}, together with \eqref{ass-h} and \eqref{222}, give 
\begin{gather}\nonumber
 \limsup_{r\to +\infty}\frac{1}{k^{n-1} r^n  t_{\rho,k} } \Big(\int_{Q^{\nu,k}_r(\frac{r x }{\rho})}f( y,  t_{\rho,k}  \nabla\tilde v_{r}^{\rho,k})\dy + \frac{r}{\rho}\int_{S_{\tilde v_{r}^{\rho,k}}\cap Q^{\nu,k}_r(\frac{r x }{\rho})} g( y,\tfrac{\rho\,  t_{\rho,k} }{r} [\tilde v_{r}^{\rho,k}],\nu_{\tilde v_{r}^{\rho,k}})\,d\mathcal H^{n-1}\Big)\\\label{h:stima-1}
 \leq f^\infty_{\rm hom}(a\otimes \nu)+2\eta+  K_1t_{\rho,k}^{-\alpha} +  \lambda (2 M_{\eta,k}\,  \rho\,  t_{\rho,k} ) C.
\end{gather}

Given $k\in \N$ and $\rho \in (0,1)$, for every $\e>0$ and $y\in\R^n$ we define  $ u^{\rho,k}_\e(y):= \e\, t_{\rho,k} \tilde v^{\rho,k}_r (\frac{y}{\e})= \frac{\rho}{r}\, t_{\rho,k}\tilde v^{\rho,k}_r (\frac{r y}{\rho})$, with $r:=\rho/\e$. Then $ u^{\rho,k}_\e\in SBV_{\rm loc}(\R^n,\R^m)$, $ u^{\rho,k}_\e =t_{\rho,k} \ell_{a \otimes \nu}$ near $\partial Q^{\nu,k}_\rho( x  )$ and in $\R^n\setminus Q^{\nu,k}_\rho( x  )$. By a change of variables, from \eqref{h:stima-1} we have 
\begin{equation}\label{scaling-h}
 \limsup_{\e\to 0+}\frac{E_\e(u^{\rho,k}_\e,Q^{\nu,k}_\rho( x ))}{k^{n-1} \rho^n  t_{\rho,k} } \leq f^\infty_{\rm hom}(a\otimes \nu)+2\eta+  K_1t_{\rho,k}^{-\alpha} +  \lambda (2 M_{\eta,k}  \rho\,  t_{\rho,k} ) C.
\end{equation}
Since $ u^{\rho,k}_\e$ coincides with $t_{\rho,k}\ell_{a\otimes \nu}$ in $\R^n\setminus Q^{\nu,k}_\rho(x)$, using
Poincar\'e's inequality and the lower bounds  $(f3)$ and $(g3)$ we deduce from \eqref{scaling-h} that for every $\rho>0$  there exists $\e(\rho)>0$ such that  the functions 
$u^{\rho,k}_\e$ are bounded in $BV_{\rm loc}(\R^n,\R^m)$ uniformly with respect to $\e\in(0,\e(\rho))$. Then there exists a subsequence, not relabelled, of the sequence $(\e_j)$ considered in Theorem \ref{thm:Gamma-conv}, such that 
 $( u^{\rho,k}_{\e_j})$ converges  in $L^1_{\rm loc}(\R^n,\R^m)$  to  some  $ u^{\rho,k} \in BV_{\rm loc}(\R^n,\R^m)$ as $j\to+\infty$  with  $u^{\rho,k}=t_{\rho,k}  \ell_{a \otimes \nu}$  in   
 $Q^{\nu,k}_{(1+\eta)\rho}( x  )\setminus Q^{\nu,k}_{\rho}( x  )$. As a consequence  of the
$\Gamma$-convergence of $E_{\e_j}(\cdot, Q^{\nu,k}_{(1+\eta)\rho}(x))$ to  $\widehat E(\cdot, Q^{\nu,k}_{(1+\eta)\rho}(x))$, from \eqref{scaling-h} we obtain 
\begin{align*}
\frac{m_{\widehat E}(\ell_{t_{\rho,k} a\otimes \nu},Q^{\nu,k}_{(1+\eta)\rho}( x  ))}{k^{n-1}\rho^n\,  t_{\rho,k}} &\leq \frac{\widehat E(u^\rho,Q^{\nu,k}_{(1+\eta)\rho}( x  ))}{k^{n-1}\rho^n\,  t_{\rho,k}}
\leq \limsup_{j\to +\infty}\frac{ E_{\e_j}( u^\rho_{\e_j},Q^{\nu,k}_{(1+\eta)\rho}(x))}{k^{n-1}\rho^n\,  t_{\rho,k}} \\ 
&\leq f^\infty_{\rm hom}(a\otimes \nu)+2\eta+  K_1t_{\rho,k}^{-\alpha} +  \lambda (2 M_{\eta,k}\,  \rho\,  t_{\rho,k} ) C.
\end{align*}
Now, passing to the limit  as $\rho \to 0+$, from \eqref{t-rho} and \eqref{acc} we get
$$
\limsup_{\rho\to 0+}\frac{m_{\widehat E}(\ell_{t_{\rho,k} a\otimes \nu},Q^{\nu,k}_{(1+\eta)\rho}( x  ))}{k^{n-1}\rho^n\,  t_{\rho,k}} \leq
f^\infty_{\rm hom}(a\otimes \nu)+2\eta.
$$
Finally, passing to the limit as $k\to +\infty$, by \eqref{C-deriv-form} we obtain
$$
(1+\eta)^{n}\frac{d\widehat E(u,\cdot)}{d|C(u)|}( x ) \le f^\infty_{\rm hom}(a\otimes \nu) + 2\eta.
$$
Since $a:= a(x)$ and $\nu = \nu(x)$, by \eqref{C-rank-one} this inequality gives \eqref{aqccc} by the arbitrariness of $\eta>0$.

\medskip

\textit{Step 2:} We claim that
\begin{equation}\label{aqscc}
\frac{d\widehat E(u,\cdot)}{d|C(u)|}(x)\geq f^\infty_{\rm hom}\bigg(\frac{dC(u)}{d|C(u)|}(x)\bigg) \quad \textrm{for } |C(u)| \textrm{-a.e. } x\in A.
\end{equation}
We extend $u$ to $\R^n$ by setting $u=0$ on $\R^n\setminus A$. By  $\Gamma$-convergence  there exists $(u_\e)\subset L^1_{\rm loc}(\R^n,\R^m)$,  with ${u_\e}|_A\in SBV(A,\R^m)$,  such that
\begin{equation}\label{h-hom:rec-seq}
u_\e \to u \quad \text{in} \quad L^1_{\rm loc}(\R^n,\R^m) \quad \text{and}\quad \lim_{\e \to 0+}E_\e(u_\e,A)=\widehat E(u,A),
\end{equation}
along the sequence $(\e_j)$ considered in Theorem \ref{thm:Gamma-conv}.

For $|C(u)|$-a.e.\ $x\in A$ there exist $a(x)\in \R^m$ and $\nu(x)\in \Sph^{n-1}$ such that
for every $k\in \N$ we have 
\begin{gather}
\lim_{\rho \to 0+}\frac{Du(Q^{\nu,k}_{\rho}(x))}{|Du|(Q^{\nu,k}_{\rho}(x))}= 
\frac{dC(u)}{d|C(u)|}(x) = a(x)\otimes \nu(x),
\label{(i)}
\\
\lim_{\rho \to 0+}\frac{|Du|(Q^{\nu,k}_{\rho}( x))}{\rho^n}=+\infty,
\label{(ii)}
\\
\lim_{\rho \to 0+}\frac{|Du|(Q^{\nu,k}_{\rho}( x ))}{\rho^{n-1}}=0,
\label{(iii)}
\\
\lim_{\rho \to 0+} \frac{\widehat E(u,Q^{\nu,k}_\rho( x))}{|Du|(Q^{\nu,k}_\rho(x))}=\frac{d \widehat E(u,\cdot)}{d|Du|}(x)<+\infty,
\label{(iv)}
\end{gather}
where \eqref{(i)} follows by \cite[Corollary 3.9]{Al93}, \eqref{(ii)} and \eqref{(iii)} are \EEE consequences \EEE of \cite[Proposition 2.2]{ADM}, while \eqref{(iv)} holds true thanks to a generalised version of the Besicovitch Differentiation Theorem (see \cite{Mor} and \cite[Sections~1.2.1-1.2.2]{FonLeo}).


Let us fix $x\in A$ such that \eqref{(i)}-\eqref{(iv)}  hold true, and set $a:=a(x)$ and $\nu:=\nu(x)$. For $k\in \N$ and $\rho\in (0,1)$ we set 
$$
t_{\rho,k} :=\frac{|Du|(Q^{\nu,k}_\rho( x ))}{k^{n-1}\rho^n}. 
$$
Then 
\begin{equation}\label{aaq}
\frac{d \widehat E(u,\cdot)}{d|C(u)|}(x )= \frac{d \widehat E(u,\cdot)}{d|Du|}( x)
=\lim_{\rho \to 0+} \frac{\widehat E(u,Q^{\nu,k}_\rho(x))}{k^{n-1}\rho^n t_{\rho,k}}.
\end{equation}
Note that, by \eqref{(ii)} and \eqref{(iii)}, the following properies hold: 
\begin{equation}\label{aan}
 \text{for every }k\in\N \quad t_{\rho,k}  \to +\infty\quad\hbox{and}\quad\rho\,  t_{\rho,k} \to {0+} \quad\hbox{as }\rho \to {0+}.
\end{equation}
Recalling \eqref{4.14bis}, we have that
$
\widehat E(u, Q^{\nu,k}_{\rho}(x))\leq c_3|Du|(Q^{\nu,k}_{\rho}(x)) + c_4 k^{n-1}\rho^n
$,
hence by \eqref{aan} there exists $\rho_k\in(0,1)$  such that
\begin{equation}\label{aaa}
Q^{\nu,k}_\rho( x  ) \subset\subset A\quad\hbox{and}\quad\frac{\widehat E(u,Q^{\nu,k}_\rho(x))}{k^{n-1}\rho^n  t_{\rho,k}} \leq  c_3 + 1 \quad  \text{ for every } \rho\in(0,\rho_k).
\end{equation}
Since $\widehat E(u,\cdot)$ is a  Radon measure, there exists a set $ B_k \subset  (0,\rho_k)$, with $(0,\rho_k)\setminus B_k $ at most countable, such that $\widehat E(u,\partial Q^{\nu,k}_\rho( x  ))=0$ for every $\rho\in B_k$. Proceeding as in the proof of \eqref{4.20bis} and \eqref{4.20bis-g}, by \eqref{h-hom:rec-seq}   we can show that for every  $\rho\in B_k$
\begin{equation}\label{4.20bis-h}
\lim_{\e \to 0+}E_\e(u_\e,Q^{\nu,k}_{\rho}( x  ))=\widehat E(u,Q^{\nu,k}_{\rho}( x  )).
\end{equation}
Hence, for every  $\rho\in B_k$ there exists $\e(\rho,k)>0$  such that for every $\e\in (0,\e(\rho,k))$
\begin{equation}\label{aab}
 \frac{E_\e(u_\e,Q^{\nu,k}_\rho( x ))}{k^{n-1}\rho^n \, t_{\rho,k}}\le \frac{\widehat E(u,Q^{\nu, k}_\rho( x))}{k^{n-1}\rho^n \, t_{\rho,k}}+\rho \le c_3
+ 2,
\end{equation}
where in the last inequality we used \eqref{aaa}.

Now, for every $\rho \in B_k $ and $\e\in (0,\e(\rho,k))$ we consider the blow-up functions defined for $y\in Q^{\nu,k}:=Q^{\nu,k}_1(0)$ by  
\begin{align*}
w_\e^{\rho,k}(y)&:=\frac{1}{k^{n-1}\rho\, t_{\rho,k}}\Big(u_\e( x +\rho y)-\frac{1}{k^{n-1}\rho^n}\int_{Q^{\nu,k}_\rho( x)}u_\e(z)\, dz\Big)\\
w^{\rho,k}(y)&:=\frac{1}{k^{n-1}\rho \,  t_{\rho,k} }\Big(u( x +\rho y)-\frac{1}{ k^{n-1}\rho^n}\int_{Q^{\nu,k}_\rho( x )}u( z)\, dz\Big).
\end{align*}
Then $ w_\e^{\rho,k}\in SBV(Q^{\nu,k},\R^m)$ and $w^{\rho,k}\in BV(Q^{\nu,k},\R^m)$. 
Since $u_\e \to u$ in $L^1(Q^{\nu,k}_\rho( x  ), \R^m)$ by \eqref{h-hom:rec-seq}, for every $\rho\in B_k$ 
we obtain 
\begin{equation}\label{h-4.31bis}
w_\e^{\rho,k}  \to w^{\rho,k} \qquad \text{in} \quad L^1(Q^{\nu,k},\R^m)\; \text{ as }\; \e \to {0+}.
\end{equation}
Moreover, recalling \eqref{(i)}, we have that the function $w^{\rho,k}$ satisfies  
$$
\int_{Q^{\nu,k}}w^{\rho,k}(y)\dy=0\quad\hbox{and}\quad
Dw^{\rho,k}(Q^{\nu,k})=\frac{{Du(Q^{\nu,k}_\rho(x))}}{{|Du|(Q^{\nu,k}_\rho( x))}} \;  \to \;\,a \otimes \nu \quad \text{as}\; \rho \to {0+}.
$$
By \cite[Theorem 2.3]{ADM} and \cite[Lemma 5.1]{Larsen} up to a subsequence (not relabelled), 
\begin{equation}\label{aqq-h}
w^{\rho,k} \to w^k \qquad \text{in} \quad L^1(Q^{\nu,k},\R^m)\; \text{ as }\; \rho \to 0+,
\end{equation}
where $w^k\in BV(Q^{\nu,k},\R^m)$ can be represented as
\begin{equation} \label{def:www}
w^k(y)=\psi^k(y \cdot \nu) a,
\end{equation} 
\begin{equation}\label{propr-psi-k0}
\psi^k(\tfrac{1}{2})-\psi^k(-\tfrac{1}{2})=\tfrac{1}{k^{n-1}}\quad \text{and}\quad \int_{-1/2}^{1/2}\psi^k(t)\,dt=0,
\end{equation}
with $\psi_k$ nondecreasing.
By monotonicity these equalities imply that $-\frac1{k^{n-1}}\le \psi^k(-\tfrac{1}{2})\le 0\le \psi^k(\tfrac{1}{2}) \le \frac1{k^{n-1}}$, and so
\begin{equation}\label{propr-psi-k}
|\psi^k(t)|\le \tfrac{1}{k^{n-1}}\quad \text{for every } t\in[-\tfrac{1}{2},\tfrac{1}{2}].
\end{equation}
By  a change of variables we obtain the equality
\begin{equation}
\frac{E_\e(u_\e,Q^{\nu,k}_\rho( x))}{k^{n-1}\rho^n \, t_{\rho,k} } = E_\e^{\rho,k}(w_\e^{\rho,k},Q^{\nu,k}),
\label{aac}
\end{equation}
where $E^{\rho,k}_\e$ is the functional corresponding to the integrands 
$
f_\e^{\rho,k}(y,\xi):=\tfrac{1}{k^{n-1} t_{\rho,k}}f(\tfrac{x +\rho y}{\e}, k^{n-1} t_{\rho,k}\, \xi)$ and $g_\e^{\rho,k}(y,\zeta,\nu):=\tfrac{1}{k^{n-1}\rho\, t_{\rho,k}}g(\tfrac{ x +\rho y}{\e}, k^{n-1}\rho\, t_{\rho,k} \zeta,\nu)
$\ie
\begin{align}\label{def energia rho k}
E_\e^{\rho,k}(w,Q^{\nu,k}):= &\int_{Q^{\nu,k}}\frac{1}{k^{n-1} \,t_{\rho,k} } f(\tfrac{x +\rho y}{\e},k^{n-1} t_{\rho,k} \nabla w(y))\dy \nonumber\\
&+\int_{S_{w}\cap Q^{\nu,k}}\frac{1}{k^{n-1}\rho \, t_{\rho,k} }g(\tfrac{x +\rho y}{\e}, k^{n-1}\rho\, t_{\rho,k}  [w](y),\nu_{w}(y))d\mathcal H^{n-1}(y)
\end{align}
for every $w\in SBV(Q^{\nu,k},\R^m)$. Note that $f_\e^{\rho,k}$ satisfies  $(f3)$ and $(f4)$, the latter with $c_4$ replaced by $c_4/(k^{n-1}t_{\rho,k})$, while  $g_\e^{\rho,k}$ satisfies  $(g3)$ and $(g4)$.

Let $\ell^k$ be the affine function on $\R^n$ which satisfies $\ell^k(y)=\psi^k(\pm\frac12)a$ for
$y\cdot \nu=\pm\frac12$\ie 
\begin{eqnarray}\label{def: function l}
\ell^k(y):= \tfrac{1}{k^{n-1}} \ell_{a \otimes \nu} (y) +\big(\psi^k(\tfrac{1}{2})-\tfrac{1}{2k^{n-1}}\big)a 
= \big(\tfrac{1}{k^{n-1}}{y\cdot \nu} +\psi^{k}(\tfrac{1}{2})-\tfrac{1}{2k^{n-1}}\big)a
\end{eqnarray}
for every $y\in\R^n$. We now want to modify $w_\e^{\rho,k}$ in a way such that it attains the boundary datum
$\ell^k$ near $\partial^\perp_\nu  Q^{\nu,k}$ (see (i) in Section \ref{Notation} and after \eqref{m-phipsi}). This  will be done by applying the Fundamental Estimate \cite[Proposition 3.1]{BDfV} to the functionals $E^{\rho,k}_\e$.
First note that, since by \eqref{aan} $t_{\rho,k} \to +\infty$ as $\rho \to 0+$, we can reduce the value of the constant $\rho_k>0$ introduced in \eqref{aaa} so that $t_{\rho,k}\ge 1$ for every $\rho\in(0,\rho_k)$. 

Therefore,
for $\eta \in (0,\frac12)$ fixed, by slightly modifying the proof of  \cite[Proposition 3.1]{BDfV}  we obtain the following property: for every $k\in \N$, $\rho\in B_k $, and $\e\in(0,\e(\rho,k))$ there exists a cut-off function $\varphi^{\rho,k}_\e \in C^\infty_c(Q^{\nu,k})$, with $0\leq\varphi^{\rho,k}_\e \leq 1$ in $Q^{\nu,k}$, ${\rm supp}(\varphi^{\rho,k}_\e)\subset R^{\nu,k}_{1-\eta}:= R_\nu \big([-\tfrac{k}{2},\tfrac{k}{2}]^{n-1}\times [-\tfrac{1-\eta}{2},\tfrac{1-\eta}{2}]\big)$, and $\varphi^{\rho,k}_\e=1$ in $R^{\nu,k}_{1-2\eta}:=R_\nu \big([-\tfrac{k}{2},\tfrac{k}{2}]^{n-1}\times [-\tfrac{1-2\eta}{2},\tfrac{1-2\eta}{2}]\big)$, such that, setting
$\hat w^{\rho,k}_\e:=\varphi^{\rho,k}_\e w^{\rho,k}_\e+ (1-\varphi^{\rho,k}_\e)\ell^k$ and $S^{\nu,k}_{2\eta}:=Q^{\nu,k}\setminus \overline R^{\nu,k}_{1-2\eta}$, we have 
\begin{equation}\label{aad}
E^{\rho,k}_\e (\hat w^{\rho,k}_\e,Q^{\nu,k}) \leq (1+\eta)\big(E^{\rho,k}_\e ( w_\e^{\rho,k},Q^{\nu,k})
+ E^{\rho,k}_\e (\ell^k, S^{\nu,k}_{2\eta})\big) +  \tfrac{L}{\eta} \| w_\e^{\rho,k}-\ell^k\|_{L^1(S^{\nu,k}_{2\eta},\R^m)},
\end{equation}
where  $L >0$ is independent of $\eta$, $k$, $\rho$, and $\e$. By definition we clearly have $\hat w^\rho_\e=\ell^k$ in  $Q^{\nu,k} \setminus R^{\nu,k}_{1-\eta}$, as desired. Moreover, from 
the bound $f_\e^{\rho,k}( y,\xi)\leq c_3|\xi|+ c_4 /(k^{n-1}t_{\rho,k})$ we obtain the bound 
\begin{align}\label{aqn-h}
E^{\rho,k}_\e (\ell^k, S^{\nu,k}_{2\eta})) &\leq 
\int_{S^{\nu,k}_{2\eta}}(c_3|\nabla\ell^{k}|+\frac{c_4}{k^{n-1} t_{\rho,k}})\dy
\nonumber\\
&= \big( \frac{c_3 |a|}{k^{n-1}}+\frac{c_4}{k^{n-1} t_{\rho,k}} \big)\mathcal L^n(S^{\nu,k}_{2\eta}) 
 =  \big( c_3 |a|+  \frac{c_4}{t_{\rho,k}} \big) 2 \eta.
\end{align}
Hence, from \eqref{aab}, \eqref{aac}, and \eqref{aad}, using also the inequalities $\eta<\frac12$ and $t_{\rho,k}\ge 1$, we obtain 
\begin{align}\label{Energia}
E^{\rho,k}_\e (\hat w^{\rho,k}_\e,Q^{\nu,k})\leq \frac32\big(c_3+2 + c_3 |a|+ c_4\big) +  \tfrac{L}{\eta} \| w_\e^{\rho,k}-\ell^k\|_{L^1( S^{\nu,k}_{2\eta},\R^m)}.
\end{align}
We can estimate the last term in the following way: 
\begin{align} \label{triangolare3}
\|w_\e^{\rho,k}\!-\!\ell^k\|_{L^1(S^{\nu,k}_{2\eta}\!,\R^m)} &\leq \|w_\e^{\rho,k}\!-\!w^{\rho,k}\|_{L^1(Q^{\nu,k}\!,\R^m)}
+\|w^{\rho,k}\!-\!w^{k}\|_{L^1(Q^{\nu,k}\!,\R^m)}+\|w^{ k}\!-\!\ell^k\|_{L^1(S^{\nu,k}_{2\eta}\!,\R^m)}.
\end{align}
By \eqref{def:www} and \eqref{def: function l}, thanks to a change of variables we have
\begin{align*}
 \|w^k-\ell^k\|_{L^1(S^{\nu,k}_{2\eta},\R^m)}
&=  |a|\int_{S^{\nu,k}_{2\eta}}|\psi^k(y \cdot \nu)-  \tfrac{1}{k^{n-1}}y\cdot \nu -\psi^k(\tfrac{1}{2})+\tfrac{1}{2 k^{n-1}}|\dy \nonumber\\
& = |a| k^{n-1}\int_{I_\eta\cup J_\eta} |\psi^k(t)-  \tfrac{1}{k^{n-1}}t -\psi^k(\tfrac{1}{2})+\tfrac{1}{2 k^{n-1}}|\, dt,
\end{align*}
where $I_\eta = (-\frac12, -\frac12+\eta)$ and $J_\eta =  (\frac12-\eta, \frac12)$. 
Hence, by \eqref{propr-psi-k0} and \eqref{propr-psi-k}, by using the continuity of $\psi^k$ at the endpoints, there exists a continuous function $\tau_k:[0,\tfrac12]\to [0, 3]$, with $\tau_k(0)=0$, such that 
$$
k^{n-1} |\psi^k(t)-  \tfrac{1}{k^{n-1}}t -\psi^k(\tfrac{1}{2})+\tfrac{1}{2 k^{n-1}}| \leq \tau_k(\eta) \quad \text{for every }t\in I_\eta\cup J_\eta.
$$
This gives
\begin{equation}\label{wkkkkk}
\|w^k-\ell^k\|_{L^1(S^{\nu,k}_{2\eta},\R^m)} \leq 2|a|\eta \tau_k(\eta).
\end{equation}
In particular, from \eqref{h-4.31bis}, \eqref{aqq-h}, \eqref{triangolare3}, and \eqref{wkkkkk}, possibly reducing the values of the constants $\rho_k>0$ and $\e({\rho,k})>0$, we obtain  
\begin{equation}\label{aee}
 \|w_\e^{\rho,k}-\ell^k\|_{L^1(S^{\nu,k}_{2\eta},\R^m)} \leq  1,
\end{equation}
for every $\rho\in B_k$ and $\e\in (0,\e(\rho,k))$.

Note that $(f3)$, $(g3)$, \eqref{Energia},  and \eqref{aee} imply that  the total variation of 
$\hat w_\e^\rho$ in $Q^{\nu,k}$ is bounded uniformly with respect to $\rho\in B_k $ and 
$\e\in(0,\e(\rho,k))$. Since $\hat w_\e^\rho=\ell^k$ near $\partial^\perp_\nu Q^{\nu,k}$, 
using Poincar\'e's inequality we obtain a uniform bound also for the $L^1$ norm of $\hat w_\e^\rho$ in $Q^{\nu,k}$.
Hence by Lemma \ref{truncation} and Remark \ref{rem:truncation0} there exists a constant $M_{\eta,k}>0$ with the following property: for every  $\rho\in B_k $ and $\e\in(0,\e(\rho,k))$ there exists $\tilde w_\e^{\rho,k}\in SBV(Q^{\nu,k},\R^m)\cap L^\infty(Q^{\nu,k},\R^m)$, with $\tilde w_\e^{\rho,k}=\ell^k$ near $\partial^\perp_\nu  Q^{\nu,k}$, such that 
\begin{equation}\label{aaf}
\|\tilde w_\e^{\rho,k}\|_{L^\infty(Q^{\nu,k},\R^m)}\le M_{\eta,k}\quad\hbox{and}\quad E^{\rho,k}_\e (\tilde w_\e^{\rho,k},Q^{\nu,k})\le E^{\rho,k}_\e (\hat w_\e^{\rho,k},Q^{\nu,k})+\eta.
\end{equation}
We now set $r:=\frac\rho\e$  and
$ v_\e^{\rho,k}( y):=
 r \tilde w_\e^{\rho,k} (\tfrac{y}{r}- \tfrac{x}{\rho})+\tfrac{r}{\rho}\frac{1}{k^{n-1}} \ell_{a \otimes \nu}( x) -r\big(\psi^k(\frac12)-\frac1{2k^{n-1}}\big)a$;  then $v_\e^{\rho,k}\in SBV(Q^{\nu,k}_r(\frac{r x  }{\rho}),\R^m)$, $ v_\e^{\rho,k} = \frac1{k^{n-1}}\ell_{a\otimes\nu}$ near $\partial^\perp_\nu  Q^{\nu,k}_r(\frac{r x  }{\rho})$, and, by a change of variables
\begin{gather}\label{aaah}
\|[v_\e^{\rho,k}]\|_{L^\infty(S_{\smash{v^{\rho,k}_\e}}\cap Q^{\nu,k}_r(\frac{rx}\rho),\R^m)}\le 2 M_{\eta,k}\, r,
\\
\int_{ Q^{\nu,k}_{r}(\tfrac{ rx}{\rho})}f(y,k^{n-1} t_{\rho,k}  \nabla v_\e^{\rho, k})\dy  
+ \frac{1}{\e} \int_{S_{\smash{v_\e^{\rho, k}}}\cap Q^{\nu,k}_{r}(\tfrac{rx}\rho)} g(y, k^{n-1}\e\, t_{\rho,k} [v_\e^{\rho,k}],\nu_{v_\e^{\rho,k}})d\mathcal H^{n-1}
\nonumber\\
=  k^{n-1} r^n t_{\rho,k} E^{\rho,k}_\e (\tilde w_\e^{\rho,k},Q^{\nu,k}).\label{aah}
\end{gather}
Moreover, recalling \eqref{def energia rho k} and combining \eqref{Energia}, \eqref{aee}, \eqref{aaf}, and  \eqref{aah} with the lower bounds $(f3)$ and $(g3)$, we deduce the existence of a constant $C>0$ such that 
\begin{gather}\label{aai} 
\frac{1}{r^n} \int_{Q^{\nu,k}_{r}(\tfrac{ rx}{\rho})}
| \nabla v_\e^{\rho,k}| (y) \, d y \leq  C   \\\label{aaii}
\frac{1}{r^n} \int_{S_{\smash{v_\e^{\rho,k}}}\cap  Q^{\nu,k}_{r}(\tfrac{ rx}{\rho})}|[v_\e^{\rho,k}]| d\mathcal{H}^{n-1} \leq  C.
\end{gather}
for every $\rho\in B_k$ and $\e\in (0,\e({\rho,k}))$. 

By Lemma \ref{estimate f finfty}, applied with $t= k^{n-1} t_{\rho,k}$, using \eqref{aai} and the inequality $t_{\rho,k}\ge 1$ we obtain 
\begin{eqnarray}\label{aaj}\nonumber
& \ds \frac{1}{r^n} \int_{Q^{\nu,k}_{r}(\frac{rx}{\rho})}\Big|f^\infty( y,\nabla v_\e^{\rho,k})-\tfrac{1}{ k^{n-1}t_{\rho,k} } 
f( y,k^{n-1}t_{\rho,k}\nabla v_\e^{\rho,k})\Big|\dy\\
&\ds  \leq  \frac1{t_{\rho,k}} c_5 (1+ c_4^{1-\alpha})+ \frac{1}{t_{\rho,k}^\alpha}c_5c_3^{1-\alpha} C ^{1-\alpha}
 \le \frac{K_1}{t_{\rho,k}^\alpha},
\end{eqnarray}
for every $\rho\in B_k $ and $\e\in (0,\e({\rho,k}))$, where, as in \eqref{Claim-f-infty}, $K_1:= c_5 (1+ c_4^{1-\alpha})+c_5c_3^{1-\alpha}   C ^{1-\alpha}$.

Now note that, since by \eqref{aan} $\rho t_{\rho,k} \to 0+$ as $\rho \to 0+$, we can reduce the value of the constant $\rho_k>0$ introduced in \eqref{aaa} so that $2 M_{\eta,k} k^{n-1} \rho\,t_{\rho,k}\le 1$ for every $\rho\in(0,\rho_k)$.  By Lemma~\ref{estimate g g0}, applied with $t:= k^{n-1}\e \,t_{\rho,k}$, using \eqref{aaah} and \eqref{aaii}, we deduce that 
\begin{eqnarray}
&\ds\frac{1}{r^n}\int_{S_{\smash{v_\e^{\rho,k}}}\cap Q^{\nu,k}_{r}(\tfrac{rx}\rho)}\big| g_0( y,[v_\e^{\rho,k}],\nu_{v_\e^{\rho,k}})-
\tfrac{1}{k^{n-1}\e\,t_{\rho,k} }g(y,k^{n-1}\e\,t_{\rho,k} [v_\e^{\rho,k}],\nu_{v_\e^{\rho,k}})\big|\,d\mathcal H^{n-1}
\nonumber
\\
&\ds\leq \lambda(2 M_{\eta,k} k^{n-1} \rho\,t_{\rho,k}) \frac{1}{r^n}\int_{S_{\smash{v_\e^{\rho,k}}}\cap Q^{\nu,k}_r(\frac{rx }{\rho})}|[v_\e^{\rho,k}]|\,
d\mathcal H^{n-1} \leq \lambda(2 M_{\eta,k} k^{n-1} \rho\,t_{\rho,k}) C, 
\label{aam}
\end{eqnarray}
for every $\rho\in B_k$ and $\e\in (0,\e({\rho,k}))$.  

From \eqref{aac}, \eqref{aad}, \eqref{aqn-h}, \eqref{aee}, \eqref{aaf}, \eqref{aah}, \eqref{aaj}, and \eqref{aam} we obtain 
\begin{align}\label{gg}
 \frac{E^{f^\infty,g_0}(v^{\rho,k}_\e,Q^{\nu,k}_r(\tfrac{r x}\rho))}{r^n}  &\leq (1+\eta) \frac{E_\e(u_\e,Q^{\nu, k}_\rho(x))}{k^{n-1}\rho^n t_{\rho,k}}+  \Big( \frac{3}{2} \big( c_3 |a|+  \frac{2 c_4}{t_{\rho,k}} \big) 
+ 1 \Big)\eta \nonumber \\
 &+  \frac{L}{\eta} \|w_\e^{\rho,k}-\ell^k\|_{L^1(S^{\nu,k}_{2\eta},\R^m)}+ \frac{K_1}{t_{\rho,k}^\alpha}+ \lambda(2 M_{\eta,k} k^{n-1} \rho\,t_{\rho,k}) C.
\end{align}
By the positive $1$-homogeneity 
of $E^{f^{\infty},g_0}$,
thanks to \eqref{min-neumann}, and recalling that $v^{\rho,k}_\e= \frac1{k^{n-1}}\ell_{a\otimes\nu}$ 
near $\partial^\perp_\nu  Q^{\nu,k}_r(\frac{r x  }{\rho})$, 
we have
\begin{equation*} 
\frac{\widetilde m^{f^\infty, g_0}\big(\ell_{a\otimes \nu},Q^{\nu,k}_{ r}\big(\frac{ rx }{\rho}\big)\big)}{k^{n-1}r^n}
\leq \frac{E^{ f^{\infty},g_0}(v^{\rho,k}_\e,Q^{\nu,k}_r(\tfrac{r x}\rho))}{r^n}.
\end{equation*}
Combining this inequality with \eqref{gg}, we obtain
\begin{align*}
\frac{\widetilde m^{f^\infty, g_0}\big(\ell_{a\otimes \nu},Q^{\nu,k}_{ r}\big(\frac{ rx }{\rho}\big)\big)}{k^{n-1}r^n}&\leq (1+\eta) \frac{E_\e(u_\e,Q^{\nu, k}_\rho(x))}{k^{n-1}\rho^n t_{\rho,k}}+ \Big( \frac{3}{2} \big( c_3 |a|+  \frac{2c_4}{t_{\rho,k}}\big) 
+ 1 \Big) \eta\\
&+   \frac{L}{\eta} \|w_\e^{\rho,k}-\ell^k\|_{L^1(S^{\nu,k}_{2\eta},\R^m)}+ \frac{K_1}{t_{\rho,k}^\alpha}+ \lambda(2 M_{\eta,k} k^{n-1} \rho\,t_{\rho,k}) C.
\end{align*}
Passing to the limit as $\e \to 0+$, recalling that $r=\frac{\rho}{\e}$, 
and thanks to \eqref{4.20bis-h} and \eqref{h-4.31bis}, we have 
\begin{align*}
\limsup_{r \to +\infty}
\frac{\widetilde m^{f^\infty, g_0}\big(\ell_{a\otimes \nu},Q^{\nu,k}_{ r}\big(\frac{ rx }{\rho}\big)\big)}{k^{n-1}r^n}
 &\leq (1+\eta) \frac{\widehat E (u ,Q^{\nu, k}_\rho(x))}{k^{n-1}\rho^n t_{\rho,k}}
 +   \Big( \frac{3}{2} \big( c_3 |a|+ \frac{2c_4}{t_{\rho,k}} \big) 
+ 1 \Big)    \eta \nonumber \\
&\hspace{.2cm}+   \frac{L}{\eta}  \|w^{\rho,k}-\ell^k\|_{L^1(S^{\nu,k}_{2\eta},\R^m)}  + \frac{K_1}{t_{\rho,k}^\alpha}+ \lambda(2 M_{\eta,k}\, k^{n-1} \rho\,t_{\rho,k}) C.
\end{align*}
Now, passing to the limit as $\rho \to 0+$, 
by \eqref{acc}, \eqref{aaq}, and \eqref{aan}, and \eqref{aqq-h},
\begin{align*}
\limsup_{r \to +\infty}
\frac{\widetilde m^{f^\infty, g_0}\big(\ell_{a\otimes \nu},Q^{\nu,k}_{ r}\big(\frac{ rx }{\rho}\big)\big)}{k^{n-1}r^n}
&\leq (1+\eta) \frac{d \widehat E(u,\cdot)}{d|C(u)|}(x )
 +    \Big(\frac{3}{2} c_3 |a|+  1 \Big)  \eta 
+  \frac{L}{\eta}  \|w^{k}-\ell^k\|_{L^1(S^{\nu,k}_{2\eta},\R^m)}.
\end{align*}
Passing to the limit as $\eta \to 0+$, by \eqref{wkkkkk} we deduce that 
\begin{align*}
\limsup_{r \to +\infty}
\frac{\widetilde m^{f^\infty, g_0}\big(\ell_{a\otimes \nu},Q^{\nu,k}_{ r}\big(\frac{ rx }{\rho}\big)\big)}{k^{n-1}r^n}
&\leq  \frac{d \widehat E(u,\cdot)}{d|C(u)|}(x ).
\end{align*}
Finally, taking the limit as $k \to +\infty$, by \eqref{h-hom-2} we deduce that 
\begin{align*}
 f^\infty_{\rm hom}(a\otimes \nu) = \lim_{k\to+\infty}
\limsup_{r \to +\infty}
\frac{\widetilde m^{f^\infty, g_0}\big(\ell_{a\otimes \nu},Q^{\nu,k}_{ r}\big(\frac{ rx }{\rho}\big)\big)}{k^{n-1}r^n}
& \leq  \frac{d \widehat E(u,\cdot)}{d|C(u)|}(x ).
\end{align*}
By \eqref{(i)}, this concludes the proof of \eqref{aqscc}, since $a=a(x)$ and $\nu = \nu(x)$.
\end{proof}

We are now in a position to prove the deterministic homogenisation theorem.

\begin{proof}[Proof of Theorem \ref{T:det-hom}]
By Lemma \ref{l:f-hom} the function $f_{\rm hom}$ defined by \eqref{f-hom} belongs to $\mathcal F$ and by Lemma
 \ref{l:g-hom} the function $g_{\rm hom}$ defined by \eqref{g-hom} belongs to $\mathcal G$. By 
 Theorem \ref{thm:Gamma-conv} for  every  sequence  of positive numbers converging to zero, 
 there exist a subsequence $(\e_{j})$ and a functional 
 $\widehat E \colon L^{1}_{\rm loc}(\R^n,\R^m)\times \mathscr{A} \longrightarrow [0,+\infty]$
such that for every $A\in\A$ the functionals $E_{\e_{j}}(\cdot,A)$ $\Gamma$-converge to 
$\widehat E(\cdot,A)$ in $L^1_{\rm loc}(\R^n,\R^m)$, as $j\to +\infty$. 

Let us fix $A\in \mathscr{A}$ and $u\in L^{1}_{\rm loc}(\R^n,\R^m)$. If $u|_A\in BV(A,\R^m)$, then
by Theorem \ref{thm:Gamma-conv}(c) and (d) the
function $\widehat E(u,\cdot)$ is a nonnegative bounded Radon measure on $\B(A)$, which satisfies the inequality
$\widehat E(u,\cdot)\le c_3|Du|+c_4\mathcal L^n$. By the decomposition of the gradient of a $BV$ function (see (f) in 
Section \ref{Notation}), the measure $\widehat E(u,\cdot)$ is absolutely continuous with respect to the measure 
$\mathcal L^n+ \mathcal H^{n-1}\LLL{S_u}+|C(u)|$. Since $\mathcal L^n$, $\mathcal H^{n-1}\LLL{S_u}$, 
and $|C(u)|$ are carried by disjoint Borel sets, by the properties of the Radon-Nikodym derivatives
mentioned in (j) of Section \ref{Notation} we have
\begin{equation*}
\widehat E(u,A)=\int_A  \frac{d \widehat E(u,\cdot)}{d\mathcal L^n} dx +
 \int_{S_u\cap A} \frac{d \widehat E(u,\cdot)}{d\mathcal H^{n-1}\LLL{S_u}} d\mathcal H^{n-1}+ 
 \int_A\frac{d \widehat E(u,\cdot)}{d|C(u)|}d|C(u)|.
\end{equation*}
Using Propositions \ref{p:homo-vol}, \ref{p:homo-sur}, and \ref{p:homo-Can} we obtain
$$
\widehat E(u,A)=\int_A f_{\mathrm{hom}}(\nabla u)\dx + \int_{S_u\cap A}g_{\mathrm{hom}}([u],\nu_u)d \mathcal{H}^{n-1}+\int_A  f_{\rm hom}^\infty\Big(\frac{dC(u)}{d|C(u)|}\Big)\,d|C(u)|.
$$
If $u|_A\notin BV(A,\R^m)$, we have $\widehat E(u,A)=+\infty$ by Theorem \ref{thm:Gamma-conv}(c).
Therefore,
$$
\widehat E(u,A)= E_{\rm hom}(u,A)\quad\text{for every } u\in L^{1}_{\rm loc}(\R^n,\R^m) \text{ and every }A\in \mathscr{A},
$$
where $E_{\rm hom}$ is the functional defined in \eqref{det-Glim}. Since the limit does not depend on the subsequence, by the Urysohn property of $\Gamma$-convergence in $L^1_{\rm loc}(\R^n,\R^m)$ (see \cite[Proposition 8.3]{DM93}) the functionals $E_\e(\cdot,A)$ $\Gamma$-converge to 
$E_{\rm hom}(\cdot,A)$ in $L^1_{\rm loc}(\R^n,\R^m)$, as $\e\to0+$.
\end{proof}


\section{Stochastic homogenisation}\label{section:stoc-hom}

In this section we prove Theorems~\ref{en-density_vs} and \ref{G-convE} concerning stationary random integrands, according to Definition~\ref{def:stationary}.
We adopt the shorthand notation introduced in~\eqref{inf for fixed omega}.

\medskip

We start  by  proving the existence of the  limits which define  the the homogenised random volume integrand $f_{\rm hom}$. 

\begin{prop}[Homogenised random volume integrand]\label{random-f-hom}
Let $f$ be a stationary random volume integrand and let $g$ be a stationary random surface integrand  with respect to a group $(\tau_z)_{z\in \Z^n}$ of $P$-preserving transformations on $(\Om,\T,P)$.
Then there exists $\Om'\in \T$, with $P(\Om')=1$,
such that for every $\om\in \Om'$, $x\in \R^{n}$, $\xi \in \R^{m\times n}$, $\nu\in \Sph^{n-1}$, $k\in \N$, and $\rho>0$ the limit  
\begin{equation}\label{lim-f-hom}
\lim_{r\to +\infty} \frac{m^{f,g_0}_{\om}(\ell_\xi,Q^{\nu,k}_{\rho r}(rx))}{k^{n-1}\rho^n r^{n}}
\end{equation}
exists and is independent of $x$, $\nu$, $k$, and $\rho$. More precisely, there  exists  a random volume integrand $f_{\mathrm{hom}} \colon \Om\times \R^{m\times n} \to [0,+\infty)$ such that for every $\om\in\Om'$, $x \in \mathbb{R}^n$, $\xi \in \mathbb{R}^{m\times n}$, $\nu\in \Sph^{n-1}$, $k\in \N$, and $\rho>0$ 
\begin{equation}\label{lim-f-hom-eq}
f_{\mathrm{hom}}(\om,\xi)=\lim_{r\to +\infty} \frac{m^{f,g_0}_{\om}(\ell_\xi,Q^{\nu,k}_{\rho r}(rx))}{k^{n-1}\rho^n r^n}=  \lim_{ r \to +\infty} \frac{m^{f,g_0}_{\om}(\ell_\xi, Q_r(0))}{r^n}.
\end{equation}
If, in addition, $(\tau_z)_{z\in \Z^n}$ is ergodic, then $f_{\mathrm{hom}}$ is independent of $\om$ and
\begin{equation*}
f_{\mathrm{hom}}(\xi)=\lim_{r\to +\infty}\, \frac{1}{r^n} {\int_\Om m^{f,g_0}_{\om}(\ell_\xi, Q_r(0))\,dP(\om)}.
\end{equation*} 
\end{prop}
\begin{proof}
We divide the proof into four main steps. 

\medskip

\noindent
\emph{Step 1: Existence of the limit in \eqref{lim-f-hom} 
for $\xi\in \Q^{m\times n}$ and $\nu\in\Sph^{n-1}\cap \Q^n$ fixed}.

\smallskip

Let $(\Om, \widehat \T, \widehat P)$ denote the completion of the probability space $(\Om, \T, P)$. 
Let $\xi\in \Q^{m\times n}$ and $\nu\in\Sph^{n-1}\cap \Q^n$ be fixed. 
For every $\om\in \Om$ and $A\in \mI_{n}$ (see \eqref{int:int}) we set
\begin{equation}\label{volume-pro}
\mu_{\xi,\nu}(\om,A):=\frac{1}{{M}^n_\nu} m^{f,g_0}_{\om}(\ell_\xi, M_\nu R_\nu A),
\end{equation} 
where $R_\nu$ is the orthogonal $n\times n$ matrix defined in (h) in Section \ref{Notation}, and $M_\nu$ 
is a positive integer such that $M_\nu R_\nu \in \Z^{n\times n}$.

We now claim that the map $\mu_{\xi, \nu} \colon \Om\times \mI_n \to \R$ defines an $n$-dimensional subadditive process on $(\Om, \widehat \T, \widehat P)$, according to Definition~\ref{Def:subadditive}. 

We start observing that the $\widehat \T$-measurability of $\om \mapsto \mu_{\xi, \nu}(\om,A)$ 
follows from the $\widehat \T$-measurability of $\om \mapsto m^{f,g_0}_{\om}(\ell_\xi,A)$ for every $A\in \A$,  which  is ensured by Proposition \ref{measurability}, taking into account Remark~\ref{DM-rem}. We are now going to prove that $\mu_{\xi, \nu}$ is covariant; that is, we show that there exists a group $(\tau^\nu_{z})_{z\in \Z^n}$ of $\widehat  P$-preserving transformations on $(\Om,\widehat \T, \widehat  P)$ such that
$$
\mu_{\xi, \nu}(\om,A+z)=\mu_{\xi, \nu}(\tau^\nu_z(\om),A), \quad \text{for every $\om\in \Om$, $z\in \Z^n$, and $A\in \mI_n$.}
$$ 
We have 
$$
M_\nu R_\nu (A+z)=M_\nu R_\nu A+M_\nu R_\nu z=M_\nu R_\nu A+ z^\nu,
$$ 
where $z^\nu:=M_\nu R_\nu z \in \Z^n$. Then by \eqref{volume-pro} we get
\begin{equation*}
\mu_{\xi, \nu}(\om,A+z)=\frac{1}{{M}^n_\nu} m^{f,g_0}_{\om}(\ell_\xi, M_\nu R_\nu A+z^\nu).
\end{equation*}
Given $u\in SBV(\mathrm{int} ( M_\nu R_\nu A+z^\nu),\R^m)$ with $u = \ell_{\xi}$ near $\partial (M_\nu R_\nu A + z^{\nu})$,
let $v \in SBV(\mathrm{int}( M_\nu R_\nu A),\R^m)$  be defined as 
$v(x):=u(x+z^\nu) - \xi z^{\nu}$ for every $x\in \R^n$. By a change of variables, using the stationarity of $f$ and 
$g_0$ we obtain 
\begin{align*}
&\int_{M_\nu R_\nu A+ z^\nu}f(\om,x,\nabla u)\dx+ \int_{S_u\cap (M_\nu R_\nu A+ z^\nu)}g_0(\om,x,[u],\nu_u)\,d\mathcal H^{n-1}
\\
&=\int_{M_\nu R_\nu A}f(\om,x+z_\nu,\nabla v)\dx+ \int_{S_v\cap (M_\nu R_\nu A)}g_0(\om,x+z_\nu,[ v],\nu_{ v})\,d\mathcal H^{n-1}
\\
&= \int_{M_\nu R_\nu A}f(\tau_{z^\nu}(\om),x,\nabla v)\dx+ \int_{S_v\cap (M_\nu R_\nu A)}g_0(\tau_{z^\nu}(\om),x,[v],\nu_v)\,d\mathcal H^{n-1}.
\end{align*}
Since we have $v = \ell_\xi $ near $\partial (M_\nu R_\nu A)$, we deduce that 
$$
m^{f,g_0}_{\om}(\ell_\xi, M_\nu R_\nu A+z^\nu)=m^{f,g_0}_{\tau_{z^\nu}(\om)}(\ell_\xi, M_\nu R_\nu A),
$$
and hence the covariance of $\mu_{\xi,\nu}$ with respect to the group of $\widehat P$-preserving transformations $(\tau^\nu_z)_{z\in \Z^n}:=(\tau_{z^\nu})_{z\in \Z^n}$. 

We now show that $\mu_{\xi,\nu}$ is subadditive. To this end let $A\in \mI_n$ and let $(A_i)_{i=1,\ldots, N} \subset \mI_n$ be a finite family of pairwise disjoint sets such that $A=\bigcup_{i=1}^N A_i$. For fixed $\eta>0$ and $i=1,\ldots, N$, let $u_i \in SBV(\mathrm{int}(M_\nu R_\nu A_i),\R^m)$,  with $u_i=\ell_\xi$ near $\partial(M_\nu R_\nu A_i)$,  be such that 
$$
\int_{M_\nu R_\nu A_i}f(\om,x,\nabla u_i)\dx+ \int_{S_{u_i}\cap (M_\nu R_\nu A_i)}g_0(\om,x,[u_i],\nu_{u_i})\,d\mathcal H^{n-1} \leq m^{f,g_0}_{\om}(\ell_\xi, M_\nu R_\nu A_i)+\eta
$$
and on $M_\nu R_\nu A$ define $u(x):=u_i(x)$ if $x\in M_\nu R_\nu A_i$ for $i=1,\ldots, N$. By construction we have that $u$ is a  competitor  for $m^{f,g_0}_{\om}(\ell_\xi, M_\nu R_\nu A)$, since 
$u\in SBV(M_\nu R_\nu A,\R^m)$ and $u=\ell_\xi$ near $\partial(M_\nu R_\nu A)$. Moreover $S_u \cap \partial(M_\nu R_\nu A_i)=\emptyset$  for every $i=1,\ldots, N$. Therefore it holds
\begin{eqnarray*}
 m^{f,g_0}_{\om}(\ell_\xi, M_\nu R_\nu A)&\leq& \int_{M_\nu R_\nu A}f(\om,x,\nabla u)\dx+ \int_{S_{u}\cap (M_\nu R_\nu A)}g_0(\om,x,[u],\nu_{u})\,d\mathcal H^{n-1}
\\
&=& \sum_{i=1}^N \Big(\int_{M_\nu R_\nu A_i}f(\om,x,\nabla u_i)\dx+ \int_{S_{u_i}\cap (M_\nu R_\nu A_i)}g_0(\om,x,[u_i],\nu_{u_i})\,d\mathcal H^{n-1}\Big)
\\
&\leq& \sum_{i=1}^N m^{f,g_0}_{\om}(\ell_\xi, M_\nu R_\nu A_i)+ N \eta,
\end{eqnarray*}
thus the subadditivity of $\mu_{\xi,\nu}$ follows from \eqref{volume-pro}, by the arbitrariness of $\eta>0$.
Note that the same proof shows that 
\begin{equation} \label{aaabbb}
m^{f,g_0}_{\om} \Big(\ell_\xi, \bigcup_{i=1}^N A_i \Big)
\leq \sum_{i=1}^N m^{f,g_0}_{\om} (\ell_\xi,  A_i )
\quad \text{ for every finite disjoint family }  (A_i)_{i=1,\ldots, N} \subset \mI_n,
\end{equation}
even when $\bigcup_{i=1}^N A_i \notin \mI_n$. 

Eventually, in view of $(f4)$ we have
\begin{equation}\label{DM-2315}
\mu_{\xi,\nu}(\om,A) = \frac{1}{M_\nu^n}m^{f,g_0}_{\om}(\ell_\xi, M_\nu R_\nu A)
\leq \frac{1}{M_\nu^n} \int_{M_\nu R_\nu A}f(\om,x,\xi)\dx
\leq (c_3|\xi|+c_4)\mathcal L^n(A),
\end{equation}
for  every $\om \in \Om$. 

We note now that for every $x\in \R^n$, $k\in \N$, and $\rho > 0$, 
we have that $(  Q_{ \rho r}^{ e_n,k}(r x) )_{r>0}$ is a regular family in $\mI_n$ (cf.\ Definition \ref{reg-fam}). 
Therefore for every fixed $\xi\in \Q^n$ and $\nu\in \Sph^{n-1}\cap \Q^n$ we can apply Theorem~\ref{ergodic} to the subadditive process $\mu_{\xi,\nu}$ on $(\Om, \widehat \T, \widehat P)$ to deduce the existence of a $\widehat \T$-measurable function $\varphi_{\xi,\nu} \colon \Om \to [0,+\infty)$ and a set $\widehat \Om_{\xi, \nu}\subset \Om$, with $\widehat \Om_{\xi, \nu}\in \widehat \T$ and $P(\widehat \Om_{\xi, \nu}) =1$ such that
\begin{equation}\label{ergodic-xi-nu}
\lim_{r \to +\infty} \frac{\mu_{\xi,\nu}(\om, Q^{e_n,k}_{\rho r}(rx))}{k^{n-1}\rho^n r^n}
=\varphi_{\xi,\nu}(\om),
\end{equation}
for every $\om \in \widehat\Om_{\xi,\nu}$, $x\in \R^n$, $k\in\N$, and $\rho>0$. Then, by the properties of the completion there exist a set $\Om_{\xi, \nu}\in \T$, with $P(\Om_{\xi, \nu}) =1$, and a $\T$-measurable function, which we still denote by $\varphi_{\xi,\nu}$, such that \eqref{ergodic-xi-nu} holds for every $\om \in \Om_{\xi, \nu}$.
Thus choosing in \eqref{ergodic-xi-nu} $x=0$, $k=1$, and $\rho=1$,  thanks to \eqref{volume-pro} we get \begin{equation*}
\varphi_{\xi,\nu}(\om)= \lim_{r \to +\infty} \frac{\mu_{\xi,\nu}(\om, Q_r(0))}{r^n}= \lim_{r \to +\infty} \frac{m^{f,g_0}_{\om}(\ell_\xi, Q^\nu_r(0))}{r^n}.
\end{equation*}
Furthermore, if $(\tau_z)_{z\in \Z^n}$ is ergodic,  then Theorem \ref{ergodic} ensures that $\varphi_{\xi,\nu}$ is independent of $\om$.

\medskip

\noindent
\emph{Step 2: Existence of the limit in \eqref{lim-f-hom} 
for every $\xi\in \R^{m\times n}$ and $\nu\in\Sph^{n-1}$}.

\smallskip

Let $\widetilde \Om$ denote the intersection of the sets $\Om_{\xi, \nu}$ for $\xi\in \Q^n$ and $\nu\in \Sph^{n-1}\cap \Q^n$; clearly $\widetilde \Om \in \T$ and $P(\widetilde \Om)=1$.
 For every $k\in\N$ and $\rho>0$ we  now introduce the auxiliary functions $\underline f{}_{\rho, k} , \overline f{}_{\rho, k} \colon \widetilde \Om\times \R^n \times \R^{m\times n} \times \Sph^{n-1}\to [0,+\infty)$ defined as
\begin{equation*}
\underline f{}_{\rho, k}(\om, x ,\xi,\nu):= \liminf_{r \to +\infty} \frac{m^{f,g_0}_{\om}(\ell_\xi, Q^{\nu,k}_{\rho r}( rx))}{k^{n-1}\rho^n r^n}
\end{equation*}
\begin{equation*}
\overline f{}_{\rho, k}(\om, x ,\xi,\nu):= \limsup_{r \to +\infty} \frac{m^{f,g_0}_{\om}(\ell_\xi, Q^{\nu,k}_{\rho r}( rx))}{ k^{n-1} \rho^n r^n}.
\end{equation*}
We notice that 
\begin{equation}\label{DM-2317}
\underline f{}_{\rho, k}(\om, x ,\xi,\nu)=\overline f{}_{\rho, k}(\om, x ,\xi,\nu)=\varphi_{\xi,\nu}(\om)
\end{equation}
for every $\om \in \widetilde \Om$, $\xi\in \Q^{m\times n}$, $\nu \in \Sph^{n-1}\cap \Q^n$, $k\in\N$, and $\rho > 0$. The proof of
property \eqref{aea} in Lemma \ref{l:f-hom} can be adapted to the rectangles $Q^{\nu,k}_{\rho r}(rx)$, obtaining that
  for every $\om \in \widetilde \Om$, $x\in\R^n$,  $\nu \in \Sph^{n-1}$, $k\in \N$, and $\rho>0$, the functions $\xi \mapsto \underline f{}_{\rho, k}(\om, x ,\xi,\nu)$ and $\xi \mapsto \overline f{}_{\rho, k}(\om, x ,\xi,\nu)$ are continuous on $\R^{m\times n}$, and their modulus of continuity does not depend on $\om$, $x$,  $\nu$, $k$, and $\rho$. By \eqref{DM-2317} this implies that  for every $\xi\in \R^{m \times n}$ and $\nu \in \Sph^{n-1}\cap \Q^n$ there exists a
 $\T$-measurable function, which we still denote by $\varphi_{\xi,\nu}$, such that  
\begin{equation}\label{aha}
\underline f{}_{\rho, k}(\om, x ,\xi,\nu)= \overline f{}_{\rho, k}(\om, x ,\xi,\nu)=\varphi_{\xi,\nu}(\om)
\end{equation}
 for every $\om \in \widetilde \Om$,  $x\in\R^n$,  $k\in\N$, and $\rho>0$.

We now show that, for every $\om \in \widetilde \Om$, $x\in\R^n$,  $\xi \in \R^{m \times n}$, $k\in\N$, and $\rho>0$, 
the functions $\nu \mapsto \underline f{}_{\rho, k}(\om, x ,\xi,\nu)$ and $\nu \mapsto \overline f{}_{\rho, k}(\om, x ,\xi,\nu)$, restricted to  $\widehat{\Sph}^{n-1}_+$ and $\widehat{\Sph}^{n-1}_-$, are continuous.
We will only prove  this property for  $\underline f{}_{\rho, k}$ and $\widehat{\Sph}^{n-1}_+$, 
the other proofs being analogous. 
To this end, let $\om \in \widetilde \Om$, $x\in\R^n$,  $\xi \in \R^{m \times n}$, $k\in \N$, and $\rho>0$ be fixed. Let $\nu \in \widehat{\Sph}^{n-1}_+$ and 
let $(\nu_j) \subset \widehat{\Sph}^{n-1}_+$  be such that $\nu_j \to \nu$ as $j \to +\infty$.
Since $\nu \mapsto R_{\nu}$ is continuous in $\widehat{\Sph}^{n-1}_+$, for every $\delta \in (0, 1/2)$
there exists an integer $\hat\jmath$, depending on $\rho$, $k$, and $\delta$, such that 
\[
Q^{\nu_j,k}_{(1 - \delta)\rho r} ( rx) \subset \subset Q^{\nu,k} _{\rho r}( rx) 
\subset \subset Q^{\nu_j,k} _{(1 + \delta)\rho r}( rx)
\]
for every $j \geq\hat\jmath$ and $r > 0$. Given $r > 0$, $j \geq \hat\jmath$,  and $\eta > 0$, let 
 $u \in SBV (Q^{\nu,k}_{\rho r} ( rx), \R^m)$ be such that $u = \ell_{\xi}$ near $\partial Q^{\nu,k}_{\rho r} ( rx)$
and 
$$
\int_{Q^{\nu,k}_{\rho r}( rx)}f(\om, y,\nabla u)\dy+ \int_{S_{u}\cap Q^{\nu,k}_{\rho r} ( rx)}
g_0(\om, y,[u],\nu_{u})\,d\mathcal H^{n-1} \leq m^{f,g_0}_{\om}(\ell_\xi, Q^{\nu,k}_{\rho r} ( rx)) + \eta k^{n-1} \rho^n r^n.  
$$
We define now $v \in SBV (Q^{\nu_j,k}_{(1 + \delta)\rho r} ( rx), \R^m)$ as
\[
v (y) =
\begin{cases}
u (y) & \text{ if } y \in Q^{\nu,k}_{\rho r} ( rx), \\
\ell_{\xi} (y) & \text{ if } y \in Q^{\nu_j,k}_{(1 + \delta)\rho r} ( rx) \setminus Q^{\nu,k}_{\rho r} ( rx).
\end{cases}
\]
Note that $v = \ell_{\xi}$ near $\partial Q^{\nu_j,k}_{(1 + \delta)\rho r} ( rx)$
and that $\partial Q^{\nu,k}_{\rho r} ( rx) \cap S_v = \emptyset$.
Therefore, 
\begin{align*}
&m^{f,g_0}_{\om}(\ell_\xi, Q^{\nu_j,k}_{(1 + \delta)\rho r} ( rx)) \\
&\leq \int_{Q^{\nu_j,k}_{(1 + \delta)\rho r} ( rx) } f(\om, y,\nabla v)  \dy 
+ \int_{S_{v}\cap Q^{\nu_j,k}_{(1 + \delta)\rho r} ( rx)}
g_0(\om, y,[v],\nu_{v})\,d\mathcal H^{n-1} \\
&\leq \int_{ Q^{\nu,k}_{\rho r}( rx)}f(\om, y,\nabla u)\dx+ \int_{S_{u}\cap Q^{\nu,k}_{\rho r} ( rx)}
g_0(\om, y,[u],\nu_{u})\,d\mathcal H^{n-1}
+ (c_3|\xi|+c_4) \big( (1 + \delta)^n  - 1 \big) k^{n-1}\rho^n r^n \\
&\leq m^{f,g_0}_{\om}(\ell_\xi,  Q^{\nu,k}_{\rho r}( rx)) + \eta k^{n-1} \rho^n r^n
+ (c_3|\xi|+c_4) \big( (1 + \delta)^n  - 1 \big) k^{n-1}\rho^n r^n.
\end{align*}
Dividing by $ k^{n-1}\rho^n r^n$ and passing to the liminf as $r \to +\infty$ we obtain 
\begin{align*}
(1 + \delta)^n \underline f{}_{\rho, k}(\om, x ,\xi,\nu_j) 
\leq \underline f{}_{\rho, k}(\om, x ,\xi,\nu) + \eta +  (c_3|\xi|+c_4) \big( (1 + \delta)^n  - 1 \big).
\end{align*}
Passing to the limsup  first  as $j \to +\infty$,  then as $\delta\to0+$ and $\eta\to 0+$ we get  
\[
\limsup_{j \to +\infty} \underline f{}_{\rho, k}(\om, x ,\xi,\nu_j) 
\leq \underline f{}_{\rho, k}(\om, x ,\xi,\nu). 
\]
A similar argument, using the  cubes  $Q^{\nu_j,k}_{(1 - \delta)\rho r} ( rx)$, gives
\[
\underline f{}_{\rho, k}(\om, x ,\xi,\nu) \leq \liminf_{j \to +\infty} \underline f(\om, x ,\xi,\nu_j),
\]
and so the continuity of $\nu \mapsto \underline f{}_{\rho, k}(\om, x ,\xi,\nu)$ follows.

It is known that $\mathbb{Q}^n\cap \Sph^{n-1}$ is dense in $\Sph^{n-1}$(see, e.g., \cite[Remark A.2]{CDMSZ}). Arguing as in the proof of \cite[Theorem 5.1]{CDMSZ} it is easy  to  show that 
$\mathbb{Q}^n\cap \widehat\Sph^{n-1}_\pm$ is dense in $\widehat\Sph^{n-1}_\pm$. 
Therefore, from  the continuity property proved above and from \eqref{aha} we deduce  
 that for every $\xi\in \R^{m \times n}$ and $\nu \in \Sph^{n-1}$ there exists a
 $\T$-measurable function, which we still denote by $\varphi_{\xi,\nu}$, such that 
\begin{equation} \label{agc}
\underline f{}_{\rho, k}(\om, x ,\xi,\nu)= \overline f{}_{\rho, k}(\om, x ,\xi,\nu)=\varphi_{\xi,\nu}(\om)
\end{equation}
 for every $\om \in \widetilde \Om$,  $x\in\R^n$,  $\xi\in \R^{m\times n}$, $\nu \in \Sph^{n-1}$, $k\in\N$, and $\rho>0$. This implies  that
 \begin{equation} \label{this one}
 \varphi_{\xi, \nu} (\omega)= \lim_{r \to +\infty} \frac{m^{f,g_0}_{\om}(\ell_\xi, Q^{\nu,k}_{\rho r}( rx))}{ k^{n-1} \rho^n r^n}
\end{equation}
for every $\om \in \widetilde \Om$,  $x\in\R^n$,  $\xi\in \R^{m\times n}$, $\nu \in \Sph^{n-1}$, $k\in\N$, and $\rho>0$,
concluding the proof of  Step~2.

\medskip

\noindent
\emph{Step 3: The limit in \eqref{lim-f-hom} 
is independent of $\nu$}. 

\smallskip

We now show that $\varphi_{\xi, \nu} (\omega)$ does not depend on $\nu$\ie 
we show that 
 \begin{equation} \label{agb}
\varphi_{\xi, \nu} (\omega) = \varphi_{\xi, e_n} (\omega)
\quad \text{ for every } \omega \in \widetilde \Om,\; \xi \in \R^{m \times n},\; \nu \in \Sph^{n-1}.
\end{equation}
For every $r>0$ let $Q^\nu_r:=Q^\nu_r(0)$ and let $\eta > 0$ be fixed. 
Let $( Q_{\rho_i} (x_i) )$ be a family of pairwise disjoint cubes, 
with $\rho_i\in (0,1)$, $i = 1, \ldots, N_{\eta}$,
with faces parallel to the coordinate axes, such that
\[
\bigcup_{i= 1}^{N_{\eta}} Q_{\rho_i} (x_i) \subset Q^\nu_r \quad
\text{ and } \quad 
\mathcal{L}^n \bigg( Q^{\nu}_{1} \setminus \bigcup_{i= 1}^{N_{\eta}} Q_{\rho_i} (x_i) \bigg)
= 1 - \sum_{i= 1}^{N_{\eta}} \rho_i^n <  \eta.
\]
By using arguments similar to those used to prove the subadditivity in Step 1, we can prove that
\[
m^{f,g_0}_{\om}(\ell_\xi, Q^{\nu}_{ r} )
\leq m^{f,g_0}_{\om}\bigg(\ell_\xi, \bigcup_{i= 1}^{N_{\eta}}  Q_{\rho_i r} ( r x_i)\bigg)
+ \eta (c_3 |\xi| + c_4 ) r^n.
\]
Then, by \eqref{aaabbb}, for every $r>0$ we have 
\begin{align*}
\frac{m^{f,g_0}_{\om}(\ell_\xi, Q^{\nu}_{ r})}{r^n}
&\leq \sum_{i = 1}^{N_{\eta}} \frac{m^{f,g_0}_{\om}(\ell_\xi,  Q_{\rho_i r} ( r x_i))}{r^n} 
+ \eta (c_3 |\xi| + c_4 )\\
&=  \sum_{i = 1}^{N_{\eta}} \frac{m^{f,g_0}_{\om}(\ell_\xi, Q_{\rho_i r} ( r x_i))}{\rho_i^n r^n } \rho_i^n + \eta (c_3 |\xi| + c_4 ).
\end{align*}
Hence, passing to the limit as $r \to +\infty$ and using \eqref{this one} we obtain
\begin{align*}
\varphi_{\xi, \nu} (\omega)
\leq  \varphi_{\xi, e_n} (\om) \sum_{i = 1}^{N_{\eta}} \rho_i^n
+ \eta (c_3 |\xi| + c_4 )
\leq \varphi_{\xi, e_n} (\om) + \eta (c_3 |\xi| + c_4 ),
\end{align*}
thus taking the limit as $\eta \to 0+$ we get $\varphi_{\xi, \nu} (\omega) \leq \varphi_{\xi, e_n} (\om)$.
By repeating a similar argument, now using coverings of $Q_{1}(0)$ 
by cubes of the form $Q^\nu_{\rho_i}(x_i)$, we obtain the opposite inequality, 
and eventually the claim.

\medskip

\noindent
\emph{Step 4: Definition and properties of $f_{\rm hom}$}.

\smallskip

For every $\om \in \Om$  and $\xi \in \R^{m \times n}$ we define
\begin{equation*}
f_{\rm hom}(\om, \xi):=\begin{cases}
\varphi_{\xi, e_n} (\omega) & \text{if}\; \om\in \widetilde \Om,  
\cr
c_2|\xi| & \text{if}\; \om\in \Om \setminus \widetilde \Om. 
\end{cases}
\end{equation*}
Then \eqref{lim-f-hom-eq} follows from \eqref{this one} and \eqref{agb}. From the measurability of
$\varphi_{\xi, e_n}$, proved in Step 2, we obtain that $f_{\rm hom}(\cdot, \xi)$
 is $\T$-measurable in $\Om$ for every $\xi\in \R^{m \times n}$. Moreover, since the function $\xi \mapsto \underline f{}_{\rho, k}(\om, x ,\xi,\nu)$
is continuous on $\R^{m\times n}$, from \eqref{agc} we deduce that $f_{\rm hom}(\om,\cdot)$ 
is continuous in $\R^{m \times n}$ for every $\om \in\Om$, 
and this implies the $\T \otimes \B^{m\times n}$-measurability of $f_{\rm hom}$ on $\Om\times \R^{m \times n}$. Finally, Lemma \ref{l:f-hom} allows us to conclude that $f_{\rm hom}(\om, \cdot) \in \mathcal F$ for every $\om \in \Om$. 
Therefore, $f_{\mathrm{hom}}$ is a random volume integrand according to Definition~\ref{ri}. 
\end{proof}

The following result is a direct consequence of Propositions~\ref{p:h} and~\ref{random-f-hom}. 
In the ergodic case, \eqref{DM-hhh} can be obtained by integrating \eqref{hhh} and observing that,
thanks to \eqref{DM-2315}, we can apply the Dominated Convergence Theorem.


\begin{prop}[Homogenised random Cantor integrand]\label{random-f-hom-infty} 
Under the assumptions of Proposition  \ref{random-f-hom}, for  every $\om\in \Om'$ and $\xi \in \R^{m\times n}$  let 
$$
f^\infty_{\rm hom}(\omega,\xi):=\lim_{t\to+\infty}\frac{f_{\rm hom}(\omega,t\xi)}{t}
$$
(since $f_{\mathrm{hom}}(\om,\cdot)\in \mathcal{F}$, the existence of the limit is guaranteed by $(f5)$). 
Then  $f^\infty_{\rm hom}$ is a random volume integrand and  for every 
$\om\in \Om'$, $x\in \R^{n}$, $\xi \in \R^{m\times n}$, $\nu\in \Sph^{n-1}$, and $k\in \N$ we have 
\begin{equation}\label{hhh}
f^\infty_{\rm hom}(\omega,\xi)=\lim_{r\to +\infty} \frac{m^{f^\infty,g_0}_{\om}(\ell_\xi,Q^{\nu,k}_r(rx))}{k^{n-1}r^{n}}=\lim_{r\to +\infty} \frac{m^{f^\infty,g_0}_{\om}(\ell_\xi, Q_r)}{r^{n}},
\end{equation}
  where $Q_r:=Q_r (0)$. 
If, in addition, $(\tau_z)_{z\in \Z^n}$ is ergodic, then $f^\infty_{\rm hom}$ is independent of $\om$ and 
\begin{equation}\label{DM-hhh}
f^\infty_{\mathrm{hom}}(\xi)=\lim_{r\to +\infty}\, \frac{1}{r^n} {\int_\Om m^{f^\infty,g_0}_{\om}(\ell_\xi,  Q_r  )\,dP(\om)}.
\end{equation} 
\end{prop}

%

The following proposition establishes the existence of the random surface integrand $g_{\rm hom}$. 

\begin{prop}[Homogenised random surface integrand]\label{random-g-hom}
Let $f$ be a stationary random volume integrand and let $g$ be a stationary random surface integrand  with respect to a group $(\tau_z)_{z\in \Z^n}$ of $P$-preserving transformations on $(\Om,\T,P)$.
Then there exists $\Om'\in \T$, with $P(\Om')=1$,
such that for every $\om\in \Om'$, $x\in \R^{n}$, $\zeta \in \R^{m}$, $\nu\in \Sph^{n-1}$, the limit  
\begin{equation}\label{lim-g-hom}
\lim_{r\to +\infty} \frac{m^{f^{\infty},g}_{\om}(u_{ r x , \zeta, \nu},Q^{\nu}_r(rx))}{r^{n-1}}
\end{equation}
exists and is independent of $x$. More precisely, there exists a random volume integrand 
$g_{\mathrm{hom}} \colon \Om\times \R^{m} \times \Sph^{n-1} \to [0,+\infty)$ 
such that for every $\om\in\Om'$, $x \in \mathbb{R}^n$, 
$\zeta \in \mathbb{R}^{m}$, and $\nu\in \Sph^{n-1}$ 
\begin{equation*}
g_{\mathrm{hom}}(\om,\zeta, \nu)
=\lim_{r\to +\infty} 
\frac{m^{f^{\infty},g}_{\om}(u_{ r x , \zeta, \nu},Q^{\nu}_r(rx))}{r^{n-1}}
=\lim_{r\to +\infty} 
\frac{m^{f^{\infty},g}_{\om}(u_{ 0, \zeta, \nu},  Q^{\nu}_r  )}{r^{n-1}},
\end{equation*}
 where $Q^{\nu}_r:=Q^{\nu}_r (0)$.
If, in addition, $(\tau_z)_{z\in \Z^n}$ is ergodic, then $g_{\mathrm{hom}}$ is independent of $\om$ and
\begin{equation*}
g_{\mathrm{hom}}(\zeta, \nu) =\lim_{r\to +\infty}\, \frac{1}{r^{n-1}} 
{\int_\Om m^{f^{\infty},g}_{\om}(u_{ 0, \zeta, \nu},  Q^{\nu}_r  ) \,dP(\om)}.
\end{equation*} 
\end{prop}

The proof of Proposition \ref{random-g-hom} follows immediately from Propositions~\ref{random-g-hom-0}
and~\ref{random-g-hom-x} below. In the first one we  prove 
the existence of the limit in \eqref{lim-g-hom} for $x=0$, while in the second one we consider the 
general case
$x\neq 0$ and prove that the limit is independent of~$x$.

\begin{prop}\label{random-g-hom-0}
Let $f$ be a stationary random volume integrand and let $g$ be a stationary random surface integrand  with respect to a group $(\tau_z)_{z\in \Z^n}$ of $P$-preserving transformations on $(\Om,\T,P)$.
Then there exist $\widetilde\Om\in \T$, with $P(\widetilde\Om)=1$, and a random surface integrand $g_{\rm hom}\colon \Om \times \R^m \times \Sph^{n-1} \to \R$ such that 
\begin{equation}\label{lim-in-zero}
g_{\rm hom}(\om,\zeta,\nu)=\lim_{r\to +\infty} 
\frac{m^{f^{\infty},g}_{\om}(u_{ 0, \zeta, \nu}, Q^\nu_r)}{r^{n-1}},
\end{equation}
for every $\om \in \widetilde \Om$, $\zeta\in \R^m$, and $\nu\in \Sph^{n-1}$, where $Q^{\nu}_r:=Q^{\nu}_r (0)$.  
If, in addition, $(\tau_z)_{z\in \Z^n}$ is ergodic, then $g_{\mathrm{hom}}$ is independent of $\om$ and
\begin{equation}\label{g-hom-det}
g_{\mathrm{hom}}(\zeta, \nu) =\lim_{r\to +\infty}\, \frac{1}{r^{n-1}} 
{\int_\Om m^{f^{\infty},g}_{\om}(u_{ 0, \zeta, \nu}, Q^\nu_r) \,dP(\om)}.
\end{equation} 

\end{prop}

\begin{proof} We adapt the proof of \cite[Theorem 5.1]{CDMSZ-stoc}. The main difference is that now the functional to be minimised depends also on $f^\infty$, while in  \cite[Theorem 5.1]{CDMSZ-stoc} it depends only on  $g$. Since this requires some changes, for completeness we prefer to give the whole proof in detail.
We divide it into four steps. 

\medskip

\noindent
\emph{Step 1: Existence of the limit in \eqref{lim-in-zero} 
for  fixed $\zeta \in \Q^{m}$ and $\nu\in\Sph^{n-1}\cap \Q^n$}. 

\smallskip

Let $\nu \in \Sph^{n-1}\cap \mathbb Q^{n-1}$ and $\zeta \in \Q^{m}$ be fixed, 
let $R_\nu \in O(n) \cap \mathbb{Q}^{n \times n}$ 
be the orthogonal $n\times n$ matrix as in  (h) in Section \ref{Notation}, and 
let $M_\nu$ be a positive integer such that $M_\nu R_\nu \in \Z^{n\times n}$. 
Note that, in particular, for every $z'\in \mathbb{Z}^{n-1}$ we have that 
$M_\nu R_\nu(z',0)\in \Pi^\nu_0 \cap \mathbb{Z}^n$, namely $M_\nu R_\nu$ maps integer vectors perpendicular to $e_n$ into integer vectors perpendicular to $\nu$.

Given $A'  = [a_1,b_1)\times\dots\times[a_{n-1},b_{n-1}) \in \mathcal I_{n-1}$ 
(see \eqref{int:int}), we define the (rotated) $n$-dimensional interval $T_\nu(A')$ as 
\begin{equation}\label{An}
T_\nu(A'):= M_\nu R_\nu \big(A'\times[-c, c)\big), \quad \hbox{with } c:=\frac12 \max_{1\leq j\leq n-1}(b_j-a_j).
\end{equation}
For every $\om\in \Om$ and $A' \in \mI_{n-1}$ we set
\begin{equation}\label{surface-pro}
\mu_{\zeta,\nu}(\om,A'):=\frac{1}{{M}^{n-1}_{\nu}} m^{f^{\infty},g}_{\om}(u_{0, \zeta, \nu}, T_\nu(A')).
\end{equation} 

Now let $(\Om, \widehat \T, \widehat P)$ denote the completion of the probability space $(\Om, \T, P)$.
We claim that the function $\mu_{\zeta,\nu} \colon \Om\times \mI_{n -1} \to \R$ as in \eqref{surface-pro} defines an $(n-1)$-dimensional subadditive 
process on $(\Om, \widehat \T, \widehat P)$. 
Indeed, thanks to Remark~\ref{DM-rem} and Proposition~\ref{measurability}, for every $A\in \A$ the function  
$\om \mapsto m^{f^{\infty},g}_{\om}(u_{0, \zeta, \nu}, A)$ is $\widehat \T$-measurable.
From this, it follows that the function   $\om \mapsto \mu_{\zeta,\nu} (\om,A')$ is $\widehat \T$-measurable 
too.

We are now going to prove that $\mu_{\zeta,\nu}$ is covariant; that is, we show that there exists a group 
$(\tau^\nu_{z'})_{z' \in \Z^{n-1}}$ of $\widehat P$-preserving transformations on $(\Om,\widehat \T,\widehat P)$ such that
$$
\mu_{\zeta,\nu} (\om,A'+z')=\mu_{\zeta,\nu}(\tau^\nu_{z'}(\om),A'), \quad \text{for every $\om\in \Om$, 
$z' \in \Z^{n-1}$, and $A' \in \mI_{n-1}$.}
$$ 
To this end fix ${z'}\in \mathbb{Z}^{n-1}$ and $A'\in \mathcal{I}_{n-1}$. Note that, by \eqref{An},
\begin{align*}
T_\nu(A'+{z'}) &= M_\nu R_{\nu}((A'+{z'})\times [-c, c)) = M_\nu R_\nu (A'\times [-c, c)) + M_\nu R_\nu({z'},0)= T_\nu(A') +z'_\nu,
\end{align*}
where $z'_\nu:= M_\nu R_\nu({z'},0)\in \Pi^\nu_0 \cap \Z^n$. 
Then, by \eqref{surface-pro}
\begin{equation}\label{eq:mu zeta nu}
\mu_{\zeta,\nu} (\omega,A'+{z'}) = 
\frac{1}{{M}^{n-1}_{\nu}} m^{f^{\infty},g}_{\om}(u_{0, \zeta, \nu}, T_\nu(A' + z'))
= \frac{1}{{M}^{n-1}_{\nu}} m^{f^{\infty},g}_{\om}(u_{0, \zeta, \nu}, T_\nu(A') + z'_{\nu}).
\end{equation}

Given $u\in SBV (\mathrm{int} (T_\nu(A')+z'_\nu),\R^m)$ with $u = u_{0,\zeta,\nu}$ near $\partial (T_\nu(A')+z'_\nu)$, 
let  $v \in SBV (\mathrm{int}(T_\nu(A')),\R^m)$  be defined as $v(x):=u(x+ z'_\nu)$ for every $x\in \R^n$. By a change of variables, using the stationarity of $f^{\infty}$ and $g$ we obtain 
\begin{align*}
&\int_{T_\nu(A')+z'_\nu}f^{\infty}(\om,x,\nabla u)\dx+ \int_{S_u\cap (T_\nu(A')+z'_\nu)}g(\om,x,[u],\nu_u)\,d\mathcal H^{n-1}
\\
&=\int_{T_\nu(A')}f^{\infty}(\om,x+z'_\nu,\nabla v)\dx+ \int_{S_{ v}\cap T_\nu(A')}g(\om,x+z'_\nu,[ v],\nu_{ v})\,d\mathcal H^{n-1}
\\
&= \int_{T_\nu(A')}f^{\infty}(\tau_{z'_\nu}(\om),x,\nabla v)\dx+ \int_{S_v\cap T_\nu(A')}g(\tau_{z'_\nu}(\om),x,[v],\nu_v)\,d\mathcal H^{n-1}.
\end{align*}
Since $z'_\nu$ is perpendicular to $\nu$, we have $u_{0,\zeta,\nu}(x)=u_{0,\zeta,\nu}(x+ z'_\nu)$ for every $x\in \R^n$.
Therefore, from \eqref{surface-pro} and \eqref{eq:mu zeta nu} we obtain that
\[
\mu_{\zeta,\nu} (\omega, A'+{z'}) = \mu_{\zeta,\nu} (\tau^{\nu}_{z'}(\omega), A'),
\]
where we set
\begin{equation*}
\big(\tau_{z'}^{\nu}\big)_{{z'}\in \Z^{n-1}}:=(\tau_{z'_\nu})_{{z'}\in \Z^{n-1}}.
\end{equation*}

We now show that $\mu_{\zeta,\nu}$ is subadditive in $\mathcal{I}_{n-1}$. To this end let $A' \in\mathcal{I}_{n-1}$ and let $(A'_i)_{1\le i\le N} \subset  \mathcal{I}_{n-1}$ be a finite family of pairwise disjoint sets such that $A' =  \bigcup_{i} A'_i$. For fixed $\eta>0$ and $i=1,\dots, N$, let  
$u_i\in SBV ( \mathrm{int}( T_\nu(A'_i)), \R^m)$ be such that $u_i=u_{0,\zeta,\nu}$ in a neighbourhood of $\partial T_\nu(A'_i)$ and
\begin{eqnarray}\label{q:min:Gi}
\int_{T_\nu(A'_i)}f^{\infty}(\om,x,\nabla u_i)\dx+ \int_{S_u\cap T_\nu(A'_i)}g(\om,x,[u_i],\nu_{u_i})
\,d\mathcal H^{n-1}
\leq m^{f^{\infty},g}_{\om}(u_{0, \zeta, \nu}, T_\nu(A'_i)) + \eta.
\end{eqnarray}
Note that $T_\nu(A')$ can differ from $\bigcup_{i}T_\nu(A'_i)$ 
(see for instance \cite[Figure~2]{CDMSZ-stoc}),  but, by construction, 
we always have $\bigcup_{i} T_\nu(A'_i) \subset T_\nu(A')$.


Now we define
$$
u(x):=\begin{cases}
u_i(x) & \text{if}\; x\in T_\nu(A'_i),\ i=1,\dots,N,
\cr
u_{0,\zeta,\nu}(x) & \text{if}\; x\in T_\nu(A') \setminus \bigcup_{i}T_\nu(A'_i);
\end{cases}
$$
then $u\in SBV (T_\nu(A'), \R^m)$  
and $u=u_{0,\zeta,\nu}$ near $\partial T_\nu(A')$. 
Note that we also have x
\[
S_u\cap T_\nu(A') = \bigcup_{i=1}^N (S_{u_i}\cap T_\nu(A'_i)).
\]
Indeed,  $S_u \cap T_\nu(A') \cap \partial T_\nu(A'_i) = \emptyset$  for every $i=1,\dots,N$. 
Moreover, $u=u_{0,\zeta,\nu}$ in $T_\nu(A') \setminus \bigcup_{i}T_\nu(A'_i)$, hence
$\nabla u=0$ a.e.\ in this set. 
Therefore, recalling that $f^{\infty} (\omega, \cdot, 0) \equiv 0$, we obtain
\begin{align*}
&\int_{T_\nu(A')}f^{\infty}(\om,x,\nabla u)\dx+ \int_{S_u\cap T_\nu(A')}g(\om,x,[u],\nu_{u})
\,d\mathcal H^{n-1}\\
& = \sum_{i=1}^N
\bigg(\int_{T_\nu(A'_i)}f^{\infty}(\om,x,\nabla u_i)\dx+ \int_{S_{ u_i}\cap T_\nu(A'_i)}g(\om,x,[u_i],\nu_{u_i}) 
\,d\mathcal H^{n-1}\bigg). 
\end{align*}
As a consequence, by \eqref{q:min:Gi},
$$
m^{f^{\infty},g}_{\om}(u_{0, \zeta, \nu}, T_\nu(A'))
\leq \sum_{i=1}^N m^{f^{\infty},g}_{\om}(u_{0, \zeta, \nu}, T_\nu(A'_i)) + N \eta,
$$
thus the subadditivity of $\mu_{\zeta,\nu}$ follows from \eqref{surface-pro}, by the arbitrariness of $\eta>0$.
 
Finally, in view of $(g4)$ for every $A'\in \mathcal I_{n-1}$ and for $\widehat P$-a.e.\ $\om \in \Om$ we have  
\begin{align}\nonumber
\mu_{\zeta,\nu}(\omega,A') &
= \frac{1}{M_\nu^{n-1}} m^{f^{\infty},g}_{\om}(u_{0, \zeta, \nu}, T_\nu(A')) 
\leq \frac{1}{M_\nu^{n-1}}
\int_{S_{u_{0, \zeta, \nu}} \cap T_\nu(A')} g(\om,x,\zeta,\nu)
\,d\mathcal H^{n-1} \\\label{integrabilita-1}
&\leq \frac{c_3 |\zeta|}{M_\nu^{n-1}} \hs^{n-1}(\Pi^\nu_0\cap T_\nu(A'))
 =  c_3 |\zeta| \mathcal{L}^{n-1}(A'),
\end{align}
where we used again the fact that $f^{\infty} (\omega, \cdot, 0) \equiv 0$.
This concludes the proof of the fact that $\mu_{\zeta,\nu}$ is an $(n-1)$-dimensional subadditive process.

We can now apply Theorem \ref{ergodic} to the subadditive process $\mu_{\zeta,\nu}$, defined on $(\Om, \widehat T, \widehat P)$ by \eqref{surface-pro},  to deduce the existence of a $\widehat \T$-measurable function 
$\psi_{\zeta,\nu} \colon \Om \to [0,+\infty)$ and a set $\widehat \Om_{\zeta, \nu}\subset \Om$, 
with $\widehat \Om_{\zeta, \nu}\in \widehat \T$ and $P(\widehat \Om_{\zeta, \nu}) =1$ such that
\begin{equation}\label{aka}
\lim_{r \to +\infty} \frac{\mu_{\zeta,\nu}(\om, rQ')}{r^{n-1}}
=\psi_{\zeta,\nu}(\om)
\end{equation}
for every $\om \in \widehat \Om_{\zeta,\nu}$, where $Q':=[-\frac{1}{2},\frac{1}{2})^{n-1}$. Then, by the properties of the completion, there exist a set 
$\Om_{\zeta, \nu} \in \T$, with $P(\Om_{\zeta, \nu}) =1$, and a $\T$-measurable function, which we still denote by $\psi_{\zeta,\nu}$, 
such that \eqref{aka}  holds for every $\om \in \Om_{\zeta, \nu}$. 
Using the definition of $\mu_{\zeta,\nu}$ we then have
\begin{equation}\label{ergodic-zeta-nu}
\psi_{\zeta,\nu}(\om)=  \lim_{r \to +\infty} \frac{m^{f^\infty,g}_{\om}(u_{0,\zeta,\nu}, Q^\nu_r)}{r^{n-1}}
\end{equation}
for every $\om \in \Om_{\zeta, \nu}$.
\medskip

\noindent
\emph{Step 2: Existence of the limit in \eqref{lim-in-zero} 
for every $\zeta \in \R^{m}$ and $\nu\in\Sph^{n-1}$}. 

\smallskip

Let $\widetilde \Om$ denote the intersection of the sets $\Om_{\zeta, \nu}$ for $\zeta \in \Q^m$ and $\nu\in \Sph^{n-1}\cap \Q^n$; clearly $\widetilde \Om \in \T$ and $P(\widetilde \Om)=1$. 
Let $\underline g , \overline g \colon \widetilde \Om \times \R^{m} \times \Sph^{n-1}\to [0,+\infty]$ 
be the functions defined 
as
\begin{align}\label{C:g-ubar}
&\underline{g}(\om,\zeta,\nu)
:= \liminf_{r\to +\infty} \frac{m_\om^{f^{\infty},g} (u_{ 0 , \zeta, \nu}, Q^\nu_r)}{r^{n-1}},
\\
\label{C:g-bar}
&\overline{g}(\om,\zeta,\nu)
:= \limsup_{r\to +\infty} \frac{m_\om^{f^{\infty},g} (u_{ 0 , \zeta, \nu}, Q^\nu_r)}{r^{n-1}}.
\end{align}
By \eqref{ergodic-zeta-nu} we have 
\begin{equation}\label{DM-7142}
\underline g(\om,\zeta,\nu)=\overline g(\om,\zeta,\nu)=\psi_{\zeta,\nu}(\om)
\quad\hbox{for every }\om \in \widetilde \Om, \, \zeta\in \Q^{m}\hbox{, and }\nu \in \Sph^{n-1}\cap \Q^n.
\end{equation}

By Lemma \ref{l:g-hom} (property \eqref{aek}), for every $\om \in \widetilde \Om$ and $\nu \in \Sph^{n-1}$ 
the functions $\zeta \mapsto \underline g(\om,\zeta,\nu)$ and $\zeta \mapsto \overline g(\om,\zeta,\nu)$ 
are continuous on $\R^{m}$ and their modulus of continuity does not depend on $\om$ and $\nu$.  
More precisely, recalling $(g4)$, for every $\om \in \widetilde \Om$ and $\nu \in \Sph^{n-1}$ we have
\begin{equation} \label{uniform continuity}
\begin{aligned}
&| \underline g(\om,\zeta_1,\nu) - \underline g(\om,\zeta_2,\nu) | 
\leq c_3 \, \sigma_2(|\zeta_1 - \zeta_2|)(| \zeta_1| + | \zeta_2 |), \vspace{.2cm}\\
&| \overline{g}(\om,\zeta_1,\nu) - \overline{g}(\om,\zeta_2,\nu) | 
\leq c_3 \, \sigma_2(|\zeta_1 - \zeta_2|)(| \zeta_1| + | \zeta_2 |),
\end{aligned} 
\quad \text{ for every } \zeta_1, \zeta_2\in \R^m.
\end{equation} 
From these inequalities and from \eqref{DM-7142}
we deduce  
 that for every $\zeta\in \R^m$ and $\nu \in \Sph^{n-1}\cap \Q^n$ there exists a
 $\T$-measurable function, which we still denote by $\psi_{\zeta, \nu}$, such that 
\begin{equation} \label{equality for nu rational}
\underline g(\om,\zeta,\nu)= \overline g(\om,\zeta,\nu)=\psi_{\zeta,\nu}(\om)
\quad\hbox{for every }\om \in \widetilde \Om.
\end{equation}

We now claim that for every $\om \in \widetilde \Om$ and every $\zeta\in\R^m$ the restrictions of the functions $\nu\mapsto \underline{g}(\om,\zeta,\nu)$ and $\nu\mapsto \overline{g}(\om,\zeta,\nu)$ to the sets $\widehat\Sph^{n-1}_+$ and $\widehat\Sph^{n-1}_-$ are continuous.
We only prove this property for  $\underline{g}$ and  $\widehat\Sph^{n-1}_+$, 
the other proofs being analogous.
To this end, let us fix $\zeta\in\R^m$, $\nu\in \widehat\Sph^{n-1}_+$, 
and  a sequence  $(\nu_j) \subset \widehat\Sph^{n-1}_+$  such that $\nu_j \to \nu$ as $j\to +\infty$. 
Since the restriction of the function $\nu\mapsto R_\nu$ to $\widehat\Sph^{n-1}_+$ is continuous, 
for every $\delta \in(0,\frac12)$ there exists an integer $j_\delta$ such that
\begin{equation}\label{cubes-0}
 Q^{\nu_j}_{(1-\delta)r}   \subset\subset   Q^\nu_r   \subset\subset   Q^{\nu_j}_{(1+\delta)r},
\end{equation}
for every $j\geq j_\delta$ and every $r>0$.
Fix $j\geq j_\delta$, $r>0$, and $\eta>0$. 
Let $u\in SBV (Q^\nu_r,\R^m)$  be such that $u=u_{0,\zeta,\nu}$ near  
$\partial Q^\nu_r$ and
\begin{align*}
&\int_{Q^\nu_r}f^{\infty}(\om, x,\nabla u)\dx+ \int_{S_{u}\cap Q^\nu_r}
g(\om, x,[u],\nu_{u})\,d\mathcal H^{n-1} \leq m^{f^{\infty},g}_\om (u_{ 0 , \zeta, \nu}, Q^\nu_r) + \eta r^{n-1}.  
\end{align*}
We define  $v \in SBV ( Q^{\nu_j}_{(1+\delta)r},\R^m)$  as
$$
v( x) := 
\begin{cases}
u(x) \,\, &\textrm{if} \;  x\in Q^\nu_r ,\\
u_{0,\zeta,\nu_j}(x) \,\, &\textrm{if} \;  x\in  Q^{\nu_j}_{(1+\delta)r}  \setminus Q^\nu_r.
\end{cases} 
$$
Then, $v=u_{0,\zeta,\nu}$ near $\partial Q^{\nu_j}_{(1+\delta)r}$
and $S_{v}\subset S_{u}\cup\Sigma$, where 
$$
 \Sigma:=\big\{x\in \partial Q^\nu_r: (x \cdot \nu) ( x \cdot \nu_j) < 0
  \big\}  \cup \big( \Pi^{\nu_j}_0 \cap( Q^{\nu_j}_{(1+\delta)r} \setminus Q^\nu_r) \big).
$$
Moreover $|[v]|\le |\zeta|$ $\mathcal H^{n-1}$-a.e.\ on $\Sigma$. 
By \eqref{cubes-0} there exists $\varsigma_{j} (\delta)>0$, independent of $r$, 
with $\varsigma_j(\delta)\to (1+ \delta)^{n-1} - 1$  as $j \to + \infty$, such that $\hs^{n-1}(\Sigma)\le  \varsigma_j(\delta)r^{n-1}$. Thanks to $(g4)$ we then have 
\begin{align*}
m^{f^{\infty},g}_\om (u_{ 0 , \zeta, \nu}, Q^{\nu_j}_{(1+\delta)r} )
&\leq \int_{Q^{\nu_j}_{(1+\delta)r} }f^{\infty}(\om, x,\nabla v)\dx
+ \int_{S_{v}\cap Q^{\nu_j}_{(1+\delta)r} }
g(\om, x,[v],\nu_{v})\,d\mathcal H^{n-1} \\
&\leq \int_{Q^\nu_r}f^{\infty}(\om, x,\nabla u)\dx+ \int_{S_{u}\cap Q^\nu_r}
g(\om, x,[u],\nu_{u})\,d\mathcal H^{n-1} 
+ c_3  |\zeta|  \varsigma_j(\delta)  r^{n-1} \\
&\leq m^{f^{\infty},g}_\om (u_{ 0 , \zeta, \nu}, Q^\nu_r) + \eta r^{n-1}
+ c_3  |\zeta|  \varsigma_j(\delta)  r^{n-1},
\end{align*}
where we used the fact that $f^{\infty} (\omega, \cdot, 0) \equiv 0$.
Recalling definition \eqref{C:g-ubar}, 
dividing by $r^{n-1}$,  and passing to the liminf as $r\to +\infty$, we obtain
\begin{align}\label{stima-errore-zeta}
\underline {g}(\om, \zeta,\nu_j)(1+\delta)^{n-1} \leq \underline{g}(\om, \zeta,\nu) + \eta+ c_3  |\zeta|  \varsigma_j(\delta).
\end{align}
Letting $j\to +\infty$, then $\delta \to 0+$, and then $\eta\to 0+$, we  deduce that 
\begin{equation*}
\limsup_{j \to +\infty} \underline {g}(\om, \zeta,\nu_j) \leq \underline{g}(\om, \zeta,\nu).
\end{equation*}
An analogous argument,  now using the cubes $Q^{\nu_j}_{(1-\delta)r}$, shows that  
\begin{equation*}
\underline{g}(\om, \zeta,\nu) \leq \liminf_{j \to +\infty} \underline {g}(\om, \zeta,\nu_j),
\end{equation*}
hence the claim follows.  
Note that, together with \eqref{uniform continuity}, this implies that 
for every $\om \in\widetilde \Om$ the restriction of the function 
$(\zeta,\nu) \mapsto \underline g(\om,\zeta,\nu)$
 to $\R^{m}\times \widehat\Sph^{n-1}_\pm$ is continuous, and the same holds true 
 for $(\zeta,\nu) \mapsto \overline{g}(\om,\zeta,\nu)$. 

As we already observed in the proof of Step 2 of Proposition \ref{random-f-hom}, the set 
$\widehat\Sph^{n-1}_\pm\cap \Q^n$ is dense in $\widehat\Sph^{n-1}_\pm$. 
Therefore, from  the continuity property proved above and from \eqref{equality for nu rational} we deduce  
 that for every $\zeta\in \R^m$ and $\nu \in \Sph^{n-1}$ there exists a
 $\T$-measurable function, which we still denote by $\psi_{\zeta, \nu}$, such that
 \begin{equation} \label{akc}
\underline g(\om,\zeta,\nu)= \overline g(\om,\zeta,\nu)=\psi_{\zeta, \nu} (\omega) 
 \quad\hbox{for every }\om \in \widetilde \Om.
\end{equation}
 By \eqref{C:g-ubar} and \eqref{C:g-bar} this  implies  that
 \begin{equation} \label{akb}
 \psi_{\zeta, \nu} (\omega)= \lim_{r \to +\infty} \frac{m^{f^\infty,g}_{\om}(u_{0,\zeta,\nu}, Q^\nu_r)}{r^{n-1}}
\end{equation}
for every $\om \in \widetilde \Om$,  $\zeta\in\R^m$, and $\nu \in \Sph^{n-1}$, concluding the proof of Step~2. 

\medskip

\noindent
\emph{Step 3: Definition and properties of $g_{\rm hom}$}.

\smallskip

For every $\om \in \Om$  and $\zeta \in \R^{m}$, and $\nu\in \Sph^{n-1}$ we define
\begin{equation*}
g_{\rm hom}(\om, \zeta,\nu):=\begin{cases}
\psi_{\zeta, \nu} (\omega) & \text{if}\; \om\in \widetilde \Om,  
\cr
c_2|\zeta| & \text{if}\; \om\in \Om \setminus \widetilde \Om. 
\end{cases}
\end{equation*}
Then \eqref{lim-in-zero} follows from \eqref{akb}.
From the measurability of
$\psi_{\zeta, \nu}$, proved in Step 2, we obtain that $g_{\rm hom}(\cdot, \zeta,\nu)$
 is $\T$-measurable in $\Om$  for every $\zeta\in \R^{m}$ and $\nu \in \Sph^{n-1}$.
 Moreover, since for every $\om \in\widetilde \Om$ the restriction of the function 
$(\zeta,\nu) \mapsto \underline g(\om,\zeta,\nu)$
 to $\R^{m}\times \widehat\Sph^{n-1}_\pm$ is continuous, from \eqref{akc} we deduce that
for every $\om\in\Om$ the restriction of $g_{\rm hom}(\om, \cdot,\cdot)$ to $\R^m\times \widehat\Sph^{n-1}_\pm$  is continuous 
and this implies the $\T \otimes \B^{m\times n} \otimes \B^n_S$-measurability of $g_{\rm hom}$ on $\Om\times \R^{m}\times \Sph^{n-1}$. Finally, Lemma \ref{l:g-hom} allows us to conclude that $g_{\rm hom}(\om, \cdot, \cdot) \in \mathcal G$ for every $\om \in \Om$.  

\medskip

\noindent
\emph{Step 4: In the ergodic case $g_{\rm hom}$ is deterministic}.

\smallskip

Set
$
\widehat{\Omega}:= \bigcap_{z\in \mathbb{Z}^n} \tau_z (\widetilde{\Omega});
$
we clearly have that $\widehat \Om\in \T$, $\widehat \Om\subset \widetilde \Om$, and $\tau_z(\widehat{\Omega})=\widehat{\Omega}$ for every $z\in\Z^n$;  moreover, since $\tau_z$ is a $P$-preserving transformation and $P(\widetilde \Om)=1$, we have $P(\widehat{\Omega})=1$.
We claim that
\begin{equation}\label{trans-inv-g_hom}
g_{\rm hom}(\tau_{z}(\om),\zeta,\nu)= g_{\rm hom}(\om,\zeta,\nu),
\end{equation}
for every $z\in \Z^n$, $\om \in \widehat \Om$, $\zeta \in \R^m$, and $\nu\in \Sph^{n-1}$. 

We start noting that to prove \eqref{trans-inv-g_hom} it is enough to show that 
\begin{equation}\label{c:invariance-enough}
g_{\mathrm{hom}}(\tau_z(\om),\zeta,\nu)\leq g_{\mathrm{hom}}(\om,\zeta,\nu)
\end{equation}
for every $z\in \Z^n$, $\omega\in \widehat{\Omega}$, $\zeta\in \R^m$, and $\nu\in \Sph^{n-1}$.
Indeed, the opposite inequality is obtained by applying  \eqref{c:invariance-enough} with $\om$ replaced by $\tau_z(\om)$ and $z$ replaced by $-z$.

Let  $z\in \mathbb{Z}^n$, $\om\in\widehat{\Omega}$,  $\zeta\in \R^m$,   and  $\nu\in \Sph^{n-1}$ be fixed. 
 For every $r> 3|z|$, let $u_r \in SBV ( Q^\nu_r ,\R^m)$  be such that $u_r =u_{0,\zeta,\nu}$ near $\partial Q^\nu_r$, and
\begin{equation}\label{quasi min t}
\int_{ Q^\nu_r }f^{\infty}(\om,x,\nabla u_r)\dx+ \int_{S_{u_r}\cap Q^\nu_r}g(\om,x,[u_r],\nu_{u_r})\,d\mathcal H^{n-1}
\leq m^{f^{\infty},g}_{\om}(u_{0,\zeta,\nu}, Q^\nu_r)+ 1.
\end{equation}
By the stationarity of $f^\infty$ and $g$, a change of variables gives
\begin{equation}\label{trasl:1}
m^{f^{\infty},g}_{\tau_z (\om)}(u_{0,\zeta,\nu}, Q^\nu_r )=m^{f^{\infty},g}_{\om}(u_{z,\zeta,\nu},Q_{r}^\nu(z)). 
\end{equation}
We now modify $u_r$ to obtain a competitor  for a minimisation problem related to  the right-hand side of \eqref{trasl:1}. 
Noting that $ Q^\nu_r  \subset\subset  Q^\nu_{r+3 |z|}(z)$ we define
$$
v_r (x):= 
\begin{cases}
u_r ( x) &\textrm{if }\; x\in Q^\nu_r ,
\cr
u_{z,\zeta,\nu}( x) &\textrm{if }\;  x\in Q^\nu_{r+3 |z|}(z) \setminus Q^\nu_r.  
\end{cases} 
$$
Clearly $v_r \in SBV (Q^\nu_{r+3 |z|}(z),\R^m)$  and $v_r =u_{z,\zeta,\nu}$ near $\partial Q^\nu_{r+3 |z|}(z)$.
Moreover we notice that $S_{v_r}= S_{u_r}\cup \Sigma_1\cup \Sigma_2$, where 
$$
\Sigma_1:=\big\{ x\in \partial Q^\nu_r: \big( x {\,\cdot\,}\nu\big)\big(( x-z){\,\cdot\,}\nu\big)<0\big\}\quad \text{and}  \quad \Sigma_2:=\Pi^\nu_{z}\cap(Q^{\nu}_{r+3 |z|}(z)\setminus Q^\nu_r).
$$%
Moreover $|[v_r]|=|\zeta|$ $\hs^{n-1}$-a.e.\ on $\Sigma_1\cup \Sigma_2$. Since $3 |z|<r$, we
have $\hs^{n-1}(\Sigma_1)=2 (n-1) |z\cdot \nu| \,r^{n-2}$ and $\hs^{n-1}(\Sigma_2)=(r+3 |z|)^{n-1}-r^{n-1}\le
 3(n-1)|z|(r+3 |z|)^{n-2}< 2^{ n }(n-1) |z|\,r^{n-2}$.
Therefore, using the fact that $f^{\infty} (\om, \cdot, 0) \equiv 0$, thanks to $(g4)$ we have
\begin{multline*}
\int_{Q^{\nu}_{r+3 |z|}(z)}f^{\infty}(\om,x,\nabla v_r)\dx+ \int_{S_{v_r}\cap Q^{\nu}_{r+3 |z|}(z)}g(\om,x,[v_r],\nu_{v_r})\,d\mathcal H^{n-1}
\\
\leq 
\int_{Q^\nu_r}f^{\infty}(\om,x,\nabla u_r)\dx+ \int_{S_{u_r}\cap Q^\nu_r}g(\om,x,[u_r],\nu_{u_r})\,d\mathcal H^{n-1} + M_{\zeta,z}\,r^{n-2},
\end{multline*}
where $M_{\zeta,z}:=c_3 (n-1)  (2+2^{n})  |z||\zeta|$.
This inequality, combined with \eqref{quasi min t} yields 
\begin{equation}\label{c:nun}
m^{f^{\infty},g}_{\om}(u_{z,\zeta,\nu},Q_{r+3 |z|}^\nu(z))\leq m^{f^{\infty},g}_{\om}(u_{0,\zeta,\nu}, Q^\nu_r)+  1  +  M_{\zeta,z}\,r^{n-2}.
\end{equation}
 Recalling that  $\tau_z(\omega)\in\widehat{\Omega}\subset \widetilde\Omega$, by \eqref{lim-in-zero} and \eqref{trasl:1} we get
\begin{align*}
g_{\mathrm{hom}}(\tau_z(\omega),\zeta,\nu) &= \lim_{r \to +\infty} \frac{m^{f^{\infty},g}_{\tau_z (\om)}(u_{0,\zeta,\nu},Q^\nu_r))}{r^{n-1}}
= \lim_{r\to +\infty} \frac{m^{f^{\infty},g}_{\om}(u_{z,\zeta,\nu},Q_{r}^\nu(z))}{r^{n-1}}
\\
&= \lim_{r\to +\infty} \frac{m^{f^{\infty},g}_{\om}(u_{z,\zeta,\nu},Q^\nu_{r+3 |z|}(z))}{r^{n-1}},
\end{align*}
where in the last equality we have used the fact that $r^{n-1}/(r+3 |z|)^{n-1}  \to  1$ as $r\to+\infty$. 
Therefore, dividing all terms of \eqref{c:nun} by $r^{n-1}$ and passing to the limit as $r\to+\infty$,  from  \eqref{lim-in-zero} we obtain the inequality 
$$
g_{\mathrm{hom}}(\tau_z(\omega),\zeta,\nu)\leq g_{\mathrm{hom}}(\omega,\zeta,\nu),
$$
 which proves  \eqref{c:invariance-enough} and hence the claim.
 
If $(\tau_z)_{z\in \Z^n}$ is ergodic we can invoke \cite[Corollary 6.3]{CDMSZ-stoc} to deduce that $g_{\rm hom}$ does not depend on $\om$
and hence is deterministic. 
In this case, \eqref{g-hom-det} can be obtained by integrating \eqref{lim-in-zero} over $\Om$ and
observing that, thanks to \eqref{integrabilita-1}, we can apply the Dominated Convergence Theorem.
\end{proof}

We now prove that the limit \eqref{lim-g-hom} that defines $g_{\rm hom}$
is independent of $x$. More precisely we prove the following result. 
\begin{prop}\label{random-g-hom-x}
Let $f$ be a stationary random volume integrand and let $g$ be a stationary random surface integrand  with respect to a group $(\tau_z)_{z\in \Z^n}$ of $P$-preserving transformations on $(\Om,\T,P)$.
Then there exist $\Om'\in \T$, with $P(\Om')=1$, and a random surface integrand $g_{\rm hom}\colon \Om \times \R^m \times \Sph^{n-1} \to \R$, independent of $x$, such that
\begin{equation}\label{lim-in-x}
g_{\rm hom}(\om,\zeta,\nu)=\lim_{r\to +\infty}\frac{m^{f^{\infty},g}_{\om}(u_{ r x , \zeta, \nu},Q^{\nu}_{r}(rx))}{{r}^{n-1}},
\end{equation}
for every $\om \in \Om'$, $x\in \R^n$, $\zeta \in \R^m$, $\nu\in \Sph^{n-1}$.  
\end{prop}

\begin{proof} The proof closely follows that of \cite[Theorem 6.1]{CDMSZ-stoc}, therefore here we only discuss the main differences with respect to \cite{CDMSZ-stoc}.  

Let $g_{\rm hom}$ be the random surface integrand introduced in Proposition \ref{random-g-hom-0}. Arguing as in the proof of \cite[Theorem 6.1]{CDMSZ-stoc},  we can prove the existence of $\Om'\in \T$, with $P(\Om')=1$, such that \eqref{lim-in-x} holds for every $\om\in \Om'$, $x\in \R^n$, $\zeta \in \R^m$, and $\nu \in \Sph^{n-1} \cap \Q^{n-1}$. Hence, to conclude it remains to show than \eqref{lim-in-x} holds true for every $\nu \in \Sph^{n-1}$.  

To this end, for fixed $\omega \in \Omega'$,  $x \in \R^n$, $\zeta \in \R^m$, 
and $\nu \in \Sph^{n-1}$, we introduce the auxiliary functions
\begin{align}\label{C:g-ubar-x-davvero}
\underline{g}(\omega, x, \zeta,\nu)
&:= \liminf_{r\to +\infty} \frac{m^{f^{\infty},g}_{\omega} (u_{ rx , \zeta, \nu},Q^{\nu}_r(r x))}{r^{n-1}},
\\
\label{C:g-bar-x-davvero}
\overline{g}(\omega, x, \zeta,\nu)&
:= \limsup_{r\to +\infty} \frac{m^{f^{\infty},g}_{\omega} (u_{ rx , \zeta, \nu},Q^{\nu}_r(r x))}{r^{n-1}}.
\end{align}

Let $\nu \in \widehat\Sph^{n-1}_+$ be fixed.  As we already observed in the proof of Step 2 of Proposition \ref{random-f-hom}, the set
$\widehat\Sph^{n-1}_+ \cap \Q^n$ is dense in $\widehat\Sph^{n-1}_+$, hence there exists a sequence 
$( \nu_j ) \subset \widehat\Sph^{n-1}_+ \cap \Q^{n-1}$ such that
$\nu_j \to \nu$ as $j \to \infty$.
We claim that for every $\delta \in (0, 1/2)$ there exists $j_{\delta} \in \mathbb{N}$
such that 
\begin{align} \label{number 1}
(1+\delta)^{n-1} \underline{g} (\omega, \tfrac{x}{1+ \delta}, \zeta,\nu_j ) 
\leq \underline{g}(\omega, x, \zeta,\nu) +   c_3  |\zeta|  \varsigma_j (\delta),
\\
 \label{number 2}
\overline{g}(\omega, x, \zeta,\nu) 
\leq (1- \delta)^{n-1} \overline{g} (
\omega, \tfrac{x}{1 - \delta}, \zeta,\nu_j ) +  c_3  |\zeta|  \varsigma_j (\delta),
\end{align}
for every $j \geq j_{\delta}$, where $\varsigma_j(\delta)$ is such that 
$\varsigma_j(\delta) \to (1+ \delta)^{n-1} - 1$ as $j \to + \infty$.

The proof of \eqref{number 1} and \eqref{number 2}  is similar to that of \eqref{stima-errore-zeta} in Proposition~\ref{random-g-hom-0}.
Thanks to the continuity of the restriction of $\nu\mapsto R_\nu$ to $\widehat\Sph^{n-1}_+$, for every $\delta \in(0,\frac12)$ 
there exists an integer $j_\delta$ such that
\begin{equation}\label{cubes}
Q^{\nu_j}_{(1-\delta)r}(rx)  \subset\subset  Q^\nu_{r}(rx)  \subset\subset  Q^{\nu_j}_{(1+\delta)r}(rx),
\end{equation}
for every $j\geq j_\delta$ and every $r>0$.
Fix $j\geq j_\delta$, $r>0$, and $\eta>0$. 
Let $u\in SBV(Q^\nu_{r}(rx),\R^m)$ be such that $u=u_{rx,\zeta,\nu}$ near  
$\partial Q^\nu_{r}(rx)$ and
$$
\int_{Q^{\nu}_{r} (rx)}f^{\infty}(\om, y,\nabla u)\dy+ \int_{S_{u}\cap Q^{\nu}_{r} (rx)}
g(\om, y,[u],\nu_{u})\,d\mathcal H^{n-1} \leq m^{f^{\infty},g}_{\omega} (u_{ rx , \zeta, \nu},Q^{\nu}_r(rx)) + \eta r^{n-1}.  
$$
We define $v \in SBV(Q^{\nu_j}_{(1+\delta)r}(r x),\R^m)$ as
$$
v(y) := 
\begin{cases}
u(y) \,\, &\textrm{if} \; y\in Q^\nu_{r}(rx),\\
u_{rx,\zeta,\nu_j}(y) \,\, &\textrm{if} \; y\in Q^{\nu_j}_{(1+\delta)r}(r x) \setminus Q^\nu_{r}(rx).
\end{cases} 
$$
Then, $v=u_{r x,\zeta,\nu_j}$ near $\partial Q^{\nu_j}_{(1+\delta)r}(r x)$, 
and $S_{v}\subset S_{u}\cup\Sigma$, where 
$$
 \Sigma:=\big\{y\in \partial Q^\nu_{r}(rx ): ( (y - rx ) \cdot \nu) ( (y - rx) \cdot \nu_j) < 0
  \big\}  \cup \big( \Pi^{\nu_j}_{rx} \cap(Q^{\nu_j}_{(1+\delta)r}(rx )\setminus Q^\nu_{r}(rx))  \big) .
$$
Moreover $|[v]|\le |\zeta|$ $\mathcal H^{n-1}$-a.e.\ on $\Sigma$. 
By \eqref{cubes} there exists $\varsigma_j(\delta)>0$, independent of $r$, 
with $\varsigma_j(\delta)\to (1+ \delta)^{n-1} - 1$  as $j \to + \infty$, such that $\hs^{n-1}(\Sigma)\le  \varsigma_j(\delta)r^{n-1}$. Thanks to $(g4)$ we then have 
\begin{align*}
&m^{f^{\infty},g}_{\omega} (u_{ rx  , \zeta, \nu_j },Q^{\nu_j}_{(1 + \delta) r} (rx)) \\
&\hspace{1cm}\leq \int_{Q^{\nu_j}_{(1 + \delta) r} (rx)}f^{\infty}(\om, x,\nabla v)\dx
+ \int_{S_{v}\cap Q^{\nu_j}_{(1 + \delta) r} (rx)}
g(\om, x,[v],\nu_{v})\,d\mathcal H^{n-1} \\
&\hspace{1cm}\leq \int_{Q^{\nu}_{r} (rx)}f^{\infty}(\om, x,\nabla u)\dx+ \int_{S_{u}\cap Q^{\nu}_{r} (rx)}
g(\om, x,[u],\nu_{u})\,d\mathcal H^{n-1} 
+ c_3  |\zeta|  \varsigma_j(\delta)  r^{n-1} \\
&\hspace{1cm} \leq m^{f^{\infty},g}_{\omega} (u_{ rx , \zeta, \nu},Q^{\nu}_r(rx)) + \eta r^{n-1}
+ c_3  |\zeta|  \varsigma_j(\delta)  r^{n-1},
\end{align*}
where we used the fact that $f^{\infty} (\omega, \cdot, 0) \equiv 0$.
Recalling definition \eqref{C:g-ubar-x-davvero}, 
dividing by $r^{n-1}$, and passing to the liminf as $r\to +\infty$, we obtain 
\begin{align*}
(1+\delta)^{n-1} \underline{g}(\omega, \tfrac{x}{1 + \delta}, \zeta,\nu_j ) 
\leq \underline{g}(\omega, x, \zeta,\nu) + \eta+ c_3  |\zeta|  \varsigma_j (\delta),
\end{align*}
which gives \eqref{number 1} by the arbitrariness of $\eta$. The proof of \eqref{number 2} is analogous.

From \eqref{number 1} and \eqref{number 2} we get
$$
(1+\delta)^{n-1} \underline{g} (\omega, \tfrac{x}{1+ \delta}, \zeta,\nu_j) -c_3  |\zeta| \varsigma_j(\delta)
\leq \underline{g}(\omega, x, \zeta,\nu)
\leq \overline{g}(\omega, x, \zeta,\nu) \leq(1- \delta)^{n-1} \overline{g} (
\omega, \tfrac{x}{1 - \delta}, \zeta,\nu_j) +  c_3  |\zeta| \varsigma_j (\delta)
$$
for every $j \geq j_{\delta}$. Since $\nu_j \in \Sph^{n-1}\cap \Q^n$, and  
\eqref{lim-in-x} holds true for rational directions, we have 
$$
\underline{g} (\omega, \tfrac{x}{1+ \delta}, \zeta,\nu_j ) 
= \overline{g} (\omega, \tfrac{x}{1- \delta}, \zeta,\nu_j ) 
= g_{\rm hom} (\omega, \zeta,\nu_j ).
$$
This, together with the previous inequality, yields 
$$
(1+\delta)^{n-1} g_{\rm hom} (\omega, \zeta,\nu_j ) -c_3  |\zeta| \varsigma_j(\delta)
\leq \underline{g}(\omega, x, \zeta,\nu)
\leq \overline{g}(\omega, x, \zeta,\nu)
\leq(1- \delta)^{n-1} g_{\rm hom} (\omega, \zeta,\nu_j ) + c_3  |\zeta| \varsigma_j (\delta)
$$
for every $j \geq j_{\delta}$.
Hence, taking the liminf as $j \to +\infty$ and then the limit as $\delta \to 0+$, we obtain 
\[
\liminf_{j \to +\infty} g_{\rm hom} (\omega, \zeta,\nu_j )
\leq \underline{g}(\omega, x, \zeta,\nu)
\leq \overline{g}(\omega, x, \zeta,\nu)
\leq \liminf_{j \to +\infty} g_{\rm hom} (\omega, \zeta,\nu_j )
\]
and hence 
\[
\underline{g}(\omega, x, \zeta,\nu)
= \overline{g}(\omega, x, \zeta,\nu)=\liminf_{j \to +\infty} g_{\rm hom} (\omega, \zeta,\nu_j ).
\]
Note that, in particular, all the terms in the above chain of equalities do not depend on $x$.
Then, in view of the definition of $\underline g$ and $\overline g$ (see \eqref{C:g-ubar-x-davvero} and \eqref{C:g-bar-x-davvero}) 
we get that the limit
\[
\lim_{r\to +\infty} \frac{m^{f^{\infty},g}_{\omega} (u_{ rx , \zeta, \nu},Q^{\nu}_r(r x))}{r^{n-1}}
\]
exists and is independent of $x$. Therefore we obtain
\begin{align*}
\lim_{r\to +\infty} \frac{m^{f^{\infty},g}_{\omega} (u_{ x , \zeta, \nu},Q^{\nu}_r(r x))}{r^{n-1}}
= \lim_{r\to +\infty} \frac{m^{f^{\infty},g}_{\omega} (u_{ 0 , \zeta, \nu},Q^{\nu}_r(0))}{r^{n-1}}
= g_{\rm hom}(\om,\zeta,\nu),
\end{align*}
for every $\omega \in \Omega'$, $x\in \R^n$, $\zeta \in \R^m$, and $\nu \in \widehat\Sph^{n-1}_+$. Since the same property holds for
$\nu \in \widehat\Sph^{n-1}_-$, this concludes the proof. 
\end{proof}

We are now in a position to prove the theorem concerning the existence, for $P$-almost every $\om\in\Om$, of the limits which 
define the homogenised integrands.

\begin{proof}[Proof of Theorem~\ref{en-density_vs}]
Property (a), \eqref{psi0}, and \eqref{DM-8401} are proved in Proposition~\ref{random-f-hom}, while property (b),
 \eqref{phi0}, and \eqref{DM-8402} are proved in Proposition~\ref{random-g-hom}.
Equalities \eqref{hhhrrr} and \eqref{DM-8403} coincide with \eqref{hhh} and \eqref{DM-hhh}, which are proved in Proposition~\ref{random-f-hom-infty}.
\end{proof}

We now prove the main result of the paper.

\begin{proof}[Proof of Theorem~\ref{G-convE}] It is enough to apply Theorem~\ref{en-density_vs} together
with the deterministic homogenisation result
in Theorem~\ref{T:det-hom}, applied for fixed $\om\in\Om'$. 
\end{proof}



\section*{Appendix. Measurability issues}
\setcounter{equation}0
\setcounter{thm}0
\renewcommand{\theequation}{A.\arabic{equation}}
\renewcommand{\thethm}{A.\arabic{thm}}

The purpose of this section is to prove the measurability
of the functions defined in \eqref{inf for fixed omega}. This will be done in  
Proposition~\ref{measurability}, which requires some
preliminary results.
\medskip

We  start by introducing  some notation that will be used throughout the proofs. For every $A\in\A$ let $\mathcal{M}_b(A,\R^{m\times n})$
be the Banach space of all $ \R^{m\times n}$-valued 
bounded Radon measures on $A$. This space is identified with the dual of the space $C_0(A,\R^{m\times n})$
of all $\R^{m\times n}$-valued continuous functions on $\overline A$ vanishing on $\partial A$. 
For every $R>0$ we set
\begin{equation*}
\mathcal M_{R,A}^{m\times n}:=\{\mu \in \mathcal M_b(A,\R^{m\times n}) \colon |\mu|(A)\leq R\}, 
\end{equation*}
where $|\mu|$ denotes the variation of $\mu$ with respect to the Euclidean norm in $\R^{m\times n}$. 
On $\mathcal M_{R,A}^{m\times n}$ we consider the topology induced by the weak$^*$ 
topology of $\mathcal{M}_b(A,\R^{m\times n})$, which will be called
the weak$^*$ 
topology on $\mathcal M_{R,A}^{m\times n}$. Since $ \mathcal M_b(A,\R^{m\times n})$
is the dual of a separable 
Banach space, there exists a distance $d_{R,A}^{m\times n}$ on $\mathcal M_{R,A}^{m\times n}$ which induces 
the weak$^*$ 
topology on $\mathcal M_{R,A}^{m\times n}$ (see \cite[Theorem V.5.1]{Dun-Sch}).
Moreover, the metric space  $(\mathcal M_{R,A}^{m\times n}, d_{R,A}^{m\times n})$ is compact by the Banach-Alaoglu Theorem.

\smallskip

For every $\mu \in \mathcal M_b(A,\R^{m\times n})$ the absolutely continuous part of $\mu$ with respect to the
Lebesgue measure $\mathcal L^n$ is denoted with $\mu^a$.  Note that, if 
$\mu\in \mathcal M_{R,A}^{m\times n} $, then 
$\mu^a\in \mathcal M_{R,A}^{m\times n}$.

The following lemma concerns the mesurability properties of the density of $\mu^a$ with respect to
$\mathcal L^n$. 

\begin{lem} \label{gamma}
Let $A\in \A$ and $R>0$.  Then there exists a $\B (A)  \otimes \B (\mathcal M_{R,A}^{m\times n})$-measurable function 
$\gamma\colon A \times \mathcal M_{R,A}^{m\times n} \to \R^{m \times n}$ such that
\begin{gather}
\gamma (\cdot, \mu) \in L^1 (A, \R^{m \times n}) \quad \text{ for every } \mu \in \mathcal M_{R,A}^{m\times n},
\nonumber
\\
\label{l:mua}
\mu^a (B)=\int_B \gamma (x,\mu)\dx \quad \text{for every } \mu \in \mathcal M_{R,A}^{m\times n} \, \text{and }  B\in \B(A). 
\end{gather}
\end{lem}

\begin{proof}
For every $(x,\mu)\in A\times \mathcal M_{R,A}^{m\times n}$ let $\gamma (x,\mu) \in \R^{m \times n}$ be defined as 
\begin{equation*}
\gamma (x,\mu):=
\begin{cases}
\displaystyle \lim_{\rho\to0+}
\frac{\mu(B_{\rho}(x)\cap A)}{\om_n {\rho}^n} 
&\text{ if the limit exists in $\R^{m \times n}$,}  
\\
0 & \text{ otherwise,}
\end{cases}
\end{equation*}
where $\omega_n$ denotes the volume of the unit ball of $\R^n$. From the theory of differentiation of measures (see, e.g., \cite[Theorem 1.155]{FonLeo}), 
for every $\mu \in \mathcal M_{R,A}^{m\times n}$ we have that 
 $\gamma(\cdot, \mu)\in L^1(A, \R^{m\times n})$ and
\begin{equation*}
\mu^a (B)=\int_B \gamma (x,\mu)\dx \quad \text{for every} \quad B\in \B(A),
\end{equation*}
which proves \eqref{l:mua}.

To prove the measurability of the function $\gamma$  it suffices to show that for every $\rho>0$ the function
\begin{equation}\label{l:prima-mappa}
(x,\mu) \mapsto  \mu (B_{\rho} (x)\cap A)
\end{equation}
from $A\times \mathcal M_{R,A}^{m\times n}$ to $\R^{m \times n}$ is $\B(A)\otimes \B(\mathcal M_{R,A}^{m\times n})$-measurable. To this end, for a fixed $\rho>0$ we introduce  an increasing sequence of nonnegative functions 
$(\varphi_j) \subset C_c (\R^n)$ pointwise converging to the characteristic function 
of the open ball $B_{\rho} (0)$, and we observe that
$$
\mu (B_{\rho}(x)\cap A)=\lim_{j \to +\infty}\int_A \varphi_j (y-x)\,d\mu (y),
$$  
by the Monotone Convergence Theorem. 

Let $A_{\rho}:=\{x \in A \colon {\rm dist}(x,\partial A)> \rho\}$. Since for every $j\in\N$ 
the function 
$$
(x,\mu) \mapsto \int_A \varphi_j(y-x)\,d\mu (y)
$$
is continuous on $A_{\rho}\times \mathcal M_{R,A}^{m\times n}$  
(considering on $\mathcal M_{R,A}^{m\times n}$ the weak$^*$ topology), the function 
\eqref{l:prima-mappa} from $A_{\rho}\times \mathcal M_{R,A}^{m\times n}$ to $\R^{m \times n}$
is $\B(A_{\rho})\otimes \B(\mathcal M_{R,A}^{m\times n})$-measurable. 
By the arbitrariness of $\rho>0$ we obtain that the same function considered on 
$A\times \mathcal M_{R,A}^{m\times n}$ is
$\B(A)\otimes \B(\mathcal M_{R,A}^{m\times n})$-measurable.
\end{proof}

To prove the measurability of the map $\mu \mapsto \mu^a$,  
from $\mathcal M_{R,A}^{m\times n}$ to $\mathcal M_{R,A}^{m\times n}$,
we need the following lemma.
\begin{lem}\label{lemma-astratto}
Let $A\in\A$, let $R>0$, let $(Y, \mathcal E)$ be a measurable space, and let $h \colon A\times Y \to \R^{m \times n}$ be a $\B(A)\otimes \mathcal E$-measurable function such that 
$$
\int_A | h (x,y)|\dx \leq R \quad \text{for every $y\in Y$}.
$$
For every $y\in Y$, we define  the $\R^{m \times n}$-valued measure 
$\lambda_y\in \mathcal M_{R,A}^{m\times n}$ as
$$
\lambda_y(B):=\int_B h(x,y)\dx\quad \text{for every}\quad B\in \B(A).
$$
Then the map $y\mapsto \lambda_y$ is measurable from $(Y,\mathcal E)$ to $(\mathcal M_{R,A}^{m\times n},\B(\mathcal M_{R,A}^{m\times n}))$. 
\end{lem}

\begin{proof}
We start by observing that for every $\varphi \in C_c (A,\R^{m\times n})$ the 
scalar function 
\begin{equation} \label{function E measurable}
y \mapsto \int_A \varphi (x){\cdot} d \lambda_y (x) \quad \text{ is $\mathcal E$-measurable,}
\end{equation}
where $\cdot$ denotes the Euclidean scalar product between matrices.
Indeed, by definition we have 
$$
 \int_A \varphi (x) {\cdot} d \lambda_y (x) = 
  \int_A \varphi (x){\cdot} h (x, y) \, d x,
$$
and the measurability with respect to $y$ follows from the Fubini Theorem.

Note now that a basis for the open sets of the space $\mathcal M_{R,A}^{m\times n}$
(endowed with the weak$^*$ topology) is given by the collection of sets
$$
\Big\{ \lambda \in \mathcal M_{R,A}^{m\times n}: \Big| \int_A \varphi_i (x) {\cdot} d \lambda (x) 
- \int_A \varphi_i (x) {\cdot} d \hat\lambda (x) \Big| < \eta 
 \hbox{ for }i=1,\dots l \Big\},
$$
with $\eta > 0$, $\hat\lambda \in \mathcal M_{R,A}^{m\times n}$, $l \in \mathbb{N}$, 
and
$\varphi_1, \ldots, \varphi_l \in C_c(A,\R^{m\times n})$.
By \eqref{function E measurable}, the pre-image of these sets under the function 
$y \mapsto \lambda_y$ belongs to $\mathcal E$.
This implies that this function is measurable from $(Y,\mathcal E)$ to $(\mathcal M_{R,A}^{m\times n},\B(\mathcal M_{R,A}^{m\times n}))$,
since the weak$^*$ topology in $ \mathcal M_{R,A}^{m\times n}$ has a countable basis.
\end{proof}

The following lemma shows the measurable dependence of $\mu^a$ on $\mu$. 
\begin{lem}\label{l:meas-ac-part}
The map $\mu \mapsto \mu^a$ is measurable from $(\mathcal M_{R,A}^{m\times n}, \B(\mathcal M_{R,A}^{m\times n}))$ to $(\mathcal M_{R,A}^{m\times n}, \B(\mathcal M_{R,A}^{m\times n}))$.  
\end{lem}
\begin{proof}
Thanks to \eqref{l:mua}, the conclusion follows from Lemma \ref{lemma-astratto} 
with $(Y,\mathcal E)=(\mathcal M_{R,A}^{m\times n},\B(\mathcal M_{R,A}^{m\times n}))$ and $h = \gamma$.
\end{proof}

Given $A\in\A$, we set 
\begin{equation}\label{DM-3335}
BV^m_{R,A} := \{u \in BV (A,\R^m): 
\|u\|_{L^1(A,\R^m)}\le R
\text{ and } |Du| (A) \leq R \}.
\end{equation}
On $BV^m_{R,A}$ we consider the topology induced by the distance $d^m_{R,A}$ defined by
\begin{equation*}
d^m_{R,A}(u,v):=\|u-v\|_{L^1(A,\R^m)}+ d^{m\times n}_{R,A}(Du,Dv),
\end{equation*}
where $d^{m\times n}_{R,A}$ is the distance on $\mathcal M_{R,A}^{m\times n}$ that metrizes the weak$^*$ topology.

Note that $BV (A, \R^m)$ is the dual of a separable space, and that, when $A$ has Lipschitz boundary, the topology just defined coincides with the
topology induced on $BV^m_{R,A}$
by the weak$^*$ topology of $BV (A, \R^m)$ (see \cite[Remark~3.12]{AFP}). 

The following lemma will be crucial in the proof of Proposition~\ref{measurability}.

\begin{lem}\label{DM-3312}
Assume that $A\in\A$ has Lipschitz boundary. Then
the metric space $(BV^m_{R,A},d^{m\times n}_{R,A})$ is compact.
\end{lem}

\begin{proof}
Let $(u_k)$ be a sequence in $BV^m_{R,A}$.
By \eqref{DM-3335} this sequence is bounded in $BV(A,\R^m)$.
Recalling the compact embedding of $BV(A,\R^m)$ into $L^1(A,\R^m)$ and the compactness of 
$\mathcal M_{R,A}^{m\times n}$,
there exist a subsequence, not relabelled, and a function $u\in BV(A,\R^m)$ such that $u_k\to u$ strongly in
$L^1(A,\R^m)$ and $Du_k\rightharpoonup Du$ weakly$^*$ in $\mathcal{M}_b(A,\R^{m\times n})$. It is easy
to see that $u\in BV^m_{R,A}$ and that $d^{m\times n}_{R,A}(u_k,u)\to0$.
\end{proof}

We now prove the measurability with respect to $(\omega, u)$ of the integral functional corresponding to
a random volume integrand. 

\begin{lem}\label{meas:volume}
Let $A\in\A$ with Lipschitz boundary, 
let $R>0$, and let $f$ be a random volume integrand  as in Definition \ref{ri}. 
Then, the function 
$$
(\omega, u) \longmapsto \int_A f(\omega, x,\nabla u)\dx
$$
from $\Omega \times BV^m_{R,A}$ to $\R$ is $\mathcal{T} \otimes \B (BV^m_{R,A})$-measurable.
\end{lem}

\begin{proof}
Let $\gamma$ be the function introduced in Lemma~\ref{gamma}.
We observe that for every $u \in BV^m_{R,A}$
$$
\gamma (x, Du) = \nabla u (x) \quad \text{ for }\mathcal{L}^n\text{-a.e.\ } x \in A.
$$
Therefore, 
$$
\int_A f (\omega, x,\nabla u)\dx = \int_A  f(\omega, x,\gamma (x, Du))\dx.
$$
We claim that the function $(x, u) \mapsto \gamma (x, Du)$ from 
$A \times BV^m_{R,A}$ to $\R^{m \times n}$ is $\B (A) \otimes \B (BV^m_{R,A})$-measurable.
Indeed, it is the composition of the functions $(x, u) \mapsto (x, Du)$, 
which is continuous from $A \times BV^m_{R,A} $ to $A \times \mathcal M_{R,A}^{m\times n}$, 
and the function $(x, \mu) \mapsto \gamma (x, \mu)$ from 
$A \times \mathcal M_{R,A}^{m\times n}$ to $\R^{m \times n}$, 
which is $\B (A) \otimes \B (BV^m_{R,A})$-measurable, by Lemma~\ref{gamma}.
Therefore, the function  $(\omega, x, u) \mapsto  (\omega, x,\gamma (x, Du))$ is measurable from
$(\Om\times A\times BV^m_{R,A}, \mathcal{T} \otimes \B (A)\otimes \B (BV^m_{R,A}))$ to 
$(\Om\times \R^n\times \R^{m\times n}, \mathcal{T} \otimes \B^n \otimes \B^{m\times n})$.
By the $\mathcal{T} \otimes \B^n \otimes \B^{m\times n}$-measurability of $f$ we deduce that
the function  $(\omega, x, u) \mapsto  f(\omega, x,\gamma (x, Du))$
from $\Omega \times A\times BV^m_{R,A}$ to $\R$ is 
$\mathcal{T} \otimes \B (A) \otimes \B (BV^m_{R,A})$-measurable.
The conclusion then follows from Fubini's Theorem.
\end{proof}

The following two lemmas are used to prove the measurable dependence on $u$ of the
surface integral functional corresponding to a continuous surface integrand.

For every $A\in\A$, $\mu\in \mathcal M_b(A,\R^{m\times n})$, $x\in A$, and $\rho>0$ we set
\begin{equation}\label{DM-theta}
\theta_{A,\rho}(\mu,x):=\frac{\mu(B_\rho(x)\cap A)}{\om_{n-1}\rho^{n-1}},
\end{equation}
where $\omega_{n-1}$ denotes the volume of the unit ball of $\R^{n-1}$.

\begin{lem}\label{l:BV-jump}
Let $A\in\A$ and $u\in BV(A,\R^m)$. Then 
\begin{equation}\label{DM-6100}
\lim_{\rho \to 0+}\theta_{A,\rho}(Du,x)=  \big(  [u](x)\otimes \nu_u(x)  \big)  \chi_{S_u}(x) \quad \text{for $\mathcal H^{n-1}$-a.e.\ $x\in A$},
\end{equation}
where $\chi_{S_u}(x)=1$ if $x\in S_u$ and $\chi_{S_u}(x)=0$ if $x\in A\setminus S_u$.
\end{lem}
\begin{proof}
\emph{Step $1$.} We claim that 
\begin{equation}\label{l:claim-1}
\lim_{\rho \to 0+}\theta_{A,\rho}(Du^a,x)=0 \quad \text{for $\mathcal H^{n-1}$-a.e.\ $x\in A$}.
\end{equation} 
We now recall that, for a positive Radon measure $\mu$ in $A$ and for $d\in \N$, the $d$-dimensional upper density of $\mu$ at $x\in A$ is defined as 
$$
\Theta^{\ast,d}(\mu, x)= \limsup_{\rho\to 0+} \frac{\mu(B_\rho(x)\cap A)}{\omega_d \rho^d},
$$
where $\omega_{d}$ denotes the volume of the unit ball of $\R^{d}$ (see, e.g., \cite[Definition 2.55]{AFP}). To prove \eqref{l:claim-1} it is then sufficient to show that 
\begin{equation}\label{DM-6120}
\Theta^{\ast,n-1}(|D^a u|, x)=0 \quad \text{for $\mathcal H^{n-1}$-a.e.\ $x\in A$}.
\end{equation}
To do so, for any $t>0$ we define the set 
$$
E_t:= \{x\in A: \Theta^{\ast,n-1}(|D^a u|, x)>t\};
$$
note that 
$$
E_t\subset \{x\in A:  \Theta^{\ast,n}(|D^a u|, x) = +\infty\}.
$$
By the Lebesgue Differentiation Theorem we have $\mathcal{L}^n(E_t)=0$
and, since  $|D^a u| << \mathcal{L}^n$, we have $|D^a u| (E_t) = 0$. 

Since $|D^a u|$ is a finite Radon measure, for every $k \in \mathbb{N}$
there exists an open set $A_k \subset A$ with $E_t \subset A_k$ such that
\[
|D^a u| (A_k) < \tfrac{1}{k}.
\]
Thanks to \cite[Section~2.10.19(3) and Section~2.10.6]{Federer} 
this implies that  
$$
t \mathcal{H}^{n-1}(E_t)\leq |D^a u| (A_k) < \tfrac{1}{k} \quad \text{ for every } k \in \mathbb{N}.
$$
Taking the limit as $k \to \infty$, we obtain that
\[
\mathcal{H}^{n-1}(E_t) = 0 \quad \text{ for every  } t > 0.
\]
From this, it follows that
\[
\mathcal{H}^{n-1} \big(  \{x\in A: \Theta^{\ast,n-1}(|D^a u|, x)> 0 \} \big) = 0
\]
and this proves \eqref{DM-6120}, which gives \eqref{l:claim-1}.

\medskip

\emph{Step $2$.} We claim that 
\begin{equation}\label{l:claim-2}
\lim_{\rho \to 0+}\theta_{A,\rho}(C(u),x)=0 \quad \text{for $\mathcal H^{n-1}$-a.e.\  $x\in A$}.
\end{equation} 
As before, it is sufficient to show that 
\begin{equation} \label{intermediate equality n-1}
\Theta^{\ast,n-1}(|C(u)|,x)=0 \quad \text{for $\mathcal H^{n-1}$-a.e.\ $x\in A$}.
\end{equation}
To do so, for any $t>0$ we define the set 
$$
E_t:= \{x\in A: \Theta^{\ast,n-1}(| C(u)|, x)>t\}.
$$
Now, let $K\subset E_t$ be a compact set with  $\mathcal{H}^{n-1}(K)<+\infty$
so that, in particular, $|C( u)|(K)=0$.
Then, by \cite[Section 2.10.19(3) and Section~2.10.6]{Federer} we have that 
$$
t\mathcal{H}^{n-1}(K)\leq | C( u)|(V) \qquad \text{ for every open set $V$ containing } K. 
$$
Since $C( u)$ is a finite Radon measure, taking the infimum of the above inequality
over all open sets $V$ containing $K$ we obtain that 
\[
t\mathcal{H}^{n-1}(K)\leq | C( u)|(K).
\]
Since $|C( u)|(K)=0$, from the above inequality it follows that $\mathcal{H}^{n-1}(K) = 0$.
Using the fact that $E_t$ is a Borel (and hence Suslin) set, by \cite[Corollary 2.10.48]{Federer} we have that 
$$
\mathcal{H}^{n-1}(E_t) = 
\sup\{\mathcal{H}^{n-1}(K): K \textrm{ compact},\ K\subset E_t,\ \mathcal{H}^{n-1}(K)<+\infty\},
$$
and so $\mathcal{H}^{n-1}(E_t) =0$ for every $t>0$, which implies 
\eqref{intermediate equality n-1} and, in turn, \eqref{l:claim-2}.

\medskip

\emph{Step $3$.} We claim that 
\begin{equation}\label{l:claim-3}
\lim_{\rho \to 0+}\theta_{A,\rho}(D^j u,x)=0 \quad \text{for $\mathcal H^{n-1}$-a.e.\ $x\in A\setminus S_u$}.
\end{equation} 
Observe that $|D^j u|(A\setminus S_u)=0$.  By \cite[Section 2.10.19(4)  and Section~2.10.6]{Federer} we have immediately that 
$$
\Theta^{\ast,n-1}(|D^j u|, x)=0 \quad \text{for $\mathcal H^{n-1}$-a.e.\ $x\in A\setminus S_u$},
$$
which implies \eqref{l:claim-3}.

\medskip

\emph{Step $4$.} By Besicovich Derivation Theorem  (see \cite[Theorems 2.22, 2.83, and 3.78]{AFP})  we have that 
$$
\lim_{\rho \to 0+}\theta_{A,\rho}(D^ju,x)=[u](x)\otimes \nu_u(x) \quad \text{for $\mathcal H^{n-1}$-a.e.\ $x\in S_u$}.
$$
Together with the previous steps, this gives \eqref{DM-6100}. 
\end{proof}

\begin{lem}\label{Lemma66}
Let $A\in\A$ and let $g  \colon A\times \R^{m\times n}\to \R$ be a continuous function.  
Assume that there exists $a>0$ such that 
\begin{equation}\label{bounds:tilde}
|g (x,\xi)|\leq a|\xi|
\end{equation}
for every $(x,\xi)\in A\times \R^{m\times n}$. Then for every $u\in BV(A,\R^m)$ 
\begin{equation*}
\lim_{\eta\to 0+}\lim_{\rho \to 0+}\int_A \frac{ g (x,\theta_{A,\rho}(Du,x))}{|\theta_{A,\rho}(Du,x)| \vee \eta}\,d|Du|(x)=
\int_{A\cap S_u} g (x,[u](x)\otimes \nu_u(x))\,d\mathcal{H}^{n-1}(x),
\end{equation*}
where $\theta_{A,\rho}$ is defined in \eqref{DM-theta}.
\end{lem}

\begin{proof}

Let $u\in BV(A,\R^m)$ be fixed. 

\emph{Step 1.} Thanks to \eqref{bounds:tilde} and to the bound 
\begin{equation}\label{DM-3892}
\int_{A\cap S_u}|[u](x)|\, d\mathcal{H}^{n-1}(x)<+\infty,
\end{equation}
the function $x\mapsto g (x,[u](x)\otimes \nu_u(x))$ is  $\mathcal{H}^{n-1}$-integrable on $A\cap S_u$.

\medskip

\emph{Step 2.} We claim that for every $\eta>0$
\begin{equation*}
\lim_{\rho \to 0+}\int_A \frac{ g (x,\theta_{A,\rho}(Du,x))}{|\theta_{A,\rho}(Du,x)| 
\vee \eta}\,d|D^a u+  C(u)|(x)=
0. 
\end{equation*}
By Lemma \ref{l:BV-jump} we have that 
$$
\lim_{\rho\to 0+}\theta_{A,\rho}(Du,x)= 0 \quad \text{for $|D^a u+ C(u)|$-a.e.\ $x\in A$},
$$
and by \eqref{bounds:tilde} we have the inequality
\begin{equation}\label{DM-4912}
 \frac{  |g  (x,\theta_{A,\rho}(Du,x))|}{|\theta_{A,\rho}(Du,x)| \vee 
\eta}\le a.
\end{equation}
The claim then follows from the Dominated Convergence Theorem, since $g$ is continuous,
$ g  (x,0)= 0$, and $|D^a u+  C(u)|$ is a bounded measure.

\medskip

\emph{Step 3.} Recalling (f) in Section~\ref{Notation}, to conclude the proof it is sufficient to show that 
\begin{equation*}
\lim_{\eta\to 0+}\lim_{\rho \to 0+}\int_{A\cap S_u} \frac{g (x,\theta_{A,\rho}(Du,x))}{|\theta_{A,\rho}(Du,x)| \vee \eta}|[u](x)\otimes \nu_u(x)|\,d\mathcal{H}^{n-1}(x)=
\int_{A\cap S_u} g (x,[u](x)\otimes \nu_u(x))\,d\mathcal{H}^{n-1}(x) . 
\end{equation*}
By Lemma \ref{l:BV-jump} and by the continuity of $g$ 
 we have that for every $\eta>0$
\begin{align*}
&\lim_{\rho \to 0+}\int_{A\cap S_u} \frac{ g (x,\theta_{A,\rho}(Du,x))}{|\theta_{A,\rho}(Du,x)| \vee \eta}|[u](x)\otimes \nu_u(x)|\,d\mathcal{H}^{n-1}(x)\\
& =\int_{A\cap S_u} g (x,[u](x)\otimes \nu_u(x))\frac{|[u](x)\otimes \nu_u(x)|}{|[u](x)\otimes \nu_u(x)| \vee \eta} d\mathcal{H}^{n-1}(x),
\end{align*}
where we used the Dominated Convergence Theorem, thanks to \eqref{DM-3892} and \eqref{DM-4912}. Note that  for $\mathcal{H}^{n-1}$-almost every $x\in S_u$
$$
\lim_{\eta\to 0+}\frac{|[u](x)\otimes \nu_u(x)|}{|[u](x)\otimes \nu_u(x)|\vee \eta} 
=\sup_{\eta> 0}\frac{|[u](x)\otimes \nu_u(x)|}{|[u](x)\otimes \nu_u(x)|\vee \eta}= 1,
$$
since $[u](x)\neq 0$.
Thanks to \eqref{bounds:tilde} and \eqref{DM-3892} 
we can apply again the Dominated Convergence Theorem and deduce the claim in the limit $\eta\to 0+$.
\end{proof}

We are now in a position to prove the measurable dependence on $u$ of the integral functional corresponding 
to a continuous surface integrand.

\begin{lem}\label{DM-6490}
Let $A\in\A$ with Lipschitz boundary, let $R>0$, and let $g$ be as in Lemma~\ref{Lemma66}.
Then the function 
$$
u \longmapsto \int_{S_u\cap A} g  (x,[u](x)\otimes \nu_u(x) )\, d\mathcal{H}^{n-1}(x)
$$
from $BV^m_{R,A}$ to $\R$ is $\B ( BV^m_{R,A} )$-measurable.
\end{lem}

\begin{proof} By Lemma \ref{Lemma66} the thesis follows by proving that 
for every $\rho$, $\eta > 0$ the function 
\begin{equation}\label{claim:tilde}
u \mapsto \int_A \frac{g (x,\theta_{A,\rho}(Du,x))}{|\theta_{A,\rho}(Du,x)| \vee \eta}\,d|Du|(x)
\end{equation}
from $BV^m_{R,A}$ to $\R$ is $\B ( BV^m_{R,A} )$-measurable. 
Let $\rho$, $\eta > 0$ be fixed.  First note that, by the  $ \B(A)\otimes \B ( \mathcal M_{R,A}^{m\times n} )$-measurability of \eqref{l:prima-mappa} and the continuity of $g$, the function  
$$
(x,\mu) \mapsto \frac{g  (x,\theta_{A,\rho}(\mu,x))}{|\theta_{A,\rho}(\mu,x)| \vee \eta}
$$
is $\B(A)\otimes \B ( \mathcal M_{R,A}^{m\times n} )$-measurable. Moreover
it is bounded by \eqref{DM-4912}. So by \cite[Corollary A.3]{CDMSZ-stoc}
$$
\mu\mapsto \int_A \frac{g  (x,\theta_{A,\rho}(\mu,x))}{|\theta_{A,\rho}(\mu,x)| \vee \eta}\, d|\mu|(x)
$$
is $\B ( \mathcal M_{R,A}^{m\times n} )$-measurable. Since $u\mapsto Du$ is continuous from $(BV^m_{R,A},d^m_{R,A})$ to $ (\mathcal M_{R,A}^{m\times n},d_{R,A}^{m\times n})$,  
the $\B ( BV^m_{R,A} )$-measurability of \eqref{claim:tilde} follows. 
\end{proof}

We now prove the measurability with respect to $(\omega, u)$ of the integral functional corresponding to
a random surface integrand, with no continuity assumption with respect to~$x$.

\begin{lem}\label{lemma:meas}
Let $A\in\A$ with Lipschitz boundary, let $R>0$, and let $ g\colon \Omega \times A\times \R^{m}\times \Sph^{n-1}\to \R$ be a $\T\otimes \B(A)\otimes \B^m\otimes \B^n_S$-measurable function. Assume that there exists $a>0$ such that
\begin{gather} \label{propr (g6)}
g(\omega,x,\zeta,\nu) = g(\omega,x,- \zeta,- \nu),
\\
 \label{DM-5749}
| g(\omega,x,\zeta,\nu)|\leq a|\zeta|,
\end{gather}
for every $(\omega,x,\zeta,\nu) \in \Omega \times A \times \R^{m} \times \Sph^{n-1}$.
Then the function 
\begin{equation*}
(\omega, u) \longmapsto \int_{S_u\cap A} g (\omega, x,[u](x),\nu_u(x) ) \, d\mathcal{H}^{n-1}(x)
\end{equation*}
from $\Omega \times BV^m_{R,A}$ to $\R$ is $\mathcal{T} \otimes \B (BV^m_{R,A})$-measurable.
\end{lem}

\begin{proof} 
We recall that a matrix $\xi\in \R^{m\times n}$ has rank $\le 1$ if and only if $\xi=\zeta\otimes \nu$ for some
$\zeta\in\R^m$ and $\nu\in\mathbb{S}^{n-1}$, and that the pair $(\zeta, \nu)$ is uniquely determinded by $\xi$, 
up to a change of sign of both terms.
Therefore, thanks to \eqref{propr (g6)}, we can define
a $\T\otimes \B(A)\otimes \B^{m\times n}$-measurable function 
$\tilde g \colon \Omega \times A\times \R^{m\times n}\to \R$ by setting for every
$ (\omega,x,\xi)\in \Omega \times A \times \R^{m\times n}$
$$
\tilde g(\omega, x, \xi) :=  
\begin{cases}
 g(\omega, x, \zeta, \nu)
&\hbox{ if }\xi=\zeta\otimes \nu,\hbox{ with }\zeta\in\R^m\hbox{ and }\nu\in \mathbb{S}^{n-1},
\\
0&\hbox{ if  rank}(\xi)>1.
\end{cases}
$$ 
By \eqref{DM-5749} we have $|\tilde g(\omega, x, \xi)|\le a|\xi|$ for every
$(\omega,x,\xi)\in \Omega \times A \times \R^{m\times n}$.

To prove the thesis it is enough to show that 
\begin{equation}\label{claim:monotone}
(\omega, u) \longmapsto \int_{S_u\cap A}\tilde  g(\omega, x,[u](x)\otimes\nu_u(x) ) \, d\mathcal{H}^{n-1} \quad \textrm{is } \mathcal{T} \otimes \B (BV^m_{R,A}) \textrm{-measurable}.
\end{equation}
Note that the function $\tilde g$ can be written as 
\begin{equation}\label{def:psihat}
\tilde g(\omega,x,\xi) = \hat g (\omega,x,\xi) a|\xi|, 
\end{equation}
where $\hat g $ is 
$\T\otimes \B(A)\otimes \B(\R^{m\times n})$-measurable and satisfies $| \hat g | \leq 1$.

Let $\mathcal{R}$ be the set of all bounded 
$\T\otimes \B(A)\otimes \B^{m\times n}$-measurable
functions 
$ \hat g\colon \Omega \times A\times \R^{m \times n}\to \R$ 
such that the function $ \tilde g $ defined as in \eqref{def:psihat} satisfies the claim \eqref{claim:monotone}. 

In order to conclude the proof, we need to show that  $\mathcal{R}$  contains 
 all bounded $\T\otimes \B(A)\otimes \B^{m\times n}$-measurable functions. 
To prove this property,  note that $\mathcal{R}$ is a vector space  of bounded real-valued functions 
that contains the constants and is closed both 
under uniform convergence and under monotone convergence of uniformly bounded sequences. 
Let $\mathcal{C}$ be the set of all functions 
 $\hat g \colon\Om\times A\times\R^{m\times n}\to\R$ that can be written as 
$$
\hat g (\omega, x,\xi)= \alpha(\omega)\beta(x,\xi),
$$
where $\alpha\colon \Omega\to \R$ is bounded and $\T$-measurable,  and $\beta\colon A\times \R^{m\times n}\to \R$ is bounded  and continuous. Note that $\mathcal C$ is stable under multiplication and that the $\sigma$-algebra generated by $\mathcal{C}$ is $\T\otimes \B(A)\otimes \B^{m\times n}$.

By Lemma~\ref{DM-6490} we have $\mathcal C\subset \mathcal{R}$. Hence the functional form 
of the Monotone Class Theorem (see \cite[Chapter I, Theorem 21]{Dellacherie}), implies that $\mathcal{R}$ contains all bounded $\T\otimes \B(A)\otimes \B^{m\times n}$-measurable functions, and this concludes the proof. 
\end{proof}

We now prove the measurability of the map $u\mapsto D^j u$.

\begin{lem}\label{l:meas-j-part}
Let  $A\in\A$ with Lipschitz boundary and let $R>0$. Then the map 
$$
u\mapsto D^j u
$$
is measurable from $(BV^m_{R,A}, \B(BV^m_{R,A}) )$ to 
$(\mathcal M_{R,A}^{m\times n}, \B(\mathcal M_{R,A}^{m\times n}))$.
\end{lem}
\begin{proof}
As in the proof of Lemma \ref{lemma-astratto}  it is sufficient to show that 
$$
u\longmapsto \int_A \varphi(x){\cdot} d(D^ju)(x) = \int_{A\cap S_u} \varphi(x){\cdot}([u](x)\otimes \nu_u(x) ) \, d\mathcal{H}^{n-1}(x)
$$
from $BV^m_{R,A}$ to $\R$ is $\B(BV^m_{R,A})$-measurable for every $\varphi\in C_c(A, \R^{m\times n})$. To this end we set $ g (x,\zeta,\nu):=\varphi(x){\cdot}(\zeta\otimes \nu)$, and note that $| g(x,\zeta,\nu)|\leq a |\zeta|$, where $a$ is the maximum of $|\varphi|$. Therefore, by Lemma \ref{lemma:meas} it follows that
$$
u \longmapsto \int_{S_u\cap A} g(x,[u](x),\nu_u(x) ) \, d\mathcal{H}^{n-1}(x)
$$
from $BV^m_{R,A}$ to $\R$ is $\B (BV^m_{R,A})$-measurable, and hence the claim.
\end{proof}

The following corollary deals with the Cantor part.

\begin{cor}\label{cor:Cantor}
Let  $A\in\A$ with Lipschitz boundary and let $R>0$. Then  the map 
$$
u\mapsto  C(u)
$$
is measurable from $(BV^m_{R,A}, \B(BV^m_{R,A}) )$ to 
$(\mathcal M_{R,A}^{m\times n}, \B(\mathcal M_{R,A}^{m\times n}))$.
\end{cor}

\begin{proof}
Since $ C(u)= Du- D^au-D^ju$, the result follows from the continuity of the map $u\mapsto Du$ from
$(BV^m_{R,A},d^m_{R,A})$ to $(\mathcal M_{R,A}^{m\times n},d_{R,A}^{m\times n})$, using
 Lemmas \ref{l:meas-ac-part} and \ref{l:meas-j-part}.
\end{proof}
We are now ready to prove the main result of the section. 


\begin{prop}\label{measurability} Let $ f$ and $ g$ be random volume and surface integrands, respectively, according to Definition \ref{ri}, and let $(\Om,\widehat\T,\widehat P)$ be the completion of the probability space $(\Om,\T, P)$.
 Let $A\in\A$, let
 $w \in SBV(A,\R^m)$, and for every $\om\in \Om$ let  $m_\om ^{f,g}(w,A)$ be defined as in  \eqref{m-phipsi}  and \eqref{inf for fixed omega}. 
 Then the function $\om\mapsto m_\om ^{f,g}(w,A)$ is $\widehat\T$-measurable.
\end{prop}

\begin{proof} For every $\om \in \Omega$, $B\in\B(A)$,  and $u\in BV(A,\R^m)$,  we define 
$$
E(\omega)(u,B):= \int_{B}  f(\om,x,\nabla u)\dx+\int_{S_u\cap B} g(\om,x,[u],\nu)\,d\mathcal H^{n-1}.
$$
Let us fix a  sequence $(A_j)$ of open sets with Lipschitz boundary, with $A_j\subset\subset A_{j+1}$ for every $j\in\mathbb N$  and $\cup_jA_j = A$. It follows easily from the definition that 
\begin{gather*}
 m_\om ^{f,g}(w,A)= \lim_{j\to +\infty}  \inf\{ E(\omega)(u,A) : u\in SBV(A,\R^m),\ u=w \; \textrm{in } A\setminus A_j \}
 \\
 = \lim_{j\to +\infty}  \Big(\inf\{ E(\omega)(u,A_{j+1}) : u\in SBV(A_{j+1},\R^m),\ u=w \; \textrm{in } A_{j+1}\setminus A_j \} + E(\omega)(w,A\setminus A_{j+1})\Big)
 \\
= \lim_{j\to +\infty} \inf\{ E(\omega)(u,A_{j+1}) : u\in SBV(A_{j+1},\R^m),\ u=w \; \textrm{in } A_{j+1}\setminus A_j \},
\end{gather*}
where in the last equality we used the fact that $E(\omega)(w,A\setminus  A_{j+1}) \to 0$ as $j \to+\infty$ since, by $(f4)$ and $(g4)$, we have
$E(\omega)(w,A\setminus  A_{j+1})\le c_3|Dw|(A\setminus  A_{j+1})+ c_4\mathcal L^n(A\setminus  A_{j+1})$.

Let us fix $j \in \mathbb{N}$. It is obvious that
\begin{gather*}
\inf\{  E(\omega)(u,A_{j+1})  : u\in SBV(A_{j+1},\R^m),\ u=w \; \textrm{in } A_{j+1}\setminus A_j \} \\
=\inf\{    E(\omega)(u,A_{j+1})  : u\in SBV(A_{j+1},\R^m), \  E(\omega)(u,A_{j+1}) \le  E(\omega)(w,A_{j+1}), \ u=w \; \textrm{in } A_{j+1}\setminus A_j \}.
\end{gather*}
By $(f4)$ and $(g4)$ we have $E(\omega)(w,A_{j+1})\le c_3|Dw|(A_{j+1})+c_4\mathcal{L}^n(A_{j+1})$. Therefore,
 from Remark~\ref{BV estimate} we obtain that there exists $ R_1>0$, depending on $A_{j+1}$ and $w$, such that 
\begin{gather*}
\inf\{ E(\omega)(u,A_{j+1}) : u\in SBV( A_{j+1},\R^m),\ u=w \; \textrm{in } A_{j+1}\setminus A_j \} \\
=\inf\{ E(\omega)(u,A_{j+1}) : u\in SBV( A_{j+1},\R^m),\ |Du|( A_{j+1})\leq R_1, \ u=w \; 
\textrm{in } A_{j+1}\setminus A_j \}.
\end{gather*}
Thanks to Poincar\'e's inequality, there exists $R\ge R_1$, depending on $A_{j+1}$, $w$, and $R_1$,  such that every function
$u\in BV(A_{j+1},\R^m)$, satisfying $|Du|(A_{j+1})\leq R_1$ and $ u=w$ in $A_{j+1}\setminus A_j$, satisfies also
 $\| u \|_{L^1 (A_{j+1}, \R^m)} \leq R$. This implies that
 \begin{gather*}
\inf\left\{E(\omega)(u,A_{j+1}) : u\in SBV(A_{j+1},\R^m),\ u=w \; \textrm{in } 
A_{j+1}\setminus A_j \right\} 
\\
=\inf\{ E(\omega)(u,A_{j+1}) : u\in SBV(A_{j+1},\R^m), \ \| u \|_{L^1 (A_{j+1}, \R^m)} \leq R,\ |Du|(A_{j+1})\leq  R, \ u=w \; \textrm{in } A_{j+1}\setminus A_j \}.
\end{gather*}

Therefore, to prove the proposition it is enough to show that the function  
\begin{equation}\label{uno:mis}
\om \mapsto \inf\{ E(\omega)(u,A_{j+1}) : u\in SBV(A_{j+1},\R^m)\cap BV^m_{R,A_{j+1}}, \ u=w \; \textrm{in } A_{j+1}\setminus A_j \} 
\end{equation}
is $\widehat\T$-measurable.

We define $H  \colon  BV^m_{R,A_{j+1}} \to [0,+\infty]$ as
$$
H  (u):= 
\begin{cases}
0 \quad &\textrm{if } C(u)=0\hbox{ and } u=  w \; \textrm{in } A_{j+1}\setminus A_j,\\
+\infty &\textrm{otherwise},
\end{cases}
$$
where the equality $C(u)=0$ means that $C( u)(B)=0$ for every $B\in \B(A_{j+1})$. By 
\eqref{uno:mis} to conclude the proof it suffices to show that the function
\begin{equation}\label{min:equality}
\om\mapsto \inf \{ E(\omega)(u,A_{j+1})  +  H   (u): u\in  BV^m_{R,A_{j+1}} \}
\end{equation}
is $\widehat\T$-measurable.

To this aim, we apply the Projection Theorem. Note that the function $(\om, u)\mapsto E(\omega)(u,A_{j+1})$ from $\Omega\times  BV^m_{R,A_{j+1}}$ to $\R$ is 
$\T\otimes \B(BV^m_{R,A_{j+1}})$-measurable, by Lemmas \ref{meas:volume} and \ref{lemma:meas}. Moreover, the function $u\mapsto  H (u)$ from $ BV^m_{R,A_{j+1}} $ to $\R$ is 
$\B(BV^m_{R,A_{j+1}})$-measurable, since the set $\{u\in BV^m_{R,A_{j+1}}: C(u)=0\}$ belongs to $\B( BV^m_{R,A_{j+1}})$  by Corollary \ref{cor:Cantor}, while the set 
$\{u\in  BV^m_{R,A_{j+1}} : u=w  \; \textrm{in } A_{j+1}\setminus A_j\}$ is closed in $ (BV^m_{R,A_{j+1}}, d^m_{R,A_{j+1}})$. Hence, for every $t>0$ we have 
\begin{equation}\label{DM-set:proj}
\{(\om,u) \in \Omega\times BV^m_{R,A_{j+1}} :  E(\omega)(u,A_{j+1})  + H  (u)<t \} \in \T\otimes \B(BV^m_{R,A_{j+1}}).
\end{equation}
Since the  metric  space $ (BV^m_{R,A_{j+1}},d^m_{R,A_{j+1}} )$ is compact  thanks to Lemma~\ref{DM-3312}, by the Projection Theorem (see, e.g., \cite[Theorem III.13 and 33(a)]{Dellacherie}) the 
projection onto $\Omega$ of the set above belongs to $\widehat\T$.
On the other hand,  the projection onto $\Omega$ of the set in \eqref{DM-set:proj}
coincides with the set of points $\om\in\Omega$ such that
\begin{equation*}
\inf \{ E(\omega)(u,A_{j+1})  +  H   (u): u\in  BV^m_{R,A_{j+1}} \} <t.
\end{equation*}
Since this set belongs to $\widehat\T$ for every $t > 0$, the function in 
\eqref{min:equality}
is  $\widehat\T$-measurable.
\end{proof}

\section*{Acknowledgments}
F. Cagnetti was supported by the EPSRC under the Grant EP/P007287/1 ``Symmetry of Minimisers in Calculus of Variations''.
The research of G. Dal Maso was partially funded by the European Research Council under
Grant No. 290888 ``Quasistatic and Dynamic Evolution Problems in Plasticity
and Fracture''. G. Dal Maso is a member of the Gruppo Nazionale per l'Analisi Matematica, la Probabilit\`a e le loro Applicazioni (GNAMPA) of the Istituto Nazionale di Alta Matematica (INdAM).
The work of C. I. Zeppieri was supported by the Deutsche Forschungsgemeinschaft (DFG, German Research Foundation) project 3160408400 and under the Germany Excellence Strategy EXC 2044-390685587, Mathematics M\"unster: Dynamics--Geometry--Structure.


}


\end{document}